\documentclass[12pt, final]{article}
\usepackage{amsmath,amssymb,amsfonts,amsthm,bbm}
\usepackage[authoryear]{natbib}
\usepackage{verbatim}
\usepackage{setspace}
\usepackage{graphicx}
\usepackage{subcaption}
\usepackage{comment}
\usepackage{enumitem}
\usepackage{algorithm}
\usepackage{algpseudocode}
\usepackage[
pagebackref,
colorlinks=true,
pdfpagemode=UseNone,
citecolor=OliveGreen,
linkcolor=BrickRed,
urlcolor=BrickRed,
pdfstartview=FitH,
linktocpage=true]{hyperref}
\usepackage[dvipsnames]{xcolor}
\usepackage[capitalise]{cleveref}

\newif\ifStat
\newif\ifNote
\newif\ifComments
\newif\ifThmitalic

\Stattrue
\Notefalse
\Commentsfalse
\Thmitalictrue

\ifComments
		\newcommand{\Tnote}[1]{{\color{blue} [Tselil: #1]}}
		
		\newcommand{\RMK}[1]{{\color{blue}{[Comment: #1]}}}
		\newcommand{\TODO}[1]{{\color{red}{[#1]}}}
		\newcommand{\rmk}[1]{{\color{blue}{[#1]}}}
\else
		\newcommand{\Tnote}[1]{}
		\newcommand{\RMK}[1]{}
		\newcommand{\TODO}[1]{}
		\newcommand{\rmk}[1]{}
\fi

\ifStat
\renewcommand{\hat}{\widehat}
\renewcommand{\tilde}{\widetilde}
\renewcommand{\bar}{\overline}
\renewcommand{\top}{{\sf T}}
\else

\fi
\newcommand{\iidsim}{\overset{\scriptsize\iid}{\sim}}

\usepackage[top=1in, bottom=1in, left=1in, right=1in]{geometry}

\makeatletter
\newcommand*{\rom}[1]{\expandafter\@slowromancap\romannumeral #1@}
\makeatother

\newcommand{\rank}{\mathrm{rank}}

\newcommand{\supp}{\mathrm{supp}}

\newcommand{\sgn}{\mathrm{sgn}}

\newcommand{\var}{\mathrm{Var}}

\newcommand{\op}{\mathrm{op}}

\newcommand{\iid}{\mathrm{i.i.d.}}
\newcommand{\beq}{\begin{equation}}
\newcommand{\eeq}{\end{equation}}

\DeclareMathOperator*{\argmax}{\mbox{argmax}}

\newcommand{\wto}{\stackrel{w}{\to}}
\newcommand{\pto}{\stackrel{p}{\to}}

\newcommand{\norm}[1]{\left\|{#1}\right\|}

\newcommand{\R}{\mathbb{R}}

\newcommand{\E}{\mathbb{E}}

\newcommand{\veps}{\varepsilon}

\renewcommand{\P}{\mathbb{P}}

\renewcommand{\d}{\textup{d}}

\newcommand{\bS}{\mathrm{\bf S}}

\newcommand{\zz}{\text{\boldmath $z$}}
\newcommand{\ww}{\text{\boldmath $w$}}

\newcommand{\vv}{\text{\boldmath $v$}}

\newcommand{\rr}{\text{\boldmath $r$}}

\newcommand{\bone}{\mathrm{\bf 1}}
\newcommand{\bzero}{\mathrm{\bf 0}}

\newcommand{\bmu}{\text{\boldmath $\mu$}}

\newcommand{\bSigma}{\text{\boldmath $\Sigma$}}

\newcommand{\cC}{\mathcal{C}}

\newcommand{\cV}{\mathcal{V}}

\newcommand{\cU}{\mathcal{U}}

\newcommand{\cF}{\mathcal{F}}
\newcommand{\cT}{\mathcal{T}}
\newcommand{\cW}{\mathcal{W}}

\def\cuP{\mathscr{P}}

\def\normal{{\sf N}}
\def\id{{\boldsymbol I}}

\newcommand{\hw}{\widehat \ww}

\ifNote

\theoremstyle{plain}
\newtheorem{thm}{Theorem}
\newtheorem{claim}{Claim}
\newtheorem{defn}{Definition}
\newtheorem{lem}[thm]{Lemma}
\newtheorem*{lemma*}{Lemma}
\newtheorem{cor}[thm]{Corollary}
\newtheorem{prop}[thm]{Proposition}
\newtheorem{ass}{Assumption}
\newtheorem{fact}{Fact}
\theoremstyle{definition}
\newtheorem{exm}{Example}
\newtheorem{rem}{Remark}
\newtheorem{conj}{Conjecture}

\else\ifThmitalic

\newtheorem{thm}{Theorem}[section]

\newtheorem{defn}{Definition}[section]
\newtheorem{lem}[thm]{Lemma}
\newtheorem*{lemma*}{Lemma}
\newtheorem{cor}[thm]{Corollary}
\newtheorem{prop}[thm]{Proposition}
\newtheorem{ass}{Assumption}

\theoremstyle{definition}

\newtheorem{rem}{Remark}[section]

\else

\fi
\fi

\usepackage{mathrsfs}
\usepackage{amsmath}
\usepackage{hyperref}
\usepackage{cleveref}
\usepackage{todonotes}
\usepackage{booktabs}
\usepackage{placeins}
\allowdisplaybreaks

\title{Efficient and Adaptive Estimation of Portfolio Weights with Spectral Risk Measures}

\author{Kangjie Zhou\thanks{Center for Data Science for Enterprise and Society, Cornell University} \,\, and \,\, Ming Yuan\thanks{Department of Statistics, Columbia University}
}

\date{(\today)}

\begin{document}

	\maketitle
	
\begin{abstract}
Spectral risk measures, including conditional value-at-risk (CVaR), generate a family of convex criteria for portfolio estimation. We study how the criterion should be chosen when different criteria share the same population minimizer, and whether the efficient criterion can itself be learned from data. We develop a general asymptotic theory for empirical spectral-risk minimization under estimated linear constraints. Under normal scale-mixture elliptical returns, all spectral risk measures identify the same population efficient portfolio, but their empirical minimizers have different sampling distributions. Their asymptotic covariance decomposes into a common component and a positive-semidefinite component scaled by a functional of the spectral measure, reducing efficiency to an optimization over probability measures. We characterize the efficiency-optimal spectral measure and show that single-level CVaR is generally inefficient. We then construct a fully data-adaptive estimator that learns the radial distribution and optimal spectral measure from the same observations used for portfolio estimation, yet has the same first-order distribution as the infeasible oracle. Simulations and an empirical application illustrate the method.
\end{abstract}	

\section{Introduction}
\label{sec:intro}

A basic principle of statistical estimation is that the choice of estimating criterion can matter even when the population target does not. Two loss functions may have the same population minimizer but produce empirical minimizers with different sampling distributions. When the criterion belongs to a rich family, this raises two natural questions: can the criterion itself be chosen to minimize the sampling variability of the resulting estimator, and, if the efficient criterion depends on the unknown data-generating distribution, can it be learned from the same data without sacrificing first-order efficiency? This paper studies these questions for an infinite-dimensional family of convex criteria generated by spectral risk measures. Portfolio optimization provides a particularly useful setting because the same risk measures that encode attitudes toward losses also define empirical procedures for estimating portfolio weights.

Conditional value-at-risk (CVaR), also known as expected shortfall, is a standard measure of downside risk in financial regulation, risk management, and portfolio optimization. Spectral risk measures extend CVaR by aggregating losses across confidence levels and form a broad class of law-invariant coherent risk measures \citep{artzner1999coherent,acerbi2002spectral,adam2008spectral,brandtner2013conditional}. Their convexity and variational representations make the corresponding portfolio problems computationally tractable \citep{rockafellar2000optimization,rockafellar2002conditional,luthi2005convex,rockafellar2006optimality,ruszczynski2006optimization,mansini2007conditional}. Yet empirical portfolio rules based on these criteria face a fundamental statistical difficulty: the return distribution is unknown. The risk objective is evaluated from a finite sample, constraints such as target expected return contain estimated quantities, and optimization can amplify relatively small sampling errors into large changes in portfolio weights.

For mean-variance portfolios, this estimation risk has been studied extensively; see, for example, \cite{jobson1980estimation,britten1999sampling,okhrin2006distributional,kan2008distribution}. For CVaR and general spectral-risk portfolios, much of the existing literature instead concerns computation or estimation of the risk functional itself. Our object of inference is the \emph{optimizer} of the empirical risk criterion, including the effect of estimating the constraints from the same observations. This distinction leads to three statistical contributions.

First, we develop a general asymptotic theory for empirical spectral-risk minimization under estimated linear equality and inequality constraints. The result does not require elliptical returns. It accommodates mixtures of CVaRs and the nonsmoothness of their empirical objectives, while jointly accounting for sampling error in the objective and in stochastic constraints. Under regular binding equalities, the empirical optimizer is asymptotically Gaussian. With weakly binding inequalities, the limiting law is characterized by a stochastic quadratic program. Thus, empirical spectral-risk portfolio selection can be treated as a constrained $\rm M$-estimation problem with an estimated feasible set.

Second, we use this theory to study efficiency when several criteria share a common estimand. Elliptically contoured return models are particularly revealing for this purpose. They preserve the location--scale geometry of the Gaussian model while allowing heavy-tailed radial variation, and have long been used to extend portfolio and asset-pricing arguments beyond normality \citep{owen1983class}. Recent work also emphasizes their role in portfolio choice under parameter uncertainty \citep{kan2025optimal}. For explicit covariance calculations and feasible implementation, we work with the normal scale-mixture subclass, which includes the Gaussian, multivariate $t$, and other widely used heavy-tailed models.

This setting exhibits what we call \emph{population equivalence but statistical non-equivalence}. Under ellipticity and a fixed target expected return, spectral-risk minimization subject to the budget constraint identifies the same population efficient portfolio as variance. The spectral measure therefore does not change the estimand. It changes the empirical criterion used to estimate that estimand. We show that the asymptotic covariance of the resulting weights has the structural form
\[
    \Sigma_{\mathrm{SRM}}(m)=\Sigma_{\mathrm{common}}+I(m;c)\,\Sigma_{\perp},
    \qquad \Sigma_{\perp}\succeq0,
\]
where neither $\Sigma_{\mathrm{common}}$ nor $\Sigma_{\perp}$ depends on the spectral measure $m$, and $c$ is a scalar index determined by the population return target and the mean--covariance geometry of the portfolio problem. All dependence on the estimating criterion is therefore summarized by the scalar functional $I(m;c)$. Consequently, minimizing the covariance matrix in Loewner order reduces to minimizing a scalar functional over probability measures. This reduction turns the choice of spectral risk measure into an infinite-dimensional efficiency problem. Throughout the paper, \emph{efficiency} refers to asymptotic covariance comparisons within the class of empirical spectral-risk estimators (and, where stated, comparison with the usual sample mean-variance estimator); we do not claim semiparametric efficiency over all regular estimators of the portfolio weights.

We characterize the unique optimizer within each admissible compactly supported spectral class. The optimal measure is absolutely continuous in the interior of its support, implying that no single interior CVaR level is efficient within this class. We further show that, as the support restriction is relaxed, the infimum of the spectral-risk covariance is no larger than the covariance of sample mean-variance optimization. In the Gaussian benchmark the optimized spectral rule approaches the mean-variance efficiency bound. Under heavy-tailed scale mixtures the inequality can be strict, as illustrated by our Laplace example and simulations. Thus, even when the mean-variance portfolio is the correct population target, its usual sample analogue need not be the most efficient way to estimate it.

Third, we make the choice of estimating criterion itself data-adaptive. The spectral measure plays a role analogous to an infinite-dimensional tuning parameter indexing a family of empirical criteria, with an important distinction: under ellipticity, changing this tuning parameter does not change the population estimand. It can therefore be selected entirely for statistical efficiency. The oracle choice depends on the unknown radial mixing distribution, which we estimate nonparametrically from estimated Mahalanobis radii. We then use the estimated distribution to learn the optimal spectral measure and minimize the resulting empirical criterion. Thus the same observations are used in the three-stage pipeline:
\[
    \widehat P_{\lambda}
    \;\longrightarrow\;
    \widehat m
    \;\longrightarrow\;
    \widehat{\ww}_{\widehat m},
\]
where $\widehat{P}_{\lambda}$, $\widehat{m}$, and $\widehat{\ww}_{\widehat m}$ are the estimations of the radial mixing distribution, optimal spectral measure, and oracle criterion, respectively. Despite this additional layer of adaptation, the final estimator has the same first-order limiting distribution as the infeasible oracle rule that knows the optimal spectral measure in advance. The main efficiency and adaptivity conclusions can therefore be summarized schematically as
\[
    \Sigma_{\mathrm{SRM}}(m) = \Sigma_{\mathrm{common}}+I(m;c)\Sigma_{\perp},
    \qquad
    m_{*}^c = \mbox{minimizer of } I(m;c),
\]
and
\[
    \hat{\ww}_{\hat m} \mbox{ is consistent and asymptotically normal, with asymptotic variance } \Sigma_{\mathrm{SRM}} (m_{*}^c).
\]
No sample splitting is required. In this sense, learning both the infinite-dimensional nuisance distribution and the efficiency-optimal estimating criterion is first-order costless for estimation of the portfolio weights. Importantly, the selected spectral measure is not interpreted as an estimate of an investor's latent risk preferences: it is a statistically selected loss criterion within a family that shares the same population minimizer. We also give a concrete constrained nonparametric maximum-likelihood estimator of the radial distribution and establish the Wasserstein consistency needed for this oracle-adaptivity result.

Our results complement several related literatures. Statistical work on CVaR includes nonparametric estimation \citep{scaillet2004nonparametric,chen2008nonparametric,kato2012weighted}, regression and forecast evaluation \citep{dimitriadis2019joint,patton2019dynamic}, robustness under heavy tails \citep{he2023robust}, and high-dimensional sparse procedures \citep{zhang2025high,wu2025linear}. These contributions primarily concern estimation of risk functionals or conditional expected shortfall. Our focus is instead the sampling distribution and efficiency of the optimizer of a constrained empirical spectral-risk problem. We also complement the sampling-error literature for mean-variance portfolios by treating the empirical loss criterion itself as an object that can be selected for statistical efficiency.

The theory leads directly to an implementable procedure. Simulations compare sample mean-variance optimization, fixed-level CVaR, the infeasible oracle spectral rule, and its feasible adaptive counterpart. An empirical study using U.S. industry portfolios evaluates weight stability, turnover, downside risk, and out-of-sample performance. The empirical analysis is intended primarily to assess whether the theoretically motivated criterion produces stable and competitive portfolio estimates outside the exact model, rather than to assert universal dominance in realized investment performance.

Technically, our analysis combines the variational representation of CVaR \citep{rockafellar2000optimization}, the spectral/mixture representation of the law-invariant coherent, comonotonic-additive class \citep{acerbi2002spectral,kusuoka2001law}, empirical-process arguments \citep{wellner2013weak}, sample-average approximation under stochastic constraints \citep[Section 5]{shapiro2021lectures}, and epi-convergence in distribution \citep{knight1999epi}. The general theory is stated without elliptical structure; that structure enters only in the efficiency analysis and construction of the adaptive rule.

The remainder of the paper is organized as follows. \Cref{sec:setup} introduces empirical spectral-risk portfolio estimators. \Cref{sec:low-dim} develops the general asymptotic theory under estimated linear constraints and gives the canonical target-return portfolio as a corollary. \Cref{sec:elliptic} specializes the theory to normal scale-mixture elliptical returns and derives the covariance decomposition underlying our efficiency analysis. \Cref{sec:optimal_srm} characterizes the efficiency-optimal spectral measure and compares it with fixed-level CVaR and MVO. \Cref{sec:feasible_rule} develops the data-adaptive estimator and establishes oracle first-order equivalence. \Cref{sec:num_exp} presents simulation and empirical evidence, and \Cref{sec:conclusion} concludes. Proofs and supplemental technical arguments appear in the appendices. Although the main text focuses on the canonical target-return problem without a risk-free asset, parallel results with unrestricted borrowing and lending at the risk-free rate can also be derived and we present them in Appendix~\ref{app:risk_free}.

\section{Spectral Risk Measures and Portfolio Estimation}
\label{sec:setup}

This section introduces the class of portfolio estimators studied in the paper and highlights the key statistical issue. Spectral risk measures are often motivated as alternatives to variance because they focus on downside risk and satisfy economically appealing coherence properties. In our setting, they play an additional role: they define a family of empirical estimating criteria for portfolio weights. Under elliptical returns and a fixed target expected return, spectral risk measures share the same population efficient portfolio as mean-variance optimization under mild conditions, while their empirical minimizers can have different sampling distributions. Thus, a spectral risk measure can play two distinct roles: as an economic description of risk and, in the present paper, as an empirical estimating criterion whose choice affects the statistical efficiency with which portfolio weights are estimated.

Let $\ww = (w_1, \ldots, w_N)^\top \in \R^N$ denote portfolio weights across $N$ risky assets, and let $\rr = (r_1, \ldots, r_N)^\top \in \R^N$ denote their excess returns. The portfolio return is
\[
r_{\ww} = \ww^\top \rr,
\]
and we take $-r_{\ww}$ as the associated loss. Throughout the paper, $\bmu=\E[\rr]$ denotes the vector of expected returns.

\subsection{Spectral Risk Measures}

For a loss random variable $X$ and confidence level $\alpha \in (0,1)$, the value-at-risk is the $\alpha$-quantile
\[
\operatorname{VaR}_{\alpha}(X)
=
\sup \left\{ y \in \R : \P(X \le y) \le \alpha \right\}.
\]
When the distribution of $X$ is continuous, the conditional value-at-risk, also known as expected shortfall, can be written as
\begin{equation}
	\label{eq:cvar_by_var}
	\operatorname{CVaR}_{\alpha}(X)
	= \E \left[ X \vert X \ge \operatorname{VaR}_{\alpha}(X) \right]
    =
	\frac{1}{1-\alpha}
	\int_{\alpha}^{1} \operatorname{VaR}_{p}(X) \, \d p .
\end{equation}

We focus on law-invariant coherent risk measures that also satisfy comonotonic additivity. Law invariance means that $\rho(X)$ depends only on the distribution of $X$. Coherence requires monotonicity, subadditivity, positive homogeneity, and translation equivariance, as in \cite{artzner1999coherent}. For the law-invariant, coherent, comonotonic-additive class considered here, the spectral representation can be written as a single mixture of CVaRs: there exists a probability measure $m$ on $[0,1]$ such that
\[
\rho(X)
=
\int_0^1 \operatorname{CVaR}_{\alpha}(X) \, \d m(\alpha).
\]
We refer to such $\rho$ as a spectral risk measure. When it is useful to emphasize its dependence on the spectral measure, we write
\[
\rho_m(X)
=
\int_0^1 \operatorname{CVaR}_{\alpha}(X) \, \d m(\alpha).
\]

Equivalently, using \cref{eq:cvar_by_var} and Fubini's theorem,
\[
\rho_m(X)
=
\int_0^1 \varphi(p) \operatorname{VaR}_{p}(X) \, \d p,
\]
where the spectral density is
\[
\varphi(p)
=
\int_0^p \frac{1}{1-\alpha} \, \d m(\alpha).
\]
Conversely, any nondecreasing density $\varphi$ on $[0,1]$ satisfying
\[
\int_0^1 \varphi(p)\,\d p=1
\]
induces a spectral risk measure. Thus, the probability measure $m$ or, equivalently, the spectral density $\varphi$, determines how losses at different quantiles enter the risk criterion.

\subsection{Population Portfolios and Elliptical Returns}

For a fixed target expected return $\mu_0$, the population spectral-risk portfolio is defined by
\begin{equation}
	\label{eq:general_risk_optimization}
	\operatorname{minimize}_{\ww \in \R^N}
	\quad
	\rho_m(-r_{\ww}),
	\qquad
	\mbox{subject to}
	\quad
	\ww^\top \bone = 1,
	\quad
	\ww^\top \bmu = \mu_0 .
\end{equation}
We denote its solution by $\ww_*(\mu_0)$. Varying $\mu_0$ traces out the efficient frontier associated with the risk criterion.

A central case in this paper is the class of elliptically contoured return distributions. We work with the normal scale-mixture subclass, represented as
\[
\rr
=
\bmu + \lambda \bSigma^{1/2} \zz,
\]
where $\zz \sim \normal(\bzero,\id_N)$, $\lambda \ge 0$ is independent of $\zz$, and $\E[\lambda^2]<\infty$. Here $\bSigma$ is a scatter matrix; under the normalization $\E[\lambda^2]=1$, it coincides with the covariance matrix. Normal scale mixtures include the Gaussian and multivariate $t$ distributions and preserve the defining location--scale structure of linear portfolio returns. This structure yields the population equivalence below, while the conditional-Gaussian representation permits the explicit covariance calculations and radial estimation developed later. See \cite{cambanis1981theory,fang1990symmetric,anderson2003introduction} for treatments of elliptical distributions and their stochastic representations.

For any portfolio $\ww$ satisfying $\ww^\top\bmu=\mu_0$, law invariance, translation equivariance, positive homogeneity, and the symmetry of $Z \sim \normal(0,1)$ imply
\begin{equation}
	\label{eq:elliptical_population_equivalence}
	\rho_m(-r_{\ww})
	=
	-\mu_0
	+
	\norm{\bSigma^{1/2}\ww}_2 \, \rho_m(\lambda Z),
\end{equation}
where $Z \sim \normal(0,1)$ is independent of $\lambda$. Hence, provided $\rho_m(\lambda Z)>0$, minimizing $\rho_m(-r_{\ww})$ subject to the target-return and budget constraints is equivalent to minimizing $\ww^\top\bSigma\ww$ subject to the same constraints. As a consequence, $\ww_* (\mu_0)$ only depends on $\mu_0$, and is independent of the probability measure $m$.

Equation~\eqref{eq:elliptical_population_equivalence} separates the population portfolio problem from the statistical estimation problem. For a fixed target return, changing the spectral measure does not change the population optimizer; it changes only the empirical criterion used to estimate that optimizer. Thus, under ellipticity, the spectral measure indexes a family of estimators of a common portfolio target. This population equivalence but statistical non-equivalence is the basis for the efficiency analysis developed later in the paper.

Outside the elliptical family, the equivalence in \cref{eq:elliptical_population_equivalence} generally fails. Spectral-risk minimization can then change both the population objective and the statistical properties of the estimator. The general asymptotic theory in \cref{sec:low-dim} does not require elliptical structure; ellipticity enters only when we obtain explicit covariance comparisons and formulate the criterion-selection problem in \cref{sec:elliptic,sec:optimal_srm}.

\subsection{Empirical Spectral-Risk Portfolios}

In practice, the return distribution is unknown and must be estimated from observed data. Given historical observations $\rr_1,\ldots,\rr_T \iidsim P_{\rr}$, let
\[
\hat{\bmu}
=
\frac{1}{T}\sum_{i=1}^T \rr_i
\]
denote the empirical mean. For a fixed portfolio $\ww$, let $-\hat r_{\ww,T}$ denote a random variable distributed according to the empirical distribution of portfolio losses:
\[
\operatorname{Law} \left( -\hat r_{\ww,T} \right)
=
\frac{1}{T}\sum_{i=1}^T
\delta_{-\ww^\top \rr_i}.
\]
The empirical spectral-risk portfolio rule is
\begin{equation}
	\label{eq:empirical_rho_risk_optimization}
	\operatorname{minimize}_{\ww \in \R^N}
	\quad
	\rho_m \left( -\hat r_{\ww,T} \right),
	\qquad
	\mbox{subject to}
	\quad
	\ww^\top \bone=1,
	\quad
	\ww^\top \hat{\bmu}=\mu_0 .
\end{equation}
We denote its solution by $\hw(\mu_0)$.

When $m$ is a point mass at a single level $\alpha$, the objective in \cref{eq:empirical_rho_risk_optimization} reduces to empirical CVaR. More generally, $m$ determines how empirical CVaRs at different confidence levels are combined and hence how different parts of the observed loss distribution contribute to the estimating criterion.

Under elliptical returns, the population solution of \cref{eq:general_risk_optimization} is invariant to $m$, whereas the empirical solution of \cref{eq:empirical_rho_risk_optimization} is not. The spectral measure can therefore be chosen to reduce the sampling variability of $\hw(\mu_0)$ without changing the population efficient portfolio being estimated. Later sections derive the sampling distribution of $\hw(\mu_0)$, reduce its covariance dependence on $m$ to a scalar functional, and characterize the efficiency-optimal spectral measure.

\subsection{Spectral Measures as Estimating Criteria}
\label{sec:spectral_estimating_criteria}

The distinction between the population objective and the empirical estimating criterion is worth making explicit. Let
\[
\mathcal W(\mu_0)
=
\left\{
\ww\in\R^N:
\ww^\top\bone=1,\;
\ww^\top\bmu=\mu_0
\right\}
\]
denote the population constraint set. From previous discussions, we know that under the elliptical target-return setting,
\[
\arg\min_{\ww\in\mathcal W(\mu_0)}
\rho_m(-r_{\ww})
=
\ww_*(\mu_0)
\qquad
\text{for every admissible }m,
\]
whereas the empirical minimizer in \cref{eq:empirical_rho_risk_optimization} has a sampling distribution that depends on $m$. Thus, $m$ indexes a family of estimators of a common parameter.

This perspective is analogous to choosing among estimating criteria or tuning parameters for statistical efficiency, with an important distinction: under ellipticity, changing $m$ does not change the population estimand. The spectral measure can therefore be selected for estimation efficiency. This also clarifies the interpretation of the data-adaptive procedure developed in \cref{sec:feasible_rule}. The selected spectral measure is not intended to estimate an investor's latent risk preferences; it is a statistically selected loss criterion used to estimate a portfolio whose population target is invariant to $m$.

\subsection{Variational Representation and Computation}
\label{sec:prelim}

The empirical problem \cref{eq:empirical_rho_risk_optimization} is convex and can be solved using the variational representation of CVaR due to \cite{rockafellar2000optimization}.

\begin{thm}[\cite{rockafellar2000optimization}]
	\label{thm:var_rep}
	Let $X$ be a random variable with continuous and strictly increasing distribution function. Then, for any $\alpha \in (0,1)$,
	\begin{align*}
		\operatorname{CVaR}_{\alpha}(X)
		&=
		\min_{z \in \R}
		\left\{
		z
		+
		\frac{1}{1-\alpha}
		\E\left[(X-z)_+\right]
		\right\}, \\
		\operatorname{VaR}_{\alpha}(X)
		&=
		\arg \min_{z \in \R}
		\left\{
		z
		+
		\frac{1}{1-\alpha}
		\E\left[(X-z)_+\right]
		\right\}.
	\end{align*}
	Moreover, $\operatorname{CVaR}_{\alpha}(\cdot)$ is convex.
\end{thm}

Applying \cref{thm:var_rep} to each CVaR component gives
\begin{align*}
	\rho_m(X)
	&=
	\int_0^1
	\operatorname{CVaR}_{\alpha}(X)
	\, \d m(\alpha) \\
	&=
	\min_{z \in M[0,1]}
	\left\{
	\int_0^1 z(\alpha)\,\d m(\alpha)
	+
	\int_0^1
	\frac{1}{1-\alpha}
	\E\left[(X-z(\alpha))_+\right]
	\d m(\alpha)
	\right\},
\end{align*}
where $M[0,1]$ denotes the set of nondecreasing measurable functions on $[0,1]$. Consequently,
\begin{align*}
	\rho_m(-\hat r_{\ww,T})
	=
	\min_{z \in M[0,1]}
	\bigg\{
	&\int_0^1 z(\alpha)\,\d m(\alpha) \\
	&+
	\frac{1}{T}\sum_{i=1}^T
	\int_0^1
	\frac{1}{1-\alpha}
	\left(-\ww^\top\rr_i-z(\alpha)\right)_+
	\d m(\alpha)
	\bigg\}.
\end{align*}
Thus, empirical spectral-risk minimization is a joint convex optimization problem over the portfolio weights $\ww$ and the auxiliary quantile function $z(\cdot)$. When $m$ has finite support, this is a finite-dimensional convex program. When $m$ has a continuous component, it can be implemented by discretizing its support, as in the data-adaptive procedure developed in \cref{sec:feasible_rule}.

The population counterpart is
\begin{align*}
	\rho_m(-r_{\ww})
	=
	\min_{z \in M[0,1]}
	\bigg\{
	&\int_0^1 z(\alpha)\,\d m(\alpha) \\
	&+
	\int_0^1
	\frac{1}{1-\alpha}
	\E\left[
	\left(-\ww^\top\rr-z(\alpha)\right)_+
	\right]
	\d m(\alpha)
	\bigg\}.
\end{align*}

The variational representation therefore converts empirical spectral-risk portfolio selection into a convex empirical minimization problem indexed by the spectral measure $m$. The next section develops the asymptotic theory for its optimizer under estimated constraints; subsequent sections exploit the dependence on $m$ to study efficiency and data-adaptive criterion selection.

\section{General Asymptotic Theory}
\label{sec:low-dim}

This section develops the large-sample theory for empirical spectral-risk portfolio rules. The central object is the sampling error of the empirical optimizer $\hw$, obtained by replacing the unknown return distribution with the empirical distribution of historical returns. We work in the classical fixed-dimensional asymptotic regime in which the number of assets $N$ is fixed and the sample size $T$ tends to infinity.

The results in this section apply beyond the elliptical model. Under regularity conditions, empirical spectral-risk portfolios are consistent and admit root-$T$ limiting distributions determined by the curvature of the population objective and the sampling variability of the empirical risk scores and estimated constraints. Under regular equality constraints the limit is Gaussian, whereas weakly binding inequalities can produce non-Gaussian limits characterized by stochastic quadratic programs. The explicit comparison between spectral-risk estimators and sample mean-variance optimization under elliptically contoured returns is developed in \Cref{sec:elliptic,sec:optimal_srm}.

The main result accommodates three features that arise simultaneously in spectral-risk portfolio estimation: a nonsmooth empirical CVaR representation, a continuum mixture of such criteria, and a feasible set whose coefficients may be estimated from the same sample. It also allows weakly binding estimated inequalities, which can generate non-Gaussian stochastic quadratic-program limits. The later efficiency analysis uses simpler equality-constrained special cases, but the general result makes clear that the underlying inference framework is not tied to ellipticity or to the particular target-return problem studied subsequently.

Throughout this section, for an integrable function $f$, write
\[
\P f = \E[f(\rr)],
\qquad
\P_T f
=
\hat{\E}_T[f(\rr)]
=
\frac{1}{T}\sum_{i=1}^T f(\rr_i).
\]
Let
\[
\mathbb{G}_T
=
\sqrt{T}(\P_T-\P)
\]
denote the associated empirical process, and let $\mathbb{G}$ denote the limiting $\P$-Brownian bridge. Thus, for square-integrable $f$,
\[
\mathbb{G}_T f
=
\frac{1}{\sqrt{T}}\sum_{i=1}^T
\left(f(\rr_i)-\E[f(\rr)]\right),
\]
and for square-integrable $f$ and $g$,
\[
\operatorname{Cov}(\mathbb{G}f,\mathbb{G}g)
=
\P(fg)-\P f\,\P g .
\]

\subsection{Scores and Curvature of Empirical Spectral Risk}

We first derive the local derivatives of the population CVaR objective. These derivatives provide the score and curvature terms that enter the first-order expansion of empirical spectral-risk portfolio estimators.

For a fixed observation $\rr\in\R^N$, define
\[
L_{\alpha}(\ww,z;\rr)
=
z+\frac{1}{1-\alpha}\left(-\ww^\top\rr-z\right)_+,
\]
the variational loss associated with $\alpha$-CVaR. For $(\ww,z)\in\R^N\times\R$, let
\[
R_{\alpha}(\ww,z)
=
\E\left[L_{\alpha}(\ww,z;\rr)\right],
\]
and define
\[
z(\ww,\alpha)
=
\arg\min_{z\in\R}R_{\alpha}(\ww,z).
\]
By \Cref{thm:var_rep},
\[
z(\ww,\alpha)
=
\operatorname{VaR}_{\alpha}(-r_{\ww})
\]
whenever the cumulative distribution function of $-r_{\ww}$ is continuous and strictly increasing.

The following condition ensures sufficient local smoothness of the CVaR objective for first-order asymptotic analysis. It is stated for general return distributions and does not impose ellipticity.

\begin{ass}
	\label{ass:regularity_density}
	For the random variable $r_{\ww}=\ww^\top\rr$, assume that, for every $\ww\neq\bzero$:
	\begin{itemize}
		\item[(i)]
		The mapping
		$(\ww,z)\mapsto\nabla_{(\ww,z)}\P(-r_{\ww}\ge z)$
		is continuous, and
		$(\ww,z)\mapsto p_{-r_{\ww}}(z)$,
		the density of $-r_{\ww}$ evaluated at $z$, is positive and continuous.
		Further, $\E[\rr\rr^\top]<\infty$.
		
		\item[(ii)]
		The functions
		\[
		\E\left[\rr\,\bone\{-r_{\ww}\ge z\}\right],
		\qquad
		\E\left[\rr\,\big\vert\, -r_{\ww}=z\right],
		\qquad
		\E\left[\rr\rr^\top\,\big\vert\, -r_{\ww}=z\right]
		\]
		are continuous in $(\ww,z)$.
	\end{itemize}
\end{ass}

\begin{rem}
	\Cref{ass:regularity_density} imposes smoothness on the one-dimensional portfolio returns $r_{\ww}=\ww^\top\rr$. It allows the multivariate return distribution to be non-Gaussian and requires no higher-order moments beyond those stated above. It is satisfied by many commonly used smooth return distributions, including the elliptically contoured models studied in \Cref{sec:elliptic}. The restriction $\ww\neq\bzero$ is necessary because $r_{\ww}$ is degenerate at $\ww=\bzero$ and therefore has no density.
\end{rem}

\begin{thm}[Derivatives of the CVaR variational objective]
	\label{thm:population_derivatives}
	Under \Cref{ass:regularity_density}, the population objective
	$R_{\alpha}(\ww,z)$ is twice continuously differentiable in $(\ww,z)$
	at every $\ww\neq\bzero$. Its gradient is
	\begin{equation}
		\label{eq:pop_gradient}
		\begin{split}
			\nabla_{\ww}R_{\alpha}(\ww,z)
			&=
			-\frac{1}{1-\alpha}
			\E\left[\rr\,\bone\{-r_{\ww}\ge z\}\right] \\
			&=
			-\frac{1}{1-\alpha}
			\P(-r_{\ww}\ge z)
			\E\left[\rr\,\big\vert\, -r_{\ww}\ge z\right],
			\\
			\partial_zR_{\alpha}(\ww,z)
			&=
			1-\frac{1}{1-\alpha}\P(-r_{\ww}\ge z),
		\end{split}
	\end{equation}
	and its Hessian is
	\begin{equation}
		\label{eq:pop_hessian}
		\begin{split}
			\nabla_{\ww}^2R_{\alpha}(\ww,z)
			&=
			\frac{1}{1-\alpha}
			p_{-r_{\ww}}(z)
			\E\left[
			\rr\rr^\top
			\,\big\vert\,
			-r_{\ww}=z
			\right],
			\\
			\nabla_{\ww}\partial_zR_{\alpha}(\ww,z)
			&=
			\frac{1}{1-\alpha}
			p_{-r_{\ww}}(z)
			\E\left[
			\rr
			\,\big\vert\,
			-r_{\ww}=z
			\right],
			\\
			\partial_z^2R_{\alpha}(\ww,z)
			&=
			\frac{1}{1-\alpha}p_{-r_{\ww}}(z).
		\end{split}
	\end{equation}
\end{thm}

The preceding result also gives the derivative of CVaR as a function of the portfolio weights alone.

\begin{thm}[Derivatives of portfolio CVaR]
	\label{thm:direct_differentiate_cvar}
	Under \Cref{ass:regularity_density}, for any $\alpha\in(0,1)$ and
	$\ww\neq\bzero$,
	\begin{align*}
		\nabla_{\ww}\operatorname{CVaR}_{\alpha}(-r_{\ww})
		&=
		-\frac{1}{1-\alpha}
		\E\left[
		\rr\,\bone\{-r_{\ww}\ge z(\ww,\alpha)\}
		\right],
		\\
		\nabla_{\ww}^2\operatorname{CVaR}_{\alpha}(-r_{\ww})
		&=
		\frac{1}{1-\alpha}
		p_{-r_{\ww}}\bigl(z(\ww,\alpha)\bigr)
		\operatorname{Cov}\left(
		\rr\,\big\vert\,
		-r_{\ww}=z(\ww,\alpha)
		\right),
	\end{align*}
	where
	$z(\ww,\alpha)=\operatorname{VaR}_{\alpha}(-r_{\ww})$.
\end{thm}

For the spectral risk measure
\[
\rho_m(X)
=
\int_0^1
\operatorname{CVaR}_{\alpha}(X)\,
\d m(\alpha),
\]
define
\[
U_m(\ww;\rr)
=
\int_0^1
L_{\alpha}\left(
\ww,z(\ww,\alpha);\rr
\right)
\d m(\alpha).
\]
Then
\[
\E[U_m(\ww;\rr)]
=
\rho_m(-r_{\ww}).
\]
The observation-level score
$\nabla_{\ww}U_m(\ww;\rr)$ is the basic random quantity whose empirical average drives the first-order sampling error of the spectral-risk portfolio estimator.

\subsection{Empirical Spectral-Risk Minimization under Estimated Linear Constraints}

We now formulate the general constrained problem. It allows both deterministic and data-dependent linear equality and inequality constraints and contains the target-return portfolio studied later as a special case.

Let $k_c^{(e)}$, $k_c^{(i)}$, $k_s^{(e)}$, and $k_s^{(i)}$ denote the numbers of equality (indexed by $(e)$) and inequality (indexed by $(i)$) constraints that are either deterministic (indexed by $c$), or stochastic/data-dependent (indexed by $s$). Let
\[
F_c^{(e)}\in\R^{k_c^{(e)}\times N},
\qquad
F_c^{(i)}\in\R^{k_c^{(i)}\times N}
\]
be deterministic constraint matrices, and let
\[
F_s^{(e)}:\R^N\to\R^{k_s^{(e)}\times N},
\qquad
F_s^{(i)}:\R^N\to\R^{k_s^{(i)}\times N}
\]
be square-integrable matrix-valued functions. The population problem is
\begin{equation}
	\label{eq:pop_rho_risk_general_linear}
	\begin{split}
		\mbox{minimize}_{\ww\in\R^N}
		\quad
		&\rho_m(-r_{\ww}),
		\\
		\mbox{subject to}
		\quad
		&F_c^{(e)}\ww=b_c^{(e)},
		\qquad
		F_c^{(i)}\ww\le b_c^{(i)},
		\\
		&\E[F_s^{(e)}(\rr)]\ww=b_s^{(e)},
		\qquad
		\E[F_s^{(i)}(\rr)]\ww\le b_s^{(i)}.
	\end{split}
\end{equation}
Let $\cC$ denote its feasible set and let $\ww_*$ denote a solution. Its empirical counterpart is
\[
\begin{split}
	\mbox{minimize}_{\ww\in\R^N}
	\quad
	&\rho_m(-\hat r_{\ww,T}),
	\\
	\mbox{subject to}
	\quad
	&F_c^{(e)}\ww=b_c^{(e)},
	\qquad
	F_c^{(i)}\ww\le b_c^{(i)},
	\\
	&\hat{\E}_T[F_s^{(e)}(\rr)]\ww=b_s^{(e)},
	\qquad
	\hat{\E}_T[F_s^{(i)}(\rr)]\ww\le b_s^{(i)}.
\end{split}
\]
Let $\hat{\cC}_T$ denote its feasible set and let $\hw$ be any measurable solution.

The formulation includes the target-return portfolio by taking the budget constraint as deterministic and the expected-return constraint as stochastic. It also accommodates no-short-sale restrictions, position bounds, and estimated factor or characteristic exposures. Gross-exposure constraints are included as well, since $\|\ww\|_1\le c$ can be represented, for fixed $N$, by the finite collection of linear inequalities $\veps^\top\ww\le c$ for $\veps\in\{\pm1\}^N$.

\begin{ass}
	\label{ass:measure_support}
	There exists $u\in(0,1/2)$ such that
	\[
	\supp(m)\subset[u,1-u].
	\]
\end{ass}

The endpoint restriction serves both a technical and a statistical regularization role. It excludes extreme quantile levels at which empirical CVaR is based on very few observations and its local behavior becomes unstable. The asymptotic theory below keeps $u$ fixed, while the comparison with the full spectral-risk class in \Cref{sec:optimal_srm} studies what happens as the restriction is relaxed by taking $u\downarrow0$. Allowing $u=u_T\downarrow0$ jointly with $T\to\infty$ is beyond the scope of the present analysis.

\begin{ass}[Identification and constraint regularity]
	\label{ass:identification_constraint}
	The population problem \eqref{eq:pop_rho_risk_general_linear} has a unique solution
	$\ww_*\neq\bzero$. For some $\delta>0$, the near-optimal level set
	\[
	\left\{
	\ww\in\cC:
	\rho_m(-r_{\ww})
	\le
	\rho_m(-r_{\ww_*})+\delta
	\right\}
	\]
	is compact.
	
	Let $(F_c^{(i,*)},b_c^{(i,*)})$ and
	$(F_s^{(i,*)},b_s^{(i,*)})$ denote the inequality constraints that bind at $\ww_*$. The matrix formed by the gradients of all equality constraints and all binding inequality constraints,
	\[
	A_*
	=
	\begin{bmatrix}
		F_c^{(e)}
		\\
		\E[F_s^{(e)}(\rr)]
		\\
		F_c^{(i,*)}
		\\
		\E[F_s^{(i,*)}(\rr)]
	\end{bmatrix},
	\]
	has full row rank. Let
	$(\eta_c^{(e)},\eta_c^{(i)},\eta_s^{(e)},\eta_s^{(i)})$
	denote the resulting unique Lagrange multipliers. Among the binding inequalities, write the subscript $+$ for those with strictly positive multipliers and the subscript $0$ for those with zero multipliers.
\end{ass}

Define the linear subspace generated by the equality and strongly active constraints as
\[
\begin{split}
	\cT_*
	=
	\big\{
	d\in\R^N: \,
	&F_c^{(e)}d=0,\quad
	\E[F_s^{(e)}(\rr)]d=0,
	\\
	&F_{c,+}^{(i,*)}d=0,\quad
	\E[F_{s,+}^{(i,*)}(\rr)]d=0
	\big\}.
\end{split}
\]

\begin{ass}[Second-order regularity]
	\label{ass:second_order_regularity}
	The population objective has positive curvature on $\cT_*$:
	\[
	\inf_{d\in\cT_*:\,\|d\|_2=1}
	d^\top
	\nabla_{\ww}^2\rho_m(-r_{\ww_*})
	d
	>0.
	\]
\end{ass}

The compact level-set condition provides global identification of the optimizer, the full-row-rank condition is the linear independence constraint qualification, and the curvature condition guarantees uniqueness of the local quadratic approximation. The curvature condition is stronger than necessary when weakly binding inequalities are present, but provides a transparent sufficient condition for the result below.

At the $T^{-1/2}$ scale, the empirical objective contributes a Gaussian linear perturbation to the population quadratic expansion, while estimated constraints contribute random shifts to the local feasible set. The limiting estimator is therefore characterized by the following random quadratic program.

\begin{defn}[Limiting stochastic quadratic program]
	\label{defn:asym_dist_low_dim}
	For $h\in\R^N$, define
	\begin{equation}
		\label{eq:defn_M_func}
		\begin{split}
			M(h)
			=
			&\,
			h^\top
			\mathbb{G}
			\left(
			\nabla_{\ww}U_m(\ww_*;\rr)
			+
			F_s^{(e)}(\rr)^\top\eta_s^{(e)}
			+
			F_{s,+}^{(i,*)}(\rr)^\top
			\eta_{s,+}^{(i,*)}
			\right)
			\\
			&+
			\frac12
			h^\top
			\nabla_{\ww}^2\rho_m(-r_{\ww_*})
			h.
		\end{split}
	\end{equation}
	Let $h_{c,s}$ be the unique random solution to
	\[
	\begin{split}
		\mbox{minimize}_{h\in\R^N}
		\quad
		&M(h),
		\\
		\mbox{subject to}
		\quad
		&F_c^{(e)}h=0,
		\qquad
		F_{c,+}^{(i,*)}h=0,
		\qquad
		F_{c,0}^{(i,*)}h\le0,
		\\
		&\E[F_s^{(e)}(\rr)]h
		=
		-\mathbb{G}(F_s^{(e)})\ww_*,
		\\
		&\E[F_{s,+}^{(i,*)}(\rr)]h
		=
		-\mathbb{G}(F_{s,+}^{(i,*)})\ww_*,
		\\
		&\E[F_{s,0}^{(i,*)}(\rr)]h
		\le
		-\mathbb{G}(F_{s,0}^{(i,*)})\ww_*.
	\end{split}
	\]
\end{defn}

Weakly active inequalities enter the random local feasible set but not the linear perturbation of the objective because their population Lagrange multipliers are zero. When no inequality constraint binds with zero multiplier, the limiting problem is a Gaussian linear perturbation of a strictly convex quadratic program and the resulting limit is Gaussian. With weakly binding inequalities, the limit may instead be non-Gaussian.

\begin{thm}[Empirical spectral-risk minimization under estimated constraints]
	\label{thm:normality_general_risk}
	Under Assumptions~\ref{ass:regularity_density},
	\ref{ass:measure_support},
	\ref{ass:identification_constraint}, and
	\ref{ass:second_order_regularity}, we have
	\[
	\hw\pto\ww_*
	\qquad
	\text{as }T\to\infty.
	\]
	If, in addition,
	\[
	\sup_{(\ww,u)\in K\times\R}
	p_{-r_{\ww}}(u)<\infty
	\]
	for every compact set
	$K\subset\R^N\backslash\{\bzero\}$, then
	\[
	\sqrt{T}(\hw-\ww_*)
	\wto
	h_{c,s},
	\]
	where $h_{c,s}$ is defined in
	\Cref{defn:asym_dist_low_dim}. Furthermore,
	\begin{align*}
		\sqrt{T}
		\left(
		\hw^\top\bmu-\ww_*^\top\bmu
		\right)
		&\wto
		h_{c,s}^\top\bmu,
		\\
		\sqrt{T}
		\left(
		\rho_m(-r_{\hw})
		-
		\rho_m(-r_{\ww_*})
		\right)
		&\wto
		h_{c,s}^\top
		\nabla_{\ww}\rho_m(-r_{\ww_*}).
	\end{align*}
\end{thm}

Theorem~\ref{thm:normality_general_risk} is the general statistical result of this section. It combines three features that are essential here: a nonsmooth empirical spectral-risk objective, a feasible set whose coefficients may be estimated from the same data, and inequality constraints whose local activity can produce non-Gaussian limits. Existing empirical-process and sample-average-approximation results provide important ingredients for the argument; the theorem assembles them in a form tailored to inference on the optimizer of an empirical spectral-risk problem. The portfolio-specific result below is obtained directly by choosing the corresponding deterministic and stochastic constraint functions. In the equality-constrained setting used in the remainder of the main paper, the stochastic quadratic program reduces to a Gaussian KKT system.

\subsection{Target-Return Portfolios}

We now specialize the general result to the canonical target-return portfolio problem used throughout the paper:
\[
\ww_*(\mu_0)
=
\arg\min_{\ww\in\R^N}
\rho_m(-r_{\ww}),
\qquad
\mbox{subject to}
\quad
\ww^\top\bone=1,
\quad
\ww^\top\bmu=\mu_0,
\]
and its empirical counterpart
\[
\hw(\mu_0)
=
\arg\min_{\ww\in\R^N}
\rho_m(-\hat r_{\ww,T}),
\qquad
\mbox{subject to}
\quad
\ww^\top\bone=1,
\quad
\ww^\top\hat{\bmu}=\mu_0.
\]

Write $\ww_*=\ww_*(\mu_0)$. The budget constraint guarantees
$\ww_*\neq\bzero$. Let $(\eta_c,\eta_s)\in\R^2$ denote the Lagrange multipliers associated with the budget and expected-return constraints, so that
\[
\nabla_{\ww}\rho_m(-r_{\ww_*})
+
\eta_c\bone
+
\eta_s\bmu
=
0.
\]
Let
\[
\bmu^{(1)}
=
[\,\bmu\ \ \bone\,],
\]
and define
\[
\begin{split}
	H_*
	&=
	\nabla_{\ww}^2\rho_m(-r_{\ww_*}),
	\\
	V_*(\rr)
	&=
	\nabla_{\ww}U_m(\ww_*;\rr)
	+
	\eta_s\rr,
	\\
	r_*(\rr)
	&=
	\ww_*^\top\rr,
	\\
	K_*
	&=
	\begin{bmatrix}
		H_* & \bmu^{(1)}
		\\
		\bmu^{(1)\top} & \bzero
	\end{bmatrix}.
\end{split}
\]

The target-return problem is an immediate specialization of
Theorem~\ref{thm:normality_general_risk}: the budget constraint is deterministic, while the expected-return constraint is estimated from the same observations as the risk criterion. Since there are no inequality constraints, the limiting stochastic program is an equality-constrained quadratic program and hence has a Gaussian solution. Solving its KKT system yields the following corollary.

\begin{cor}[Target-return portfolio]
	\label{prop:normality_risky_constraints}
	Suppose $\ww_*$ is unique, Assumptions~\ref{ass:regularity_density} and \ref{ass:measure_support} hold, the population objective is level-bounded on its constraint set,  and $K_*$ is nonsingular. Let $\hw(\mu_0)$ be any measurable solution of the empirical problem. Then
	\[
	\hw(\mu_0)
	\pto
	\ww_*(\mu_0).
	\]
	If, in addition,
	\[
	\sup_{(\ww,u)\in K\times\R}
	p_{-r_{\ww}}(u)<\infty
	\]
	for every compact set
	$K\subset\R^N\backslash\{\bzero\}$, then
	\[
	\sqrt{T}
	\left(
	\hw(\mu_0)-\ww_*(\mu_0)
	\right)
	\wto
	\normal\left(
	0,\Sigma(\mu_0,\rho_m)
	\right),
	\]
	where
	\[
	\Sigma(\mu_0,\rho_m)
	=
	\var\left(
	\begin{bmatrix}
		\id_N & \bzero
	\end{bmatrix}
	K_*^{-1}
	\begin{bmatrix}
		V_*(\rr)
		\\
		r_*(\rr)
		\\
		0
	\end{bmatrix}
	\right).
	\]
\end{cor}

The covariance has the usual constrained-$\rm M$-estimation interpretation. The inverse KKT matrix $K_*^{-1}$ propagates sampling noise through the local curvature of the objective and the linear constraints, while $V_*(\rr)$ and $r_*(\rr)$ collect the random perturbations from the empirical risk criterion and the estimated expected-return constraint, respectively.

The same general theorem also covers portfolios with unrestricted borrowing and lending at the risk-free rate. In terms of excess returns, the risky-asset portfolio is then subject only to the estimated target-return constraint. Because the argument directly parallels the analysis above, the corresponding asymptotic result and the explicit covariance calculations under normal scale-mixture elliptical returns are collected in Appendix~\ref{app:risk_free}.

\section{Estimation Risk under Elliptically Contoured Returns}
\label{sec:elliptic}

We now specialize the general asymptotic theory to elliptically contoured returns. This case is central to the paper because it separates two issues that are often conflated in portfolio choice: the population efficient frontier and the statistical efficiency of its empirical estimator. The population equivalence between spectral-risk and mean-variance optimization follows from the common location--scale form of linear projections and therefore holds for the broader elliptical family. For the explicit covariance calculations below, we use the normal scale-mixture representation, whose conditional-Gaussian structure permits closed-form expressions.

The detailed covariance expressions are most easily understood through a structural decomposition. The asymptotic covariance of a spectral-risk estimator will take the form
\[
\Sigma_{\mathrm{SRM}}(m)
=
\Sigma_{\mathrm{common}}
+
I(m;c)\Sigma_{\perp},
\qquad
\Sigma_{\perp}\succeq\bzero,
\]
where $\Sigma_{\mathrm{common}}$ captures sampling error common to all spectral measures and $\Sigma_{\perp}$ governs the criterion-dependent component of sampling variability. The corresponding sample mean-variance estimator has asymptotic variance
\[
\Sigma_{\mathrm{MV}}
=
\Sigma_{\mathrm{common}}
+
J(c)\Sigma_{\perp}.
\]
Thus, the matrix comparison ultimately reduces to the scalar comparison $I(m;c)$ versus $J(c)$. The results below derive this decomposition and identify the scalar index $c$.

Throughout this section, we assume that the return vector admits the representation
\begin{equation}
	\rr
	=
	\bmu+\lambda\bSigma^{1/2}\zz,
	\label{eq:elliptical_representation_section4}
\end{equation}
where $\zz\sim\normal(\bzero,\id_N)$, $\lambda\ge0$ is independent of $\zz$, and $\E[\lambda^2]<\infty$. The random variable $\lambda$ captures radial variation and allows for heavy tails beyond the Gaussian model. Throughout \Cref{sec:elliptic,sec:optimal_srm}, we retain $\E[\lambda^2]$ explicitly, so $\bSigma$ is treated as a scatter matrix. The scale decomposition $(\lambda,\bSigma)$ is not separately identified without a normalization. We impose $\E[\lambda^2]=1$ only in \Cref{sec:feasible_rule}, where identification of the radial distribution is required for the data-adaptive construction; under that normalization $\bSigma=\var(\rr)$.

\begin{ass}
	\label{ass:elliptic_return}
	Let $\rr_i\iidsim P_{\rr}$. We assume that $P_{\rr}$ admits the representation \eqref{eq:elliptical_representation_section4}, where $\bmu$ is the mean vector, $\bSigma\succ\bzero$ is the scatter matrix, and $\E[\lambda^2]<\infty$.
\end{ass}

By \cref{eq:elliptical_population_equivalence}, for any portfolio satisfying $\ww^\top\bmu=\mu_0$,
\[
\rho_m(-r_{\ww})
=
-\mu_0
+
\|\bSigma^{1/2}\ww\|_2\,\rho_m(\lambda Z),
\]
where $Z\sim\normal(0,1)$ is independent of $\lambda$. Hence, provided $\rho_m(\lambda Z)>0$, the population spectral-risk portfolio coincides with the corresponding mean-variance efficient portfolio. We therefore focus below on differences in the sampling distributions of their empirical estimators.


Consider the target-return problem
\[
\ww^\top\bone=1,
\qquad
\ww^\top\bmu=\mu_0.
\]
Let $\ww_*(\mu_0)$ and $\hw(\mu_0)$ denote the population and empirical spectral-risk portfolios defined in \cref{eq:general_risk_optimization,eq:empirical_rho_risk_optimization}. Under elliptical returns, $\ww_*(\mu_0)$ is the usual mean-variance efficient portfolio.

Define
\[
\bmu^{(1)}
=
[\,\bmu\ \ \bone\,]
\in\R^{N\times2},
\qquad
\mu_0^{(1)}
=
\begin{bmatrix}
	\mu_0\\
	1
\end{bmatrix},
\qquad
E_{11}
=
\begin{bmatrix}
	1&0\\
	0&0
\end{bmatrix}.
\]
For later use, define
\[
\begin{split}
	P_{\mu}
	&=
	\bSigma^{-1}\bmu^{(1)}
	\left(
	\bmu^{(1)\top}\bSigma^{-1}\bmu^{(1)}
	\right)^{-1}
	E_{11}
	\left(
	\bmu^{(1)\top}\bSigma^{-1}\bmu^{(1)}
	\right)^{-1}
	\bmu^{(1)\top}\bSigma^{-1},
	\\
	P_{\perp}
	&=
	\bSigma^{-1}
	-
	\bSigma^{-1}\bmu^{(1)}
	\left(
	\bmu^{(1)\top}\bSigma^{-1}\bmu^{(1)}
	\right)^{-1}
	\bmu^{(1)\top}\bSigma^{-1}.
\end{split}
\]
The matrix $P_{\perp}$ is positive semidefinite and represents the covariance geometry in directions orthogonal, in the $\bSigma$-induced geometry, to the binding budget and expected-return constraints.

\begin{prop}
	\label{prop:no_risk_free_lowdim}
	Under Assumptions~\ref{ass:regularity_density}--\ref{ass:elliptic_return}, suppose additionally that
	$\rank(\bmu^{(1)})=2$. Then
	\begin{equation}
		\label{eq:w_star_no_risk_free}
		\ww_*(\mu_0)
		=
		\bSigma^{-1}\bmu^{(1)}
		\left(
		\bmu^{(1)\top}\bSigma^{-1}\bmu^{(1)}
		\right)^{-1}
		\mu_0^{(1)}.
	\end{equation}
	Further,
	\[
	\hw(\mu_0)\xrightarrow{p}\ww_*(\mu_0),
	\]
	and
	\[
	\sqrt{T}
	\left(
	\hw(\mu_0)-\ww_*(\mu_0)
	\right)
	\xrightarrow{d}
	\normal\left(
	0,\Sigma(\mu_0,\rho_m)
	\right),
	\]
	where
	\[
	\begin{split}
		\Sigma(\mu_0,\rho_m)
		=
		&\,
		\E[\lambda^2]\,
		\ww_*(\mu_0)^\top\bSigma\ww_*(\mu_0)\,
		P_{\mu}
		\\
		&+
		\frac{
			\ww_*(\mu_0)^\top\bSigma\ww_*(\mu_0)
		}{
			\rho_m(\lambda Z)^2
		}
		\left(
		\int_{[0,1]^2}
		\frac{
			A(\alpha_1,\alpha_2)
		}{
			(1-\alpha_1)(1-\alpha_2)
		}
		\d m(\alpha_1)\d m(\alpha_2)
		\right)
		P_{\perp},
	\end{split}
	\]
	with
	\[
	\begin{split}
		A(\alpha_1,\alpha_2)
		=
		&\,
		\E\left[
		\lambda^2
		\Phi\left(
		-\frac{1}{\lambda}
		q_{\alpha_1\vee\alpha_2}(\lambda Z)
		\right)
		\right]
		\\
		&-
		(1-\alpha_1)\eta_s
		\E\left[
		\lambda^2
		\Phi\left(
		-\frac{1}{\lambda}
		q_{\alpha_2}(\lambda Z)
		\right)
		\right]
		\\
		&-
		(1-\alpha_2)\eta_s
		\E\left[
		\lambda^2
		\Phi\left(
		-\frac{1}{\lambda}
		q_{\alpha_1}(\lambda Z)
		\right)
		\right]
		\\
		&+
		(1-\alpha_1)(1-\alpha_2)
		\eta_s^2
		\E[\lambda^2],
	\end{split}
	\]
	and
	\[
	\eta_s
	=
	1
	-
	\frac{
		\rho_m(\lambda Z)
	}{
		\|\bSigma^{1/2}\ww_*(\mu_0)\|_2
	}
	[\,1\ \ 0\,]
	\left(
	\bmu^{(1)\top}\bSigma^{-1}\bmu^{(1)}
	\right)^{-1}
	\mu_0^{(1)}.
	\]
	Here $\alpha_1 \vee \alpha_2 = \max \{ \alpha_1, \alpha_2 \}$, $Z\sim\normal(0,1)$ is independent of $\lambda$, $\Phi$ is the standard normal distribution function, and $q_\alpha(\lambda Z)$ denotes the $\alpha$-quantile of $\lambda Z$.
\end{prop}

The covariance in \cref{prop:no_risk_free_lowdim} has two components. The first reflects sampling error from estimating expected returns and is common across spectral-risk criteria. The second reflects sampling error from estimating the risk criterion and depends on the spectral measure $m$. Under the stated support restriction, the spectral-risk covariance requires only the second radial moment; the sample mean-variance covariance below additionally requires a finite fourth radial moment.

\paragraph{Comparison with Sample mean-variance Optimization.}

The sample mean-variance estimator solves
\begin{equation}
	\label{eq:empirical_mvo_risky}
	\operatorname{minimize}_{\ww\in\R^N}
	\quad
	\frac12\ww^\top\hat{\bS}\ww,
	\qquad
	\mbox{subject to}
	\quad
	\ww^\top\hat{\bmu}=\mu_0,
	\quad
	\ww^\top\bone=1,
\end{equation}
where
\[
\hat{\bmu}
=
\frac1T\sum_{i=1}^T\rr_i,
\qquad
\hat{\bS}
=
\frac1T\sum_{i=1}^T
(\rr_i-\hat{\bmu})
(\rr_i-\hat{\bmu})^\top.
\]
Its solution is
\[
\hw_{\rm MV}(\mu_0)
=
\hat{\bS}^{-1}\hat{\bmu}^{(1)}
\left(
\hat{\bmu}^{(1)\top}
\hat{\bS}^{-1}
\hat{\bmu}^{(1)}
\right)^{-1}
\mu_0^{(1)},
\]
where
\[
\hat{\bmu}^{(1)}
=
[\,\hat{\bmu}\ \ \bone\,].
\]

\begin{prop}
	\label{prop:asymptotic_variance_mv}
	Under Assumptions~\ref{ass:regularity_density}--\ref{ass:elliptic_return}, suppose additionally that
	$\rank(\bmu^{(1)})=2$ and $\E[\lambda^4]<\infty$. Then
	\[
	\hw_{\rm MV}(\mu_0)
	\xrightarrow{p}
	\ww_*(\mu_0),
	\]
	and
	\[
	\sqrt{T}
	\left(
	\hw_{\rm MV}(\mu_0)-\ww_*(\mu_0)
	\right)
	\xrightarrow{d}
	\normal\left(
	0,\Sigma_{\rm MV}(\mu_0)
	\right),
	\]
	where
	\[
	\begin{split}
		\Sigma_{\rm MV}(\mu_0)
		=
		&\,
		\E[\lambda^2]\,
		\ww_*(\mu_0)^\top\bSigma\ww_*(\mu_0)\,
		P_{\mu}
		\\
		&+
		\left(
		\frac{\E[\lambda^4]}{\E[\lambda^2]^2}
		\ww_*(\mu_0)^\top\bSigma\ww_*(\mu_0)
		+
		\E[\lambda^2](\gamma_s^{(e)})^2
		\right)
		P_{\perp},
	\end{split}
	\]
	with
	\[
	\gamma_s^{(e)}
	=
	[\,1\ \ 0\,]
	\left(
	\bmu^{(1)\top}\bSigma^{-1}\bmu^{(1)}
	\right)^{-1}
	\mu_0^{(1)}.
	\]
\end{prop}

Both $\hw(\mu_0)$ and $\hw_{\rm MV}(\mu_0)$ estimate the same population portfolio $\ww_*(\mu_0)$. Their comparison is therefore a comparison of estimators rather than population objectives.

To express the comparison compactly, define for $c\in\R$
\[
J(c)
=
\frac{\E[\lambda^4]}{\E[\lambda^2]^2}
+
c^2\E[\lambda^2],
\]
and, for $m\in\cuP([0,1])$ satisfying \cref{ass:measure_support},
\[
\begin{split}
	I(m;c)
	=
	\frac{1}{\rho_m(\lambda Z)^2}
	\int_{[0,1]^2}
	\frac{
		A(\alpha_1,\alpha_2;c)
	}{
		(1-\alpha_1)(1-\alpha_2)
	}
	\d m(\alpha_1)\d m(\alpha_2),
\end{split}
\]
where
\[
\begin{split}
	A(\alpha_1,\alpha_2;c)
	=
	&\,
	\cT\bigl(1-(\alpha_1\vee\alpha_2)\bigr)
	\\
	&-
	\left(
	1-c\rho_m(\lambda Z)
	\right)
	\left\{
	(1-\alpha_1)\cT(1-\alpha_2)
	+
	(1-\alpha_2)\cT(1-\alpha_1)
	\right\}
	\\
	&+
	(1-\alpha_1)(1-\alpha_2)
	\left(
	1-c\rho_m(\lambda Z)
	\right)^2
	\E[\lambda^2],
\end{split}
\]
and
\[
\cT(\alpha)
=
\E\left[
\lambda^2
\Phi\left(
\frac{1}{\lambda}
q_\alpha(\lambda Z)
\right)
\right].
\]

Finally, define the scalar population index
\[
\begin{split}
	c(\mu_0)
	&=
	\frac{
		[\,1\ \ 0\,]
		\left(
		\bmu^{(1)\top}
		\bSigma^{-1}
		\bmu^{(1)}
		\right)^{-1}
		\mu_0^{(1)}
	}{
		\|\bSigma^{1/2}\ww_*(\mu_0)\|_2
	}.
\end{split}
\]
The scalar $c(\mu_0)$ summarizes the interaction between the target return and the population mean--scatter geometry of the portfolio problem. Once $c(\mu_0)$ is fixed, the dependence of first-order efficiency on the spectral criterion is entirely through $I(m;c(\mu_0))$.

\begin{cor}[Scalar reduction of the covariance efficiency problem]
	\label{cor:structural_efficiency_risky}
	Under the conditions of Propositions~\ref{prop:no_risk_free_lowdim} and
	\ref{prop:asymptotic_variance_mv}, define
	\[
	\Sigma_{\mathrm{common}}(\mu_0)
	=
	\E[\lambda^2]\,
	\ww_*(\mu_0)^\top\bSigma\ww_*(\mu_0)\,
	P_{\mu},
	\]
	and
	\[
	\Sigma_{\perp}(\mu_0)
	=
	\ww_*(\mu_0)^\top\bSigma\ww_*(\mu_0)\,
	P_{\perp}.
	\]
	Then
	\[
	\Sigma(\mu_0,\rho_m) := \Sigma_{\mathrm{SRM}}(m)
	=
	\Sigma_{\mathrm{common}}(\mu_0)
	+
	I\bigl(m;c(\mu_0)\bigr)
	\Sigma_{\perp}(\mu_0),
	\]
	whereas
	\[
	\Sigma_{\rm MV}(\mu_0) := \Sigma_{\mathrm{MV}}
	=
	\Sigma_{\mathrm{common}}(\mu_0)
	+
	J\bigl(c(\mu_0)\bigr)
	\Sigma_{\perp}(\mu_0).
	\]
	Consequently, for any two admissible spectral measures $m_1$ and $m_2$,
	\[
	I\bigl(m_1;c(\mu_0)\bigr)
	\le
	I\bigl(m_2;c(\mu_0)\bigr)
	\quad\Longrightarrow\quad
	\Sigma_{\mathrm{SRM}}(m_1)
	\preceq
	\Sigma_{\mathrm{SRM}}(m_2).
	\]
	Thus, minimizing asymptotic covariance in Loewner order reduces to the scalar optimization problem
	\[
	\min_m I\bigl(m;c(\mu_0)\bigr).
	\]
\end{cor}

Corollary~\ref{cor:structural_efficiency_risky} is the main statistical implication of the explicit covariance calculations. Sampling error associated with the estimated return constraint is common across criteria, while the criterion-dependent component has a common positive-semidefinite covariance geometry and differs only through the scalar functional $I(m;c)$. Thus, the population portfolio geometry is compressed into the scalar index $c(\mu_0)$, while criterion selection is governed by the functional optimization problem
\[
\min_m I\bigl(m;c(\mu_0)\bigr).
\]
The corresponding extension with a risk-free asset has the same structural form and is collected in Appendix~\ref{app:risk_free}. This scalar reduction is the starting point for the efficiency analysis in the next section.

\section{Optimal Spectral Risk Measures}
\label{sec:optimal_srm}

Corollary~\ref{cor:structural_efficiency_risky} shows that, under elliptically contoured returns, spectral-risk minimization and mean-variance optimization estimate the same population target, while the sampling covariance of the spectral-risk estimator depends on the spectral measure $m$. More specifically, once the population portfolio problem is summarized by the scalar index $c=c(\mu_0)$, minimizing asymptotic covariance in Loewner order reduces to minimizing the scalar functional $I(m;c)$.

We now study this criterion-selection problem. For fixed $u\in(0,1/2)$, let $m_{*,u}^c$ denote the efficiency-optimal measure over the admissible class $\mathscr{P}([u,1-u])$. We first characterize this fixed-$u$ optimizer and then compare the endpoint-relaxed bound
\[
\inf_{m\in\mathscr{P}([0,1])} I(m;c)
\]
with the covariance contribution $J(c)$ of sample mean-variance optimization. Section~\ref{sec:feasible_rule} subsequently develops a data-adaptive estimator $\hat m$ of the fixed-$u$ oracle $m_{*,u}^c$. Throughout, ``optimal'' and ``efficient'' refer to statistical efficiency of the estimating criterion within the spectral-risk class, not to an economically optimal risk preference or to a semiparametric information bound.

\subsection{The Efficiency-Optimal Spectral Measure}
\label{sec:opt_spec_meas}

For fixed $u$ and $c$, consider
\[
\inf_{m\in\mathscr{P}([u,1-u])} I(m;c),
\]
where $I(m;c)$ is defined in \Cref{sec:elliptic}. Under the common-estimand property established above, the optimizer should be interpreted as an efficiency-optimal estimating criterion rather than as an estimate of an investor's latent preference over tail losses.

For later analysis, it is useful to rewrite $I(m;c)$ as
\begin{equation}
	\label{eq:simplified_I_m_c}
	\begin{split}
		I(m;c)
		=
		\frac{1}{\rho_m(\lambda Z)^2}
		\bigg(
		&\int_{[0,1]^2}
		\frac{\cT\bigl(1-(\alpha_1\vee\alpha_2)\bigr)}
		{(1-\alpha_1)(1-\alpha_2)}
		\d m(\alpha_1)\d m(\alpha_2)
		\\
		&-
		2\left(1-c\rho_m(\lambda Z)\right)
		\int_0^1
		\frac{\cT(1-\alpha)}
		{1-\alpha}
		\d m(\alpha)
		\\
		&+
		\left(1-c\rho_m(\lambda Z)\right)^2
		\E[\lambda^2]
		\bigg).
	\end{split}
\end{equation}
The risk of the standardized shock can be written as
\[
\rho_m(\lambda Z)
=
\int_0^1
\frac{\cU(1-\alpha)}
{1-\alpha}
\d m(\alpha),
\]
where
\[
\cU(\alpha)
=
\E\left[
\lambda
\phi\left(
\frac{1}{\lambda}
q_{\alpha}(\lambda Z)
\right)
\right],
\]
and $\phi$ denotes the standard normal density.

For fixed $u\in(0,1/2)$, compactness of $\mathscr{P}([u,1-u])$ under weak convergence and continuity of $I(\cdot;c)$ imply existence of a minimizer. Under the mild additional condition $\lambda>0$ almost surely, the minimizer is unique.

\begin{prop}
	\label{prop:unique_minimizer}
	If $\lambda>0$ a.s., then for every $c\in\R$ there exists a unique minimizer
	\[
	m_{*,u}^c
	=
	\arg\min_{m\in\mathscr{P}([u,1-u])} I(m;c).
	\]
\end{prop}

The uniqueness in Proposition~\ref{prop:unique_minimizer} is important for the data-adaptive construction in \Cref{sec:feasible_rule}. It gives criterion selection a single population target, so uniform convergence of the estimated efficiency functional yields convergence of $\hat m$ to $m_{*,u}^c$, rather than only convergence to an argmin set.

The next result characterizes the interior structure of the efficiency-optimal measure.

\begin{thm}
	\label{thm:continuous_optimal_measure}
	If $\E[1/\lambda]<\infty$, then for every $c\in\R$, the unique optimizer $m_{*,u}^c$ is absolutely continuous with respect to Lebesgue measure on the open interval $(u,1-u)$.
\end{thm}

Thus, the optimizer has no singular component in the interior of the admissible interval. In particular, a point mass $m=\delta_\alpha$ with $\alpha\in(u,1-u)$, corresponding to a single-level CVaR criterion, cannot be efficiency-optimal.

\begin{cor}[Inefficiency of a fixed CVaR level]
	\label{cor:single_cvar_inefficient}
	Under the assumptions of \Cref{thm:continuous_optimal_measure}, no point mass $m=\delta_\alpha$ with $\alpha\in(u,1-u)$ minimizes $I(m;c)$ over $\mathscr{P}([u,1-u])$. Consequently, no fixed interior CVaR level attains the minimum asymptotic covariance within the admissible spectral-risk class.
\end{cor}

The result is statistical rather than preference-based. A high CVaR level concentrates the empirical criterion on a relatively small number of tail observations and can therefore produce a noisy empirical criterion. A spectral mixture combines information across quantile levels and can trade off tail emphasis against sampling variability. Accordingly, the relatively low CVaR levels receiving substantial weight in the examples below should not be interpreted as preferred tail probabilities; they contribute information that stabilizes estimation of a population portfolio common to the spectral class.

\subsection{Comparison with Sample mean-variance Optimization}

We next compare the efficiency achievable by spectral-risk criteria with the sample mean-variance benchmark. Recall from \Cref{sec:elliptic} that the criterion-dependent covariance contributions are summarized by $I(m;c)$ for spectral-risk estimation and
\[
J(c)
=
\frac{\E[\lambda^4]}{\E[\lambda^2]^2}
+
c^2\E[\lambda^2]
\]
for sample mean-variance optimization.

The following theorem shows that the best attainable covariance bound over the closure of the spectral-risk class is no larger than the MVO benchmark.

\begin{thm}
	\label{thm:SRM_optimality}
	For any $c\in\R$,
	\[
	\inf_{m\in\mathscr{P}([0,1])} I(m;c)
	\le
	J(c).
	\]
	Consequently, for any $\veps>0$, there exists $u\in(0,1/2)$ such that
	\[
	\inf_{m\in\mathscr{P}([u,1-u])} I(m;c)
	\le
	J(c)+\veps.
	\]
\end{thm}

The theorem concerns an infimum and does not assert that an admissible fixed-$u$ spectral measure always attains covariance weakly below the MVO covariance. Rather, it establishes that the best covariance bound over the closure of the spectral-risk class is no larger than the MVO benchmark. If an admissible compactly supported measure $m$ satisfies $I(m;c)<J(c)$, its portfolio estimator strictly improves on MVO in Loewner order. When equality holds only at the endpoint-relaxed infimum, admissible spectral measures can instead approximate the MVO benchmark arbitrarily closely.

This comparison is especially meaningful because both estimators target the same population portfolio. Any strict gain is therefore a gain in statistical efficiency rather than a consequence of changing the population investment objective. Together with \Cref{thm:continuous_optimal_measure,cor:single_cvar_inefficient}, Theorem~\ref{thm:SRM_optimality} yields a simple message: efficient estimation generally requires aggregating information across quantile levels, and doing so can recover or improve upon the sample-MVO efficiency benchmark.

The theorem also clarifies the role of the endpoint restriction in \Cref{ass:measure_support}. For fixed $u$, $m_{*,u}^c$ is the efficiency-optimal criterion within $\mathscr{P}([u,1-u])$. Letting $u\downarrow0$ enlarges the admissible class and yields the endpoint-relaxed comparison with MVO.

\subsection{Gaussian and Laplace Benchmarks}
\label{sec:optimal_examples}

We illustrate the efficiency-optimal spectral measure in two benchmark models. The Gaussian case admits a closed-form solution and provides a natural efficiency benchmark. The Laplace model illustrates how heavy tails alter the optimal spectral weights and can generate strict improvements relative to sample mean-variance optimization.

\paragraph{Gaussian returns.}

When returns are Gaussian, $\lambda\equiv1$. Hence
\[
\cT(\alpha)=\alpha,
\qquad
\cU(\alpha)=\phi(q_\alpha(Z)).
\]
In this case,
\begin{align*}
	I(m;c)
	&=
	c^2
	+
	\frac{1}{\rho_m(Z)^2}
	\left(
	\int_{[0,1]^2}
	\frac{\cT\bigl(1-(\alpha_1\vee\alpha_2)\bigr)}
	{(1-\alpha_1)(1-\alpha_2)}
	\d m(\alpha_1)\d m(\alpha_2)
	-
	1
	\right)
	\\
	&=
	c^2
	+
	\frac{1}{\rho_m(Z)^2}
	\int_{[0,1]^2}
	\frac{\alpha_1\wedge\alpha_2}
	{1-(\alpha_1\wedge\alpha_2)}
	\d m(\alpha_1)\d m(\alpha_2),
\end{align*}
where $\alpha_1 \wedge \alpha_2 = \min \{ \alpha_1, \alpha_2 \}$, and
\[
\rho_m(Z)
=
\int_0^1
\frac{\cU(1-\alpha)}
{1-\alpha}
\d m(\alpha)
=
\int_0^1
\frac{\phi(q_{1-\alpha}(Z))}
{1-\alpha}
\d m(\alpha).
\]
The efficiency-optimal measure is therefore independent of $c$.

\begin{prop}
	\label{prop:optimal_measure_gaussian}
	Under Gaussian returns, the minimizer over $\mathscr{P}([u,1-u])$ is
	\[
	m_{*,u}
	=
	m_{*,u}(\{u\})\delta_u
	+
	f_*(\alpha)\bone_{(u,1-u)}(\alpha)\d\alpha
	+
	m_{*,u}(\{1-u\})\delta_{1-u},
	\]
	where
	\begin{align*}
		f_*(\alpha)
		=
		\frac{u(1-\alpha)}
		{\cU(1-u)\cU(1-\alpha)},
		\qquad 
		m_{*,u}(\{u\})
		=
		(1-u)
		\left[
		1+
		\frac{u\cU'(1-u)}
		{\cU(1-u)}
		\right],
	\end{align*}
	and
	\begin{align*}
		m_{*,u}(\{1-u\})
		=
		u
		\left[
		1+
		\frac{u\cU'(1-u)}
		{\cU(1-u)}
		\right].
	\end{align*}
	Furthermore,
	\[
	I(m_{*,u};c)
	=
	c^2
	+
	\left[
	1-2u+
	\frac{2\cU(1-u)^2}{u}
	+
	2\cU(1-u)\cU'(1-u)
	\right]^{-1}.
	\]
\end{prop}

As $u\to0$,
\[
I(m_{*,u};c)\to1+c^2=J(c).
\]
Thus, in the Gaussian case, the best fixed-$u$ spectral-risk estimator approaches the asymptotic covariance of sample mean-variance optimization as the endpoint restriction is relaxed, but does not yield a strict first-order improvement. This provides a useful benchmark: under Gaussianity, the return distribution is fully characterized by its first two moments, precisely the quantities used by the Markowitz criterion. The optimized spectral criterion therefore recovers the same limiting first-order covariance benchmark rather than improving upon it.

For a single CVaR level $m=\delta_\alpha$,
\[
I(\delta_\alpha;0)
=
\frac{\alpha(1-\alpha)}
{\phi(q_{1-\alpha}(Z))^2}.
\]
This quantity is minimized at $\alpha=1/2$ and increases sharply as $\alpha$ approaches the endpoints. Panels (a)--(b) of Figure~\ref{fig:optimal_spectral_examples} illustrate the efficiency-optimal spectral measure and the corresponding single-CVaR comparison.

\paragraph{Laplace returns.}

We next consider a heavy-tailed elliptical model in which $\lambda$ follows the Rayleigh distribution with density
\[
p_\lambda(x)
=
x\exp(-x^2/2),
\qquad x\ge0.
\]
Then $\lambda Z$ has the standard Laplace density
\[
p_{\lambda Z}(x)
=
\frac12\exp(-|x|).
\]
Under this parameterization,
$
\E[\lambda^2]=2$,
and
$\E[\lambda^4]=8$,
consistent with the scatter normalization retained in Sections~\ref{sec:elliptic}--\ref{sec:optimal_srm}. Hence
\[
J(c)=2+2c^2.
\]

Unlike the Gaussian case, the efficiency-optimal measure $m_{*,u}^c$ generally depends on $c$ and does not admit a closed-form expression. We compute it numerically by discretizing $[u,1-u]$ and solving the corresponding finite-dimensional optimization problem over the probability vector $m$. Panel (c) of Figure~\ref{fig:optimal_spectral_examples} displays the resulting densities for $c\in\{1/4,1/2,3/4,1\}$ and $u=0.001$. The optimal measure places substantial mass near small values of $\alpha$, and its shape changes with $c$, reflecting dependence of the efficiency-optimal criterion on the target portfolio through the scalar index $c$. The discontinuity around $\alpha=1/2$ reflects the kink in the Laplace density at zero.

For a single CVaR level $m=\delta_\alpha$,
\[
I(\delta_\alpha;c)
=
\frac{1}{\rho_m(\lambda Z)^2}
\left(
\frac{\cT(1-\alpha)}
{(1-\alpha)^2}
-
2\left(1-c\rho_m(\lambda Z)\right)
\frac{\cT(1-\alpha)}
{1-\alpha}
+
2\left(1-c\rho_m(\lambda Z)\right)^2
\right),
\]
where
\[
\rho_m(\lambda Z)
=
\frac{\cU(1-\alpha)}
{1-\alpha}.
\]

In this Laplace example, CVaR at $\alpha=1/2$ coincides with the MVO benchmark. Indeed,
\[
\cT(1/2)
=
\frac12\E[\lambda^2]
=
1,
\qquad
\cU(1/2)
=
\frac{1}{\sqrt{2\pi}}\E[\lambda]
=
\frac12,
\]
which implies
\[
I(\delta_{0.5};c)
=
2+2c^2
=
J(c).
\]
Other CVaR levels can be substantially less efficient, especially for small $c$.

Panel (d) of Figure~\ref{fig:optimal_spectral_examples} compares $I(m_{*,u}^c;c)$ with $I(\delta_\alpha;c)$ for $\alpha\in\{0.1,0.5,0.9\}$ and with the MVO benchmark $J(c)$. Over the range of $c$ shown, the efficiency-optimal spectral measure has a smaller asymptotic variance contribution than sample mean-variance optimization, while the performance of single-level CVaR depends strongly on $\alpha$.

\begin{figure}[!ht]
	\centering
	
	\begin{subfigure}[t]{0.48\textwidth}
		\centering
		\includegraphics[width=\linewidth]{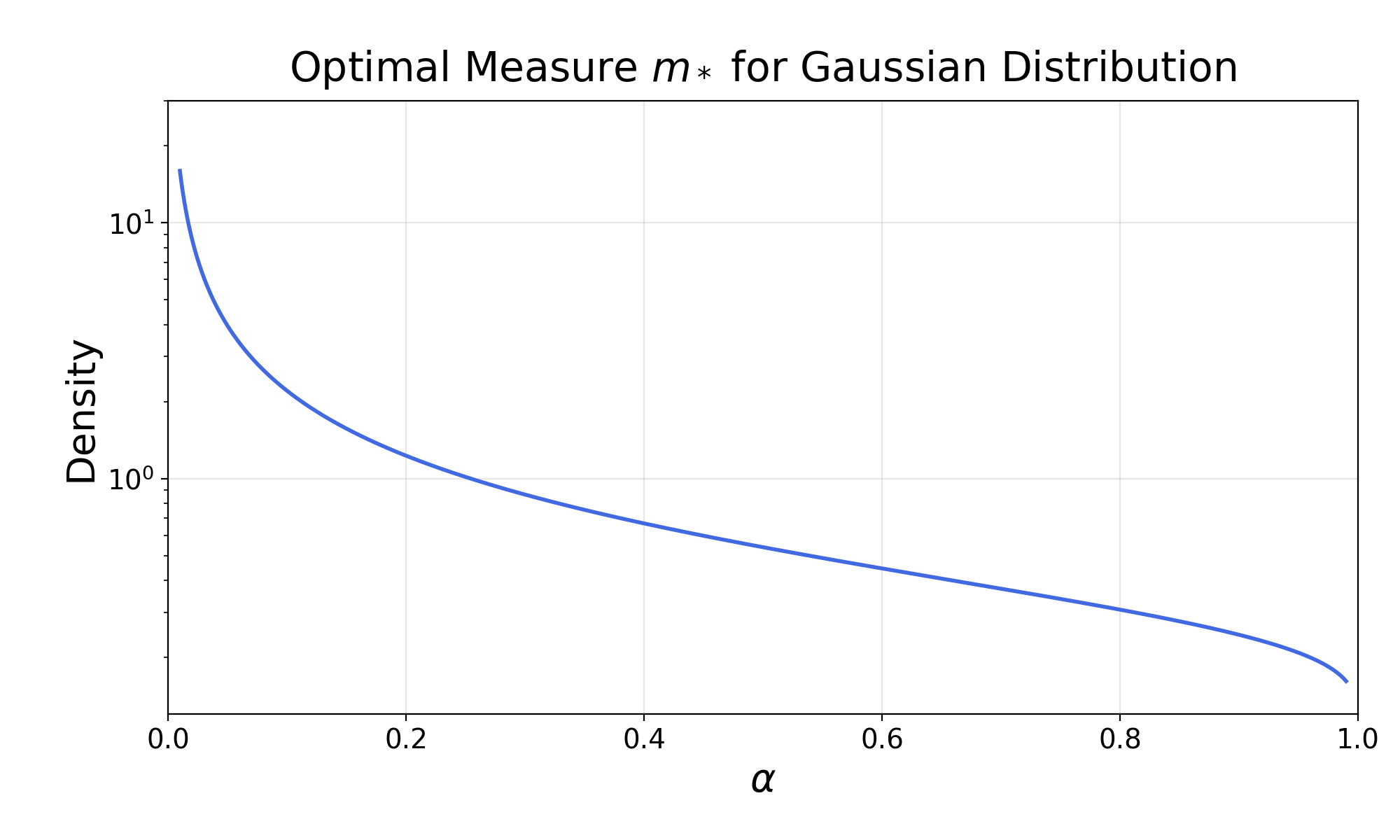}
		\caption{Efficiency-optimal spectral measure under Gaussian returns
			for $u=0.01$. The $y$-axis is logarithmic.}
		\label{fig:gaussian_density}
	\end{subfigure}
	\hfill
	\begin{subfigure}[t]{0.48\textwidth}
		\centering
		\includegraphics[width=\linewidth]{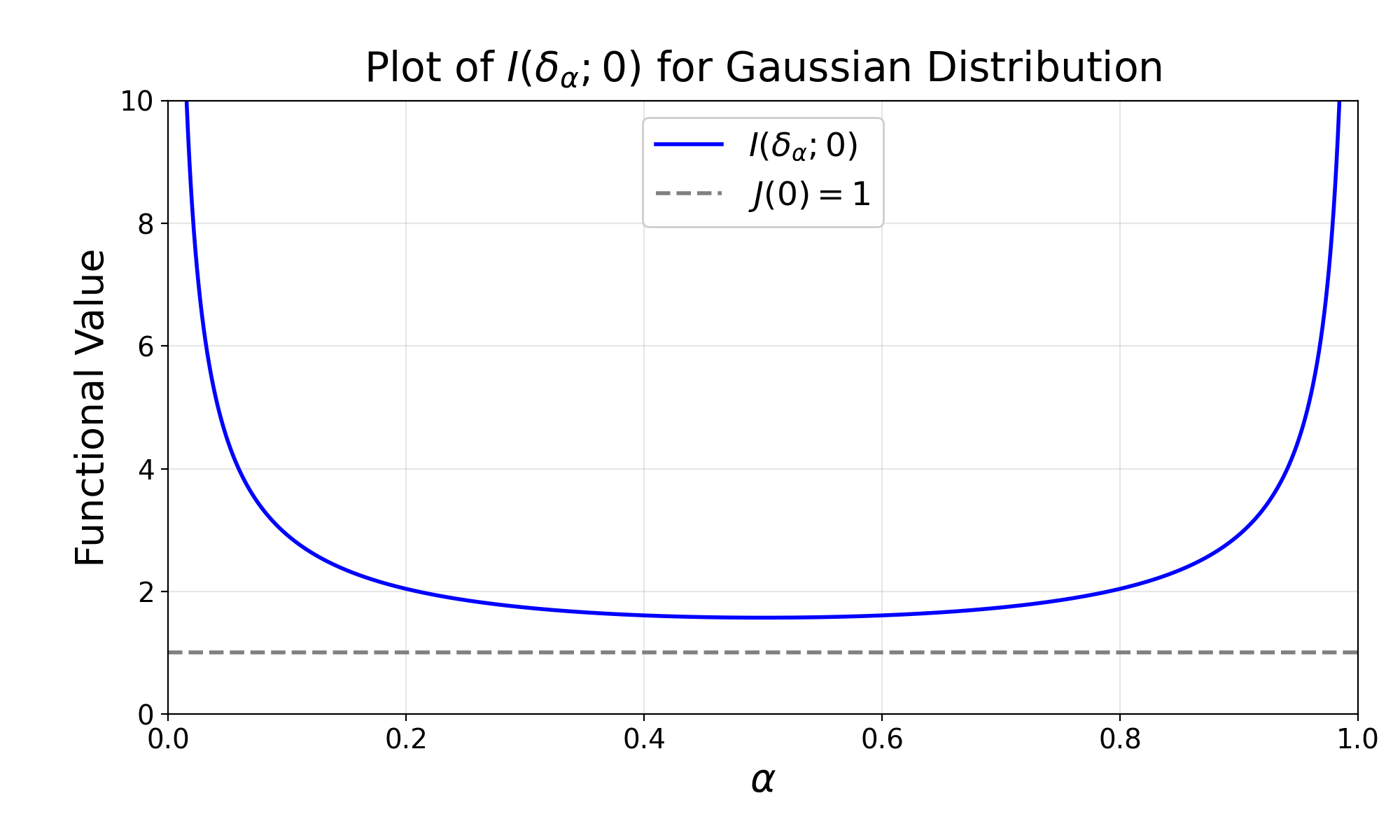}
		\caption{Single-level CVaR under Gaussian returns. The dashed line
			is the endpoint-relaxed benchmark $J(0)=1$.}
		\label{fig:gaussian_cvar}
	\end{subfigure}
	
	\vspace{0.3cm}
	
	\begin{subfigure}[t]{0.48\textwidth}
		\centering
		\includegraphics[width=\linewidth]{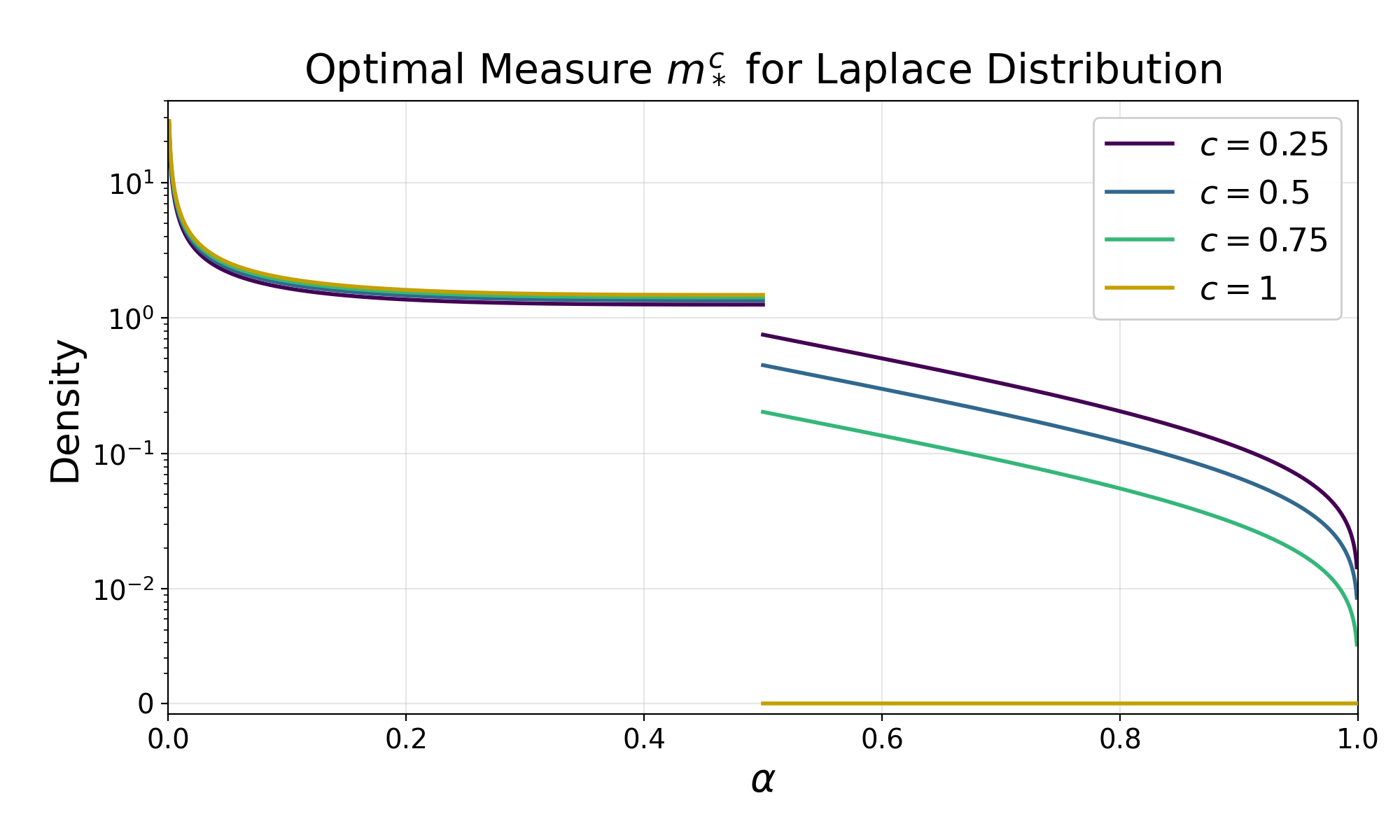}
		\caption{Efficiency-optimal spectral measures under Laplace returns
			for $c\in\{1/4,1/2,3/4,1\}$ and $u=0.001$. At $c=1$, the density is
			zero for $\alpha>1/2$. The $y$-axis is logarithmic.}
		\label{fig:laplace_density}
	\end{subfigure}
	\hfill
	\begin{subfigure}[t]{0.48\textwidth}
		\centering
		\includegraphics[width=\linewidth]{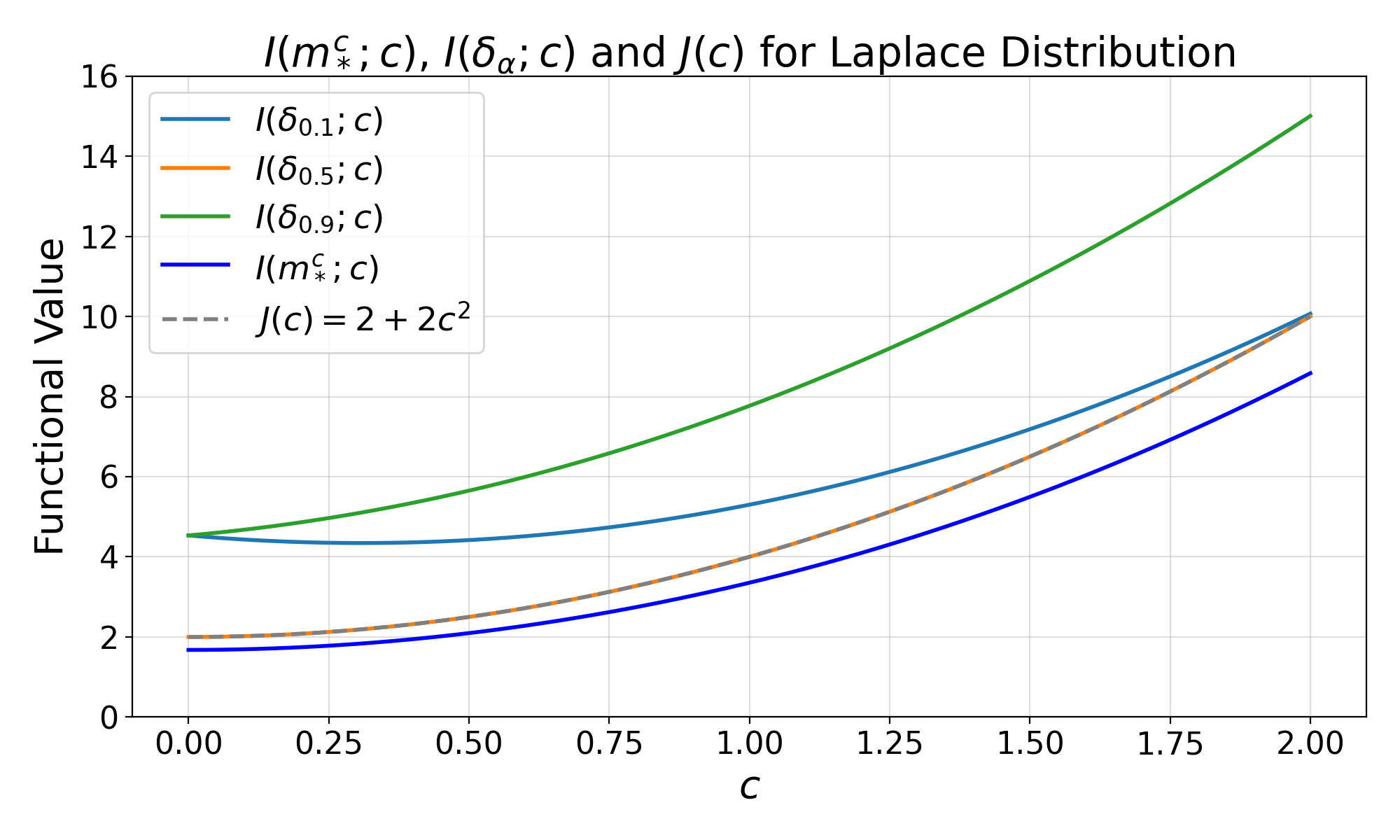}
		\caption{Laplace efficiency comparison: optimized spectral measure,
			single-level CVaR at $\alpha\in\{0.1,0.5,0.9\}$, and MVO.}
		\label{fig:laplace_efficiency}
	\end{subfigure}
	
	\caption{
		Efficiency-optimal spectral measures and asymptotic covariance
		comparisons under Gaussian and Laplace returns. Panels (a)--(b)
		show the Gaussian benchmark; panels (c)--(d) show how heavy-tailed
		radial variation changes the efficiency-optimal criterion and can
		produce strict improvements over sample mean-variance optimization.
	}
	\label{fig:optimal_spectral_examples}
\end{figure}

These benchmarks illustrate the main implication of our theory. Under Gaussian returns, the efficiency-optimal spectral criterion approaches the sample mean-variance first-order covariance benchmark as the endpoint restriction is relaxed. For heavy-tailed elliptical models such as the Laplace example, a suitably chosen spectral risk measure can estimate the same population efficient portfolio with lower sampling variability than sample mean-variance optimization. The next section turns this oracle characterization into a fully data-adaptive procedure that learns the radial distribution, selects the estimating criterion, and estimates the portfolio weights from the same observations.

\section{Data-Adaptive Efficient Portfolio Estimation}
\label{sec:feasible_rule}

The preceding analysis characterizes a fixed-$u$ oracle estimating criterion:
$m_{*,u}^c$, the spectral measure that minimizes the first-order covariance
of the portfolio weights over $\mathscr{P}([u,1-u])$ when the radial
distribution is known. This is distinct from the endpoint-relaxed infimum
over $\mathscr{P}([0,1])$ used for the theoretical comparison with MVO.
We now make the criterion itself data-adaptive. Starting from the observed
returns, we estimate the radial distribution $P_\lambda$, use that estimate
to learn the efficiency-optimal spectral measure, and then estimate the
portfolio weights by minimizing the selected empirical risk criterion.
Thus, the estimating criterion and the portfolio weights are learned from
the same data.

The central result of this section is an oracle-adaptivity theorem.
Although the nuisance distribution is infinite-dimensional and the selected
spectral measure is random and constructed from the same observations used
in the final optimization, the data-adaptive portfolio estimator has the
same first-order distribution as its oracle counterpart. Thus, the
fixed-$u$ oracle efficiency bound characterized in
\Cref{sec:optimal_srm} can be attained without knowing the radial
distribution and without sample splitting.

Throughout this section, we assume in addition that $\bSigma$ is positive
definite and that $\lambda>0$ almost surely. From this point onward, we
impose the normalization
\[
\E[\lambda^2]=1.
\]
Under this normalization, $\bSigma=\var(\rr)$ and the radial distribution
$P_\lambda$ is separately identified from the scatter matrix.

\subsection{Estimating the Radial Distribution}
\label{sec:estimate_radial_dist}

Suppose $P_{\rr}$ satisfies \Cref{ass:elliptic_return} and the
representation \eqref{eq:elliptical_representation_section4}. Under the
normalization $\E[\lambda^2]=1$, the mean vector $\bmu$ and covariance
matrix $\bSigma$ can be consistently estimated by
\[
\hat{\bmu}
=
\frac{1}{T}\sum_{i=1}^T\rr_i,
\qquad
\hat{\bSigma}
=
\frac{1}{T}\sum_{i=1}^T
(\rr_i-\hat{\bmu})(\rr_i-\hat{\bmu})^\top.
\]

For the first three subsections, we leave the estimator of the radial
distribution abstract. Let $\hat P_\lambda$ be any random probability
measure on $(0,\infty)$ satisfying
\begin{equation}
	\label{eq:radial_high_level_condition}
	W_2(\hat P_\lambda,P_\lambda)
	\pto 0.
\end{equation}
This high-level condition is all that is required for the criterion-selection
and oracle-adaptivity results below. In
\Cref{sec:radial_npmle}, we construct a concrete nonparametric estimator
based on estimated Mahalanobis radii and verify
\eqref{eq:radial_high_level_condition}.

\subsection{Learning the Efficiency-Optimal Spectral Measure}
\label{sec:plugin_optimal_measure}

Fix a target return $\mu_0$ and write
\[
c=c(\mu_0),
\]
where $c(\mu_0)$ is the scalar population index defined in
\Cref{sec:elliptic}. The fixed-$u$ oracle criterion is
\[
m_{*,u}^c
=
\arg\min_{m\in\mathscr{P}([u,1-u])}
I(m;c).
\]
By \Cref{prop:unique_minimizer}, this optimizer is unique when
$\lambda>0$ almost surely.

We estimate the ingredients of $I(m;c)$ by plug-in. Let
$\hat q_\alpha$ denote the $\alpha$-quantile of $\hat\lambda Z$, where
$\hat\lambda\sim\hat P_\lambda$ is independent of
$Z\sim\normal(0,1)$. Define
\[
\hat{\cT}(\alpha)
=
\hat{\E}_{\lambda}
\left[
\lambda^2
\Phi\left(
\frac{\hat q_\alpha}{\lambda}
\right)
\right],
\qquad
\hat{\cU}(\alpha)
=
\hat{\E}_{\lambda}
\left[
\lambda
\phi\left(
\frac{\hat q_\alpha}{\lambda}
\right)
\right],
\]
where $\hat{\E}_{\lambda}$ denotes expectation under $\hat P_\lambda$.

The scalar $c(\mu_0)$ is estimated by its empirical analogue. Using the
notation of \Cref{sec:elliptic}, define
\[
\hat{\bmu}^{(1)}
=
[\,\hat{\bmu}\ \ \bone\,],
\qquad
\mu_0^{(1)}
=
\begin{bmatrix}
	\mu_0\\
	1
\end{bmatrix},
\]
and let
\[
\hat{\ww}_{\rm MV}(\mu_0)
=
\hat{\bSigma}^{-1}\hat{\bmu}^{(1)}
\left(
\hat{\bmu}^{(1)\top}
\hat{\bSigma}^{-1}
\hat{\bmu}^{(1)}
\right)^{-1}
\mu_0^{(1)}.
\]
We then set
\[
\hat c
=
\frac{
	[\,1\ \ 0\,]
	\left(
	\hat{\bmu}^{(1)\top}
	\hat{\bSigma}^{-1}
	\hat{\bmu}^{(1)}
	\right)^{-1}
	\mu_0^{(1)}
}{
	\|\hat{\bSigma}^{1/2}
	\hat{\ww}_{\rm MV}(\mu_0)\|_2
}.
\]

Define $\hat I(m;\hat c)$ by replacing
$P_\lambda$, $c$, $\cT$, and $\cU$ in $I(m;c)$ by
$\hat P_\lambda$, $\hat c$, $\hat{\cT}$, and $\hat{\cU}$,
respectively. The data-adaptive spectral measure is any measurable
minimizer
\[
\hat m
\in
\arg\min_{m\in\mathscr{P}([u,1-u])}
\hat I(m;\hat c).
\]
Finite-sample uniqueness of this optimization problem is not required for
the asymptotic results below; population uniqueness of $m_{*,u}^c$,
together with uniform convergence of the estimated criterion, identifies
the limiting target.

\begin{prop}
	\label{prop:plugin_optimal_measure}
	Assume that $\bmu^{(1)}=[\,\bmu\ \ \bone\,]$ has full column rank and
	that \eqref{eq:radial_high_level_condition} holds. Then
	\[
	\hat c\pto c
	\]
	and
	\[
	\sup_{m\in\mathscr{P}([u,1-u])}
	\left|
	\hat I(m;\hat c)-I(m;c)
	\right|
	\pto0.
	\]
	Consequently,
	\[
	d_{\rm BL}(\hat m,m_{*,u}^c)
	\pto0,
	\]
	where $d_{\rm BL}$ denotes the bounded-Lipschitz metric on probability
	measures on $[u,1-u]$. Equivalently, $\hat m$ converges weakly in
	probability to $m_{*,u}^c$.
\end{prop}

\subsection{Oracle Adaptivity of the Data-Adaptive Rule}
\label{sec:oracle_adaptivity}

For any $m\in\mathscr{P}([u,1-u])$, define
\[
\hat{\ww}_m(\mu_0)
=
\arg\min_{\ww\in\R^N}
\rho_m(-\hat r_{\ww,T}),
\qquad
\mbox{subject to}
\quad
\ww^\top\hat{\bmu}=\mu_0,
\quad
\ww^\top\bone=1.
\]
We allow $m$ to be random. The estimator
$\hat{\ww}_{\hat m}(\mu_0)$ is the data-adaptive portfolio rule, whereas
$\hat{\ww}_{m_{*,u}^c}(\mu_0)$ is the infeasible oracle rule that uses
the efficiency-optimal spectral measure determined by the population
radial distribution.

By \Cref{cor:structural_efficiency_risky}, the oracle asymptotic covariance
is
\[
\Sigma\left(
\mu_0,\rho_{m_{*,u}^c}
\right).
\]
The next theorem shows that estimating the radial distribution and
selecting the spectral measure from the same observations used in the
final portfolio optimization has no first-order effect on the portfolio
weights.

\begin{thm}[Oracle adaptivity]
	\label{thm:feasible_asymptotic_normality}
	Suppose the conditions of
	\Cref{prop:normality_risky_constraints} hold, and suppose
	\Cref{ass:elliptic_return} holds with $\lambda>0$ almost surely and
	$\E[\lambda^2]=1$. If \eqref{eq:radial_high_level_condition} holds, then
	\[
	\hat{\ww}_{\hat m}(\mu_0)
	\pto
	\ww_*(\mu_0),
	\]
	and
	\[
	\sqrt{T}
	\left(
	\hat{\ww}_{\hat m}(\mu_0)-\ww_*(\mu_0)
	\right)
	\wto
	\normal\left(
	0,
	\Sigma\left(
	\mu_0,\rho_{m_{*,u}^c}
	\right)
	\right).
	\]
\end{thm}

Theorem~\ref{thm:feasible_asymptotic_normality} shows that
\[
\sqrt{T}
\left(
\hat{\ww}_{\hat m}(\mu_0)-\ww_*(\mu_0)
\right)
\quad\text{and}\quad
\sqrt{T}
\left(
\hat{\ww}_{m_{*,u}^c}(\mu_0)-\ww_*(\mu_0)
\right)
\]
have the same weak limit, even though the radial distribution, the
estimating criterion, and the portfolio weights are all learned using
the same observations. No sample splitting is required.

This conclusion is stronger than consistency of $\hat m$. The mechanism
is a combination of uniform stability of the efficiency criterion and
stochastic equicontinuity of the empirical portfolio objective over
spectral measures in the admissible class.
Proposition~\ref{prop:plugin_optimal_measure} gives
\[
\hat m \to m_{*,u}^c
\]
weakly in probability through uniform convergence of $\hat I$, while
the proof of Theorem~\ref{thm:feasible_asymptotic_normality} establishes
the stronger local equivalence required for oracle adaptivity: replacing
$m_{*,u}^c$ by the random $\hat m$ perturbs the root-$T$ local empirical
optimization problem by only $o_p(1)$. Consistency of the selected
criterion alone would not imply the oracle first-order limit.

\begin{cor}[Data-adaptive efficient estimation]
	\label{cor:data_adaptive_efficiency}
	Under the conditions of
	\Cref{thm:feasible_asymptotic_normality}, for fixed
	$u\in(0,1/2)$,
	\[
	\sqrt{T}
	\left(
	\hat{\ww}_{\hat m}(\mu_0)-\ww_*(\mu_0)
	\right)
	\wto
	\normal\left(
	0,\Sigma_u^*(\mu_0)
	\right),
	\]
	where
	\[
	\Sigma_u^*(\mu_0)
	:=
	\Sigma\left(
	\mu_0,\rho_{m_{*,u}^c}
	\right)
	\]
	is the minimum asymptotic covariance, in Loewner order, among the
	admissible empirical spectral-risk estimators indexed by
	$m\in\mathscr{P}([u,1-u])$. Hence, learning the radial distribution and
	selecting the estimating criterion from the same data does not alter the
	first-order covariance bound attainable by the fixed-$u$ oracle
	spectral-risk estimator.
\end{cor}

\subsection{Practical Implementation}
\label{sec:radial_npmle}

The oracle-adaptivity theory above requires only a radial-distribution
estimator satisfying
\eqref{eq:radial_high_level_condition}; it is not tied to a particular
first-stage method. We now construct one such estimator using a
constrained nonparametric maximum-likelihood estimator. An important
feature is that the mixing distribution is estimated from Mahalanobis
radii computed with the same estimated mean and covariance matrix. The
resulting radii are therefore neither oracle radii nor independent, and
the consistency theorem below explicitly accommodates this plug-in
dependence.

It is convenient to work with
\[
V=\lambda^2,
\qquad
P_V=\operatorname{Law}(V).
\]
The normalization $\E[\lambda^2]=1$ implies
\[
\int_0^\infty v\,\d P_V(v)=1.
\]
Define the estimated squared Mahalanobis radii
\begin{equation}
	\label{eq:estimated_squared_radius}
	\hat D_i
	=
	(\rr_i-\hat{\bmu})^\top
	\hat{\bSigma}^{-1}
	(\rr_i-\hat{\bmu}),
	\qquad
	i=1,\ldots,T.
\end{equation}
Their oracle counterparts satisfy
\[
D_i
=
(\rr_i-\bmu)^\top
\bSigma^{-1}
(\rr_i-\bmu)
=
V_i\|\zz_i\|_2^2,
\]
where
\[
V_i\sim P_V,
\qquad
\zz_i\sim\normal(\bzero,\id_N),
\]
independently. Consequently,
\[
\frac{D_i}{V_i}\sim\chi_N^2,
\]
independently of $V_i$, and the conditional density of $D_i$ given
$V_i=v$ is
\begin{equation}
	\label{eq:scaled_chisquare_kernel}
	k_N(x\mid v)
	=
	\frac{x^{N/2-1}}
	{2^{N/2}\Gamma(N/2)v^{N/2}}
	\exp\left(-\frac{x}{2v}\right),
	\qquad
	x>0,\quad v>0.
\end{equation}

Estimating $P_V$ is therefore a one-dimensional nonparametric mixture
problem. We use a constrained version of the classical NPMLE of
\citep{kiefer1956consistency,lindsay1983geometry}. Let
$0<a_T<1<b_T<\infty$ be deterministic sequences with
$a_T\to0$ and $b_T\to\infty$, and define
\[
\mathcal P_T
=
\left\{
P\in\mathscr{P}([a_T,b_T]):
\int_{a_T}^{b_T}v\,\d P(v)=1
\right\}.
\]
The estimator $\hat P_V$ maximizes the mixture log-likelihood:
\begin{equation}
	\label{eq:mixing_npmle}
	\hat P_V
	\in
	\argmax_{P\in\mathcal P_T}
	\left\{
	\frac1T
	\sum_{i=1}^T
	\log\left(
	\int_0^\infty
	k_N(\hat D_i\mid v)\,
	\d P(v)
	\right)
	\right\}.
\end{equation}
For $\hat D_i>0$, the kernel
$k_N(\hat D_i\mid v)$ is strictly positive and continuous in $v$ on
the compact interval $[a_T,b_T]$. Hence the objective in
\eqref{eq:mixing_npmle} is continuous under weak convergence, and
compactness of $\mathcal P_T$ guarantees existence of a maximizer.

Define the estimator of $P_\lambda$ by the square-root transformation
\begin{equation}
	\label{eq:lambda_npmle}
	\hat P_\lambda(A)
	=
	\hat P_V
	\left(
	\{v:\sqrt v\in A\}
	\right)
\end{equation}
for every Borel set $A\subset(0,\infty)$.

\begin{thm}[Consistency of the radial NPMLE]
	\label{thm:radial_npmle_consistency}
	Suppose \Cref{ass:elliptic_return} holds,
	$\bSigma\succ\bzero$, $\lambda>0$ almost surely, and
	\[
	\E[\lambda^2]=1,
	\qquad
	\E[\lambda^4]<\infty,
	\qquad
	\E[|\log\lambda|]<\infty.
	\]
	Further, suppose
	\[
	a_T\to0,
	\qquad
	b_T\to\infty,
	\qquad
	\sqrt{T}\,a_T\to\infty.
	\]
	Then any measurable maximizer $\hat P_V$ of
	\eqref{eq:mixing_npmle} satisfies
	\[
	W_1(\hat P_V,P_V)
	\pto0.
	\]
	Consequently, the estimator $\hat P_\lambda$ defined in
	\eqref{eq:lambda_npmle} satisfies
	\[
	W_2(\hat P_\lambda,P_\lambda)
	\pto0.
	\]
\end{thm}

Thus the NPMLE satisfies exactly the high-level condition required by
\Cref{prop:plugin_optimal_measure,thm:feasible_asymptotic_normality}.

Theorem~\ref{thm:radial_npmle_consistency} also addresses an important
feature of the data-adaptive procedure: the Mahalanobis radii are
computed using $\hat{\bmu}$ and $\hat{\bSigma}$ rather than the
unobserved population parameters. The estimated radii are dependent
through these common estimates, yet this plug-in dependence is
asymptotically negligible for recovery of the one-dimensional mixing
distribution. In particular, consistency does not require leave-one-out
estimation of the location and covariance parameters.

The optimization problem \eqref{eq:mixing_npmle} is infinite-dimensional.
Carath\'{e}odory's theorem implies that there exists a maximizer
$\hat P_V$ supported on at most $T+1$ atoms, and hence the corresponding
$\hat P_\lambda$ also has finite support. For computation, we approximate
the continuous-support NPMLE on a finite grid
\[
\mathcal V_T
=
\{v_1,\ldots,v_K\}
\subset[a_T,b_T].
\]
The resulting finite-dimensional problem is
\begin{align}
	\label{eq:grid_mixing_npmle}
	\max_{\pi_1,\ldots,\pi_K}
	&\quad
	\frac1T
	\sum_{i=1}^T
	\log\left(
	\sum_{j=1}^K
	\pi_j k_N(\hat D_i\mid v_j)
	\right)
	\nonumber\\
	\mbox{subject to}
	&\quad
	\pi_j\ge0,
	\qquad
	\sum_{j=1}^K\pi_j=1,
	\qquad
	\sum_{j=1}^K\pi_jv_j=1.
\end{align}
The objective is concave in the weights
$\{\pi_j\}_{j=1}^K$ and the constraints are linear, so
\eqref{eq:grid_mixing_npmle} is a finite-dimensional convex optimization
problem. Theorem~\ref{thm:radial_npmle_consistency} concerns the
continuous-support NPMLE in \eqref{eq:mixing_npmle}; in the numerical
implementation we use a sufficiently fine grid to approximate that
estimator.

\subsection{Implementation Summary}
\label{sec:adaptive_algorithm}

For clarity, the complete data-adaptive procedure can be summarized as
follows:
\begin{enumerate}
	\item Estimate the return mean and covariance matrix using
	$\hat{\bmu}$ and $\hat{\bSigma}$.
	
	\item Compute the estimated Mahalanobis radii
	$\hat D_1,\ldots,\hat D_T$ from
	\eqref{eq:estimated_squared_radius}.
	
	\item Estimate the radial distribution $P_\lambda$, for example by
	solving the constrained mixture problem
	\eqref{eq:grid_mixing_npmle} and applying the square-root
	transformation \eqref{eq:lambda_npmle}.
	
	\item Use $\hat P_\lambda$ to compute
	$\hat{\cT}$, $\hat{\cU}$, and the plug-in scalar $\hat c$, and form
	the estimated efficiency functional $\hat I(m;\hat c)$.
	
	\item Select the spectral measure
	\[
	\hat m
	\in
	\arg\min_{m\in\mathscr{P}([u,1-u])}
	\hat I(m;\hat c).
	\]
	
	\item Estimate the portfolio weights by solving
	\[
	\hat{\ww}_{\hat m}(\mu_0)
	=
	\arg\min_{\ww\in\R^N}
	\rho_{\hat m}(-\hat r_{\ww,T}),
	\qquad
	\mbox{subject to}
	\quad
	\ww^\top\hat{\bmu}=\mu_0,
	\quad
	\ww^\top\bone=1.
	\]
\end{enumerate}

Under the conditions above, Steps 1--5 learn the efficiency-optimal
estimating criterion without changing the first-order distribution of
the final portfolio estimator relative to the infeasible fixed-$u$
oracle rule. The next section examines the finite-sample behavior of
this procedure in simulations and empirical data.

\section{Numerical Experiments}
\label{sec:num_exp}

We use simulations and an empirical portfolio application to examine the main implications of the theory. The simulations address three questions. First, how accurately do the asymptotic covariance formulae approximate the finite-sample distribution of estimated portfolio weights? Second, does learning the spectral measure from the same observations used for portfolio estimation impose a detectable first-order cost relative to the oracle rule? Third, do the efficiency gains predicted under heavy-tailed returns translate into lower finite-sample weight estimation error? We then apply the data-adaptive procedure to U.S. industry portfolios and examine the stability and out-of-sample behavior of the resulting portfolio estimates.

\subsection{Simulation Evidence}
\label{sec:numerical}

We report the canonical risky-asset target-return problem in the main discussion. For completeness, the numerical tables and figures also report the corresponding risk-free-asset design, which provides additional evidence that the conclusions are not specific to the budget constraint.

We generate i.i.d. returns according to
\[
\rr_i
=
\bmu+\lambda_i\bSigma^{1/2}\zz_i,
\qquad
\zz_i\iidsim\normal(\bzero,\id_N),
\qquad
i=1,\ldots,T,
\]
where $\{\lambda_i\}_{i=1}^T$ are i.i.d. draws from $P_\lambda$ and are independent of $\{\zz_i\}_{i=1}^T$. We set $N=5$ and choose
\begin{equation}
	\label{eq:simulation_parameters}
	\begin{split}
		&\bSigma_{jk}=0.5^{|j-k|},
		\qquad
		j,k=1,\ldots,N,
		\\
		&\widetilde{\bmu}
		=
		(0.2,0.4,0.6,0.8,1)^\top,
		\qquad
		\bmu
		=
		\frac{\widetilde{\bmu}}
		{(\widetilde{\bmu}^\top\bSigma^{-1}\widetilde{\bmu})^{1/2}}.
	\end{split}
\end{equation}
This normalization gives
\[
\bmu^\top\bSigma^{-1}\bmu=1.
\]

For the canonical target-return problem, we choose $\mu_0^r\approx0.65$, which gives
\[
c(\mu_0^r)=\frac12.
\]
For the supplementary risk-free design, we set $\mu_0^f=1$, corresponding to the scalar index $c^f=1$.

We consider three radial distributions. For $Z_i\sim\normal(0,1)$ independent of $\lambda_i$,
\[
\lambda_i
=
\begin{cases}
	1,
	& \lambda_iZ_i\text{ is Gaussian},\\[2mm]
	\sqrt{E_i},\quad E_i\sim\operatorname{Exp}(1),
	& \lambda_iZ_i\text{ is Laplace},\\[2mm]
	2/\sqrt{S_i},\quad S_i\sim\chi_6^2,
	& \lambda_iZ_i\text{ is scaled Student }t_6.
\end{cases}
\]
The Laplace marginal has scale $1/\sqrt{2}$, while the Student marginal is distributed as $\sqrt{2/3}\,t_6$. In all three designs,
\[
\E[\lambda^2]=1,
\]
so $\bSigma$ is the covariance matrix. Moreover,
\[
\E[\lambda^4]=1
\]
under Gaussian returns and
\[
\E[\lambda^4]=2
\]
under both heavy-tailed designs.

\paragraph{Accuracy of the first-order approximation.}

We first examine the asymptotic covariance and Gaussian approximation. We use
\[
T\in\{500,1000,2000\}
\]
and $5000$ Monte Carlo replications for each combination of sample size and radial distribution. We set $u=0.01$ and approximate both the oracle and adaptive spectral measures on $100$ equally spaced points in $[u,1-u]$. The theoretical covariance is evaluated at the same discretized oracle measure used to compute the oracle portfolio weights.

For each simulated sample, we compute the oracle estimator
$\hat{\ww}_{m_{*,u}^c}(\mu_0^r)$ and the Adaptive SRM estimator
$\hat{\ww}_{\hat m}(\mu_0^r)$. We also retain the corresponding risk-free calculations as a supplementary comparison. The radial distribution is estimated using the discrete NPMLE in \Cref{sec:radial_npmle}. For $P_V$, we use a logarithmically spaced grid on
\[
[0.1T^{-1/3},5T^{1/3}],
\]
augmented by the point $1$, with $131$, $151$, and $181$ grid points for the three sample sizes, respectively. The same observations are used to estimate the radial distribution, select the spectral measure, and estimate the final portfolio weights, and the expected-return constraint uses $\hat{\bmu}$ for both the oracle and adaptive rules.

Let $\Sigma_*$ denote the theoretical asymptotic covariance for the corresponding design, and let $\hat\Sigma_T^{\mathrm O}$ and $\hat\Sigma_T^{\mathrm A}$ denote the Monte Carlo covariance matrices of the corresponding $\sqrt{T}$-scaled portfolio-weight errors. The superscripts $\mathrm O$ and $\mathrm A$ denote the Oracle SRM and Adaptive SRM rules, respectively. We report relative operator-norm errors
\[
\Delta_T^{\mathrm O}
=
\frac{
	\|\hat\Sigma_T^{\mathrm O}-\Sigma_*\|_{\op}
}{
	\|\Sigma_*\|_{\op}
},
\qquad
\Delta_T^{\mathrm A}
=
\frac{
	\|\hat\Sigma_T^{\mathrm A}-\Sigma_*\|_{\op}
}{
	\|\Sigma_*\|_{\op}
}.
\]

\begin{table}[htbp]
	\centering
	\caption{Relative operator-norm errors of Monte Carlo covariance estimates, based on $5000$ replications for each design and $u=0.01$. The superscripts $\mathrm O$ and $\mathrm A$ denote the Oracle SRM and Adaptive SRM rules, respectively.}
	\label{tab:simulation_covariance_error}
	\small
	\begin{tabular*}{0.95\textwidth}{@{\extracolsep{\fill}}llrrrr@{}}
		\toprule
		& & \multicolumn{2}{c}{Risky assets only}
		& \multicolumn{2}{c}{Risk-free asset available}\\
		\cmidrule(lr){3-4}\cmidrule(lr){5-6}
		Distribution & $T$
		& $\Delta_T^{\mathrm O}$ & $\Delta_T^{\mathrm A}$
		& $\Delta_T^{\mathrm O}$ & $\Delta_T^{\mathrm A}$\\
		\midrule
		Gaussian & $500$  & 0.0534 & 0.0539 & 0.0285 & 0.0287\\
		& $1000$ & 0.0422 & 0.0421 & 0.0422 & 0.0425\\
		& $2000$ & 0.0383 & 0.0383 & 0.0319 & 0.0321\\
		\addlinespace
		Laplace  & $500$  & 0.0257 & 0.0246 & 0.0187 & 0.0188\\
		& $1000$ & 0.0565 & 0.0551 & 0.0488 & 0.0490\\
		& $2000$ & 0.0331 & 0.0334 & 0.0281 & 0.0283\\
		\addlinespace
		Student $t_6$
		& $500$  & 0.0626 & 0.0586 & 0.0408 & 0.0390\\
		& $1000$ & 0.0316 & 0.0327 & 0.0338 & 0.0322\\
		& $2000$ & 0.0229 & 0.0224 & 0.0137 & 0.0136\\
		\bottomrule
	\end{tabular*}
\end{table}

Table~\ref{tab:simulation_covariance_error} shows close agreement between the simulated covariance matrices and their theoretical counterparts. At $T=2000$, all relative errors are below $4\%$. The errors need not decrease monotonically with $T$, since they include Monte Carlo error from estimating an entire covariance matrix using a finite number of replications.

More importantly for the oracle-adaptivity result in \Cref{thm:feasible_asymptotic_normality}, the Oracle SRM and Adaptive SRM covariance estimates are nearly indistinguishable. The relative operator-norm difference between their covariance matrices is below $0.7\%$ in every experiment. Thus, at these sample sizes, estimating the radial distribution and selecting the spectral measure from the same observations produces essentially no detectable first-order covariance penalty relative to the oracle rule.

\begin{figure}[!ht]
	\centering
	\includegraphics[width=0.95\textwidth]{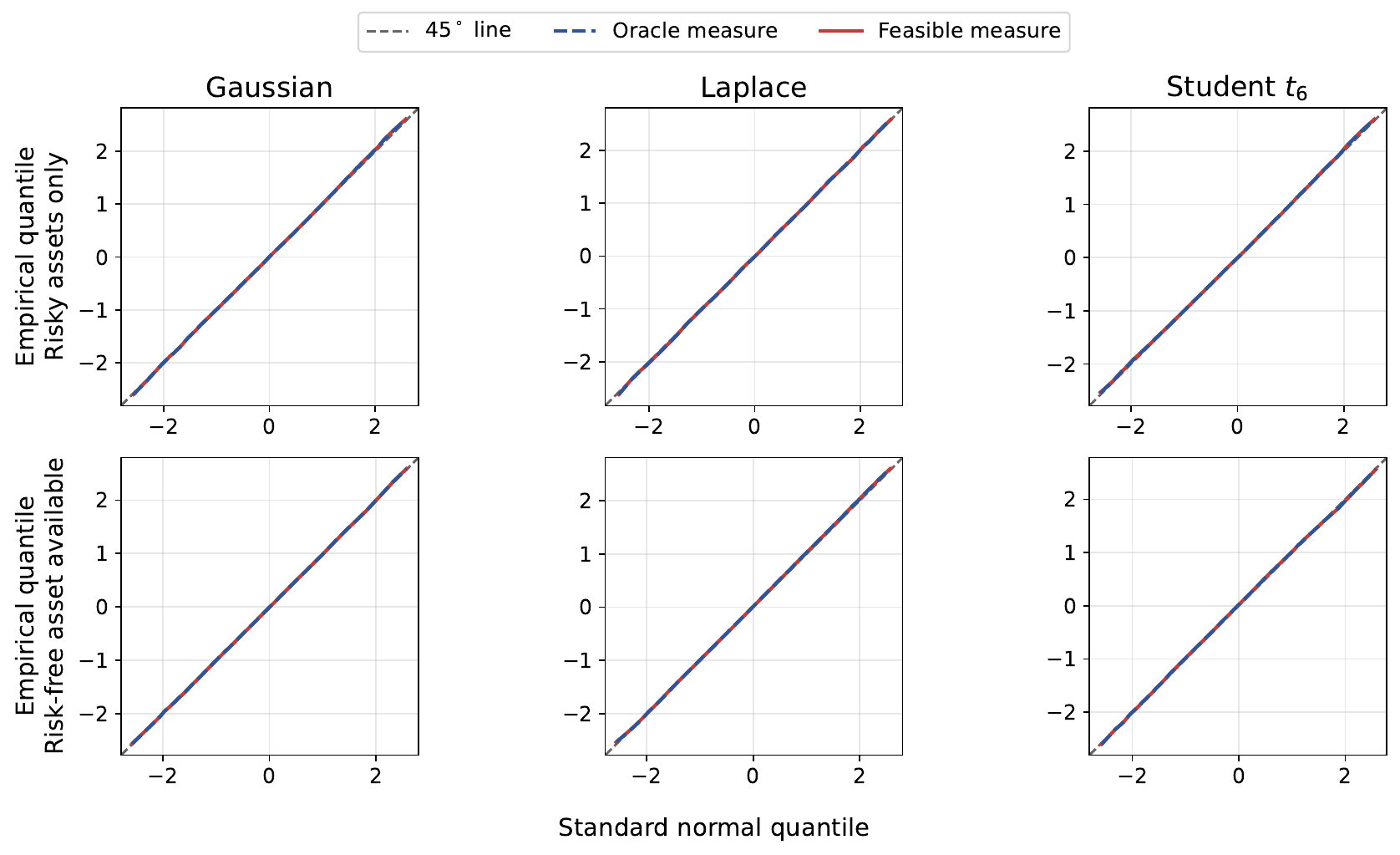}
	\caption{Normal Q-Q plots for the portfolio-weight errors at $T=2000$, based on $5000$ Monte Carlo replications and $u=0.01$. The $\sqrt{T}$-scaled errors are projected onto the positive-eigenvalue subspace of the theoretical covariance matrix and standardized by the square roots of the corresponding eigenvalues; the resulting coordinates are pooled within each panel. The blue dashed and red solid curves correspond to Oracle SRM and Adaptive SRM, respectively, and the gray dashed line is the $45^\circ$ reference line. The top row reports the canonical risky-asset problem and the bottom row the supplementary risk-free design.}
	\label{fig:simulation_normal_qq}
\end{figure}

Figure~\ref{fig:simulation_normal_qq} examines the Gaussian approximation at $T=2000$. In each design, the standardized empirical quantiles closely follow the normal reference line, while the Oracle SRM and Adaptive SRM curves nearly coincide. Together with Table~\ref{tab:simulation_covariance_error}, these results support both the asymptotic covariance formula and the oracle-equivalence conclusion of \Cref{thm:feasible_asymptotic_normality}.

\paragraph{Finite-sample efficiency.}

The preceding experiments validate the first-order distribution but do not directly measure the gain in portfolio-weight accuracy relative to MVO or fixed-level CVaR. We therefore compare sample MVO, $\operatorname{CVaR}_{0.5}$, $\operatorname{CVaR}_{0.9}$, Oracle SRM, and Adaptive SRM under the same three radial models. We use
\[
T\in\{100,250,500,1000\}
\]
and discretize the optimized spectral measures on $31$ equally spaced points in $[0.01,0.99]$.

For each method $a$, define relative weight mean-squared error by
\begin{equation}
	\label{eq:relative_weight_mse}
	R_T(a)
	=
	\frac{
		\E\|\hat{\ww}_a-\ww_*\|_2^2
	}{
		\E\|\hat{\ww}_{\mathrm{MVO}}-\ww_*\|_2^2
	}.
\end{equation}
Thus, $R_T(a)<1$ indicates lower estimation error than sample MVO for the common population portfolio.

The expectations for MVO, fixed CVaR, and Oracle SRM are estimated from $300$ paired Monte Carlo samples. Because re-estimating the radial distribution and re-optimizing the spectral measure in every sample is substantially more computationally intensive, Adaptive SRM uses the first $100$ paired samples; its MVO denominator in \eqref{eq:relative_weight_mse} is calculated from exactly those same samples.

\begin{figure}[!ht]
	\centering
	\includegraphics[width=0.98\textwidth]{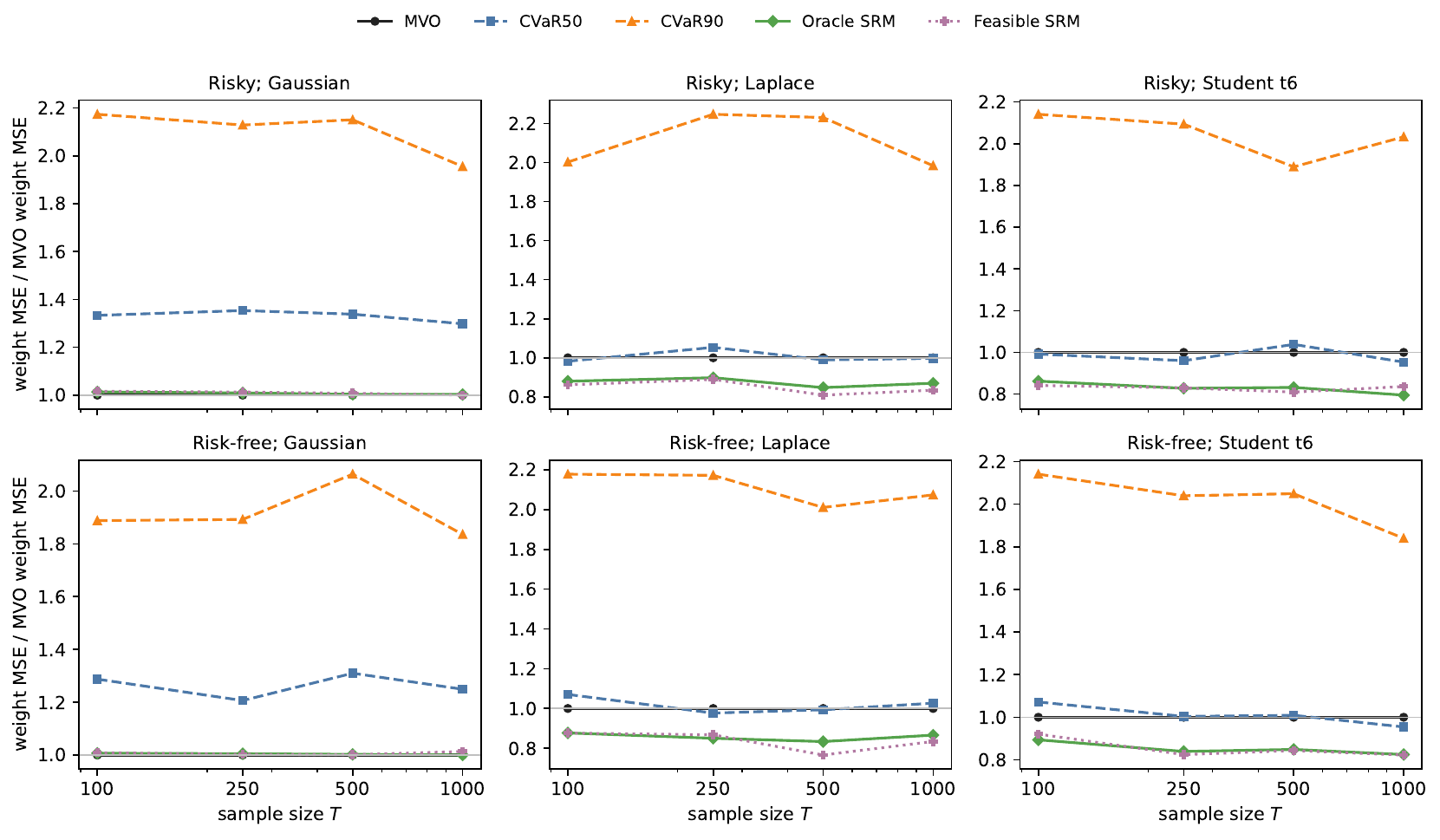}
	\caption{Finite-sample weight MSE relative to sample MVO. Values below one indicate lower weight-estimation error. The upper row reports the canonical risky-asset problem and the lower row the supplementary risk-free design. MVO, fixed-CVaR, and Oracle-SRM results use $300$ replications; Adaptive SRM uses $100$ paired replications.}
	\label{fig:simulation_finite_sample_efficiency}
\end{figure}

Figure~\ref{fig:simulation_finite_sample_efficiency} shows a pattern closely aligned with the efficiency analysis in \Cref{sec:optimal_srm}. Under Gaussian returns, Oracle SRM is essentially indistinguishable from MVO, while the fixed-CVaR rules are less efficient. Under Laplace and Student $t_6$ returns, the optimized spectral criteria reduce finite-sample weight MSE. At $T=500$, Oracle SRM reduces MSE by approximately $15$--$17\%$ in both portfolio designs, with Adaptive SRM delivering reductions of similar magnitude. In contrast, $\operatorname{CVaR}_{0.9}$ has roughly twice the MSE of MVO. The complete comparison at $T=500$ is reported in \Cref{tab:simulation_finite_sample_mse}.

Because the Adaptive SRM entries are based on only $100$ replications, small numerical differences between Oracle SRM and Adaptive SRM should not be interpreted as evidence that the adaptive rule is systematically more efficient than its oracle counterpart. The theoretical prediction is first-order equivalence; the observed differences are consistent with finite-sample and Monte Carlo variation.

\begin{table}[htbp]
	\centering
	\caption{Weight MSE relative to sample MVO at $T=500$. The MVO benchmark equals one.}
	\label{tab:simulation_finite_sample_mse}
	\small
	\begin{tabular*}{0.92\textwidth}{@{\extracolsep{\fill}}llrrrr@{}}
		\toprule
		Distribution & Problem
		& $\operatorname{CVaR}_{0.5}$
		& $\operatorname{CVaR}_{0.9}$
		& Oracle SRM
		& Adaptive SRM\\
		\midrule
		Gaussian & Risky only & 1.338 & 2.150 & 1.003 & 1.008\\
		& Risk-free  & 1.310 & 2.064 & 1.003 & 1.001\\
		\addlinespace
		Laplace  & Risky only & 0.989 & 2.231 & 0.847 & 0.809\\
		& Risk-free  & 0.994 & 2.011 & 0.834 & 0.766\\
		\addlinespace
		Student $t_6$
		& Risky only & 1.038 & 1.889 & 0.832 & 0.810\\
		& Risk-free  & 1.009 & 2.049 & 0.849 & 0.844\\
		\bottomrule
	\end{tabular*}
\end{table}

\paragraph{Implementation sensitivity.}

We examine sensitivity to two numerical choices in the spectral-measure optimization. Varying the support truncation over
\[
u\in\{0.01,0.025,0.05\}
\]
and the spectral grid over $21$, $41$, and $81$ points changes the optimized asymptotic ratio
\[
\frac{I(m_{*,u}^c;c)}{J(c)}
\]
by at most $0.0022$ within each heavy-tailed design. The ratio ranges from $0.841$ to $0.844$ under Laplace returns and from $0.799$ to $0.839$ under Student $t_6$ returns across the two portfolio designs; under Gaussian returns it remains within $0.014$ of one. Thus, the main efficiency comparisons are insensitive to the numerical discretization choices considered here.

\paragraph{Robustness to serial dependence.}

Finally, we introduce serial dependence through
\[
\rr_t-\bmu
=
0.30(\rr_{t-1}-\bmu)
+
\sqrt{1-0.30^2}\,
\lambda_t\bSigma^{1/2}\zz_t
\]
with Laplace radial innovations. This experiment lies outside the i.i.d. theory. Nevertheless, at $T=250$ and $500$, Oracle SRM has relative MSE between $0.905$ and $0.958$ across the two portfolio designs, while $\operatorname{CVaR}_{0.9}$ ranges from $1.77$ to $1.90$. We interpret this only as evidence that the efficiency pattern is not immediately destroyed by moderate dependence; formal inference for dependent returns would require an extension of the theory.

\subsection{Empirical Portfolio Performance}
\label{sec:empirical}

We next examine whether Adaptive SRM produces stable portfolio estimates and competitive out-of-sample performance in a long panel of real returns. The purpose of the exercise is not to test ellipticity as a literal model of monthly returns. Instead, we use the elliptical model as a working approximation for selecting the estimating criterion and then evaluate the resulting portfolios without imposing that model on their realized returns.

\paragraph{Data and portfolio design.}

We use value-weighted returns on the 30 U.S. industry portfolios from the Kenneth French Data Library \citep{french2026data}. The portfolios assign NYSE, AMEX, and NASDAQ stocks to industries using four-digit SIC codes and are reconstituted annually. Monthly observations are available from July 1926 through July 2026. We subtract the matched one-month Treasury bill return from each industry return and use a rolling window of 240 months for estimation. The out-of-sample evaluation begins in January 2000 and contains 319 monthly returns through July 2026.

Portfolios are re-estimated at the beginning of January, April, July, and October using only observations available before the rebalancing month. At each date, we set the target expected excess return $\mu_0$ equal to the cross-sectional average of the 30 estimated industry means. This makes the equal-weight portfolio feasible by construction, ensures that all optimized rules face the same attainable target in every estimation window, and avoids introducing an additional target-return tuning parameter into the comparison. Portfolio holdings are allowed to drift between quarterly rebalancing dates, and turnover is measured relative to the resulting pre-trade weights.

We compare six rules. The first is Adaptive SRM developed in \Cref{sec:feasible_rule}. The second is sample mean-variance optimization. The third replaces the sample covariance matrix by the linear shrinkage estimator of \cite{ledoit2004well}. The fourth and fifth minimize empirical CVaR at levels $0.50$ and $0.90$, respectively. The sixth is the equal-weight portfolio, a deliberately simple benchmark that is often difficult for estimated optimal portfolios to outperform \citep{demiguel2009optimal}. Short sales are allowed for the optimized rules, consistent with the equality-constrained theory.

\paragraph{Implementation of Adaptive SRM.}

Within each rolling window, we estimate $\bmu$ and $\bSigma$ by the sample mean and covariance matrix and form the squared Mahalanobis radii in \eqref{eq:estimated_squared_radius}. We approximate the constrained NPMLE \eqref{eq:grid_mixing_npmle} on a 55-point logarithmic grid augmented to contain $v=1$, imposing the constraint
\[
\int v\,\d P(v)=1
\]
exactly. We then calculate $\hat c$ and minimize $\hat I(m;\hat c)$ over probability masses on
\[
\{0.05,0.15,\ldots,0.95\}.
\]
Finally, we minimize the associated mixture of empirical CVaRs by linear programming. The same observations are used to estimate the radial distribution, select the spectral measure, and estimate the portfolio weights, as permitted by \Cref{thm:feasible_asymptotic_normality}.

All 107 rolling NPMLE and spectral-measure problems converged. Across rebalancing dates, $\hat c$ ranges from $-0.10$ to $0.95$, with median $0.52$, while the estimated $\E[\lambda^4]$ ranges from $1.19$ to $1.36$. In particular, the estimated fourth radial moment exceeds its Gaussian value of one in every estimation window.

\begin{figure}[!ht]
	\centering
	\begin{subfigure}[t]{0.49\textwidth}
		\centering
		\includegraphics[width=\linewidth]{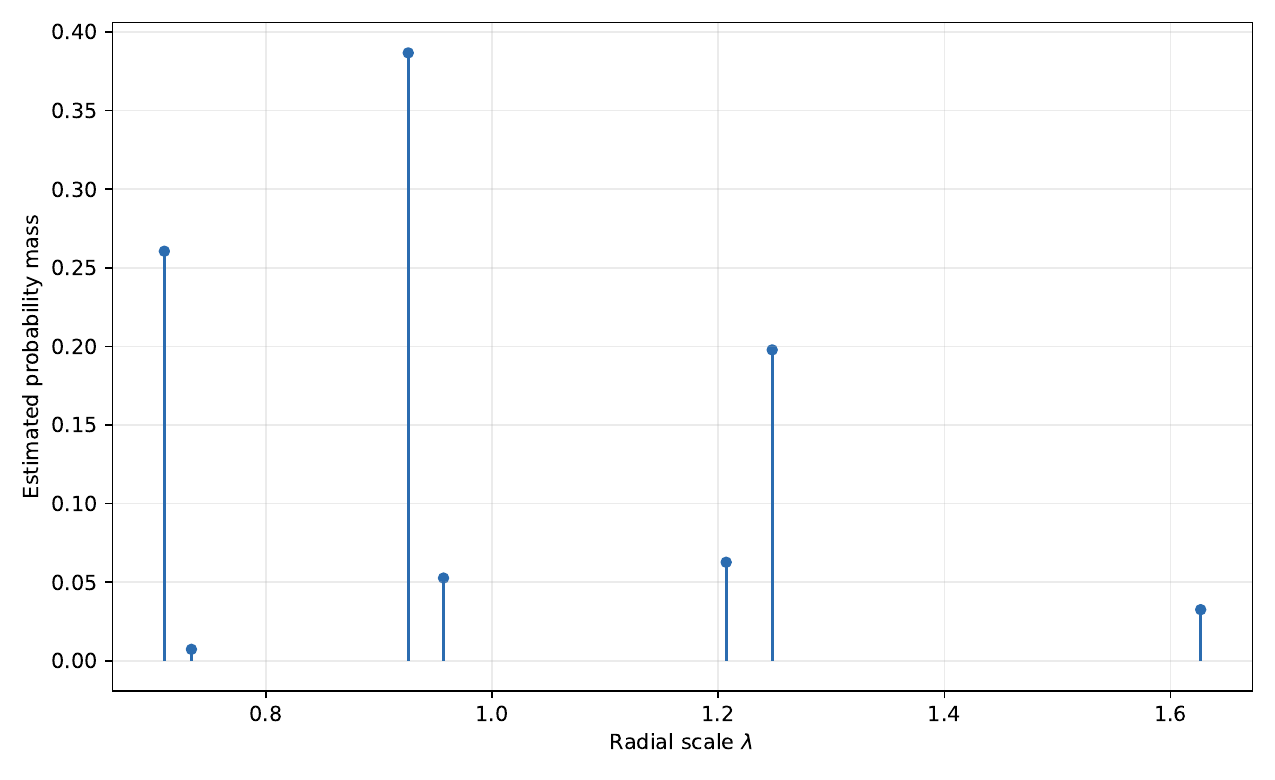}
		\caption{Estimated radial distribution.}
	\end{subfigure}
	\hfill
	\begin{subfigure}[t]{0.49\textwidth}
		\centering
		\includegraphics[width=\linewidth]{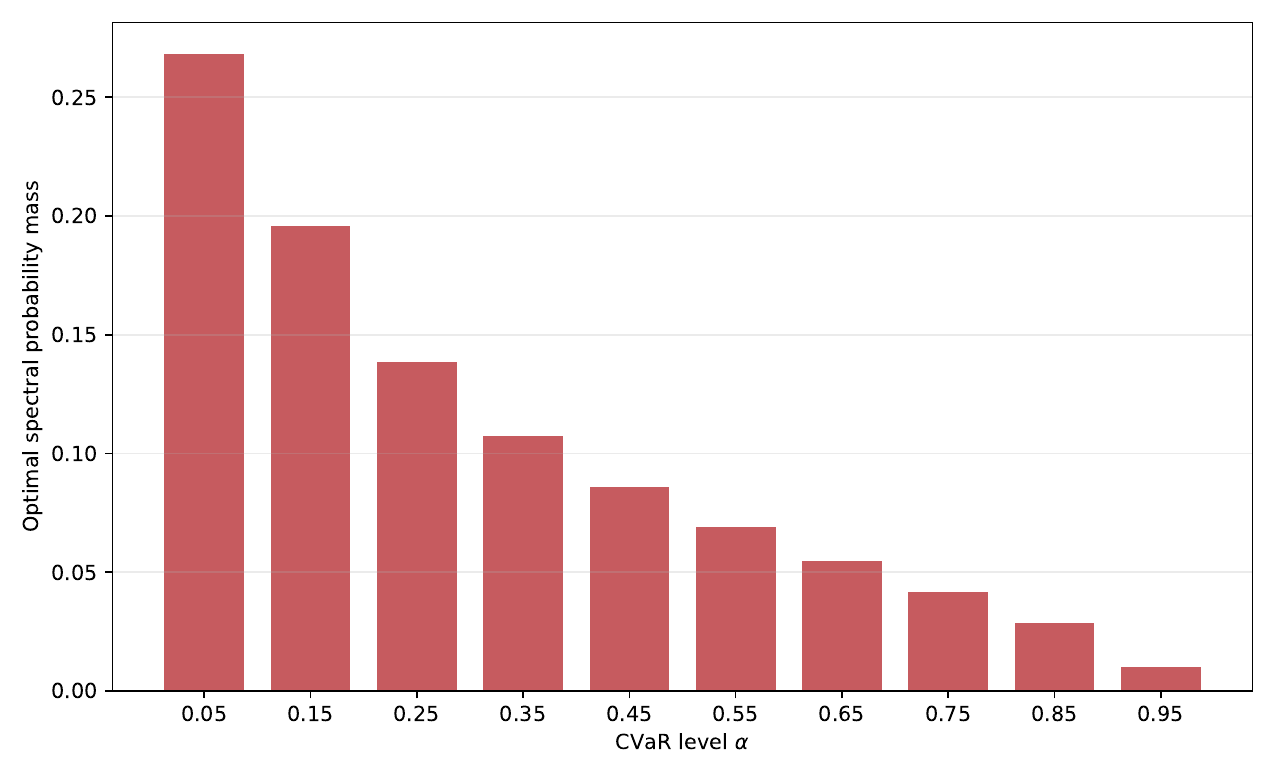}
		\caption{Estimated efficiency-optimal spectral measure.}
	\end{subfigure}
	\caption{Components of Adaptive SRM in the final estimation window. The left panel shows the non-negligible masses of the constrained NPMLE of $P_\lambda$. The right panel shows the resulting probability masses over CVaR levels. The selected criterion combines the full range of levels and places the largest weights on lower values of $\alpha$, rather than selecting a single tail probability.}
	\label{fig:empirical_fitted_components}
\end{figure}

\paragraph{Out-of-sample results.}

\Cref{tab:empirical_performance} reports annualized mean excess return and volatility, the annualized Sharpe ratio, monthly expected shortfall below the fifth percentile, maximum drawdown, annualized turnover, average gross exposure, and the average largest absolute portfolio weight. Mean, volatility, expected shortfall, and maximum drawdown are reported in percentage points. Transaction costs are not deducted, so turnover should be interpreted as an implementation and stability diagnostic rather than as a performance adjustment.

\begin{table}[htbp]
	\centering
	\caption{Out-of-sample performance, January 2000--July 2026.}
	\label{tab:empirical_performance}
	\resizebox{\textwidth}{!}{
		\begin{tabular}{lrrrrrrrr}
			\toprule
			& Mean & Vol. & Sharpe & ES$_{5\%}$ & Max DD
			& Turnover & Gross & Max $|w|$\\
			\midrule
			Equal weight    & 8.97 & 16.56 & 0.54 & 10.80 & 52.10 & 0.12 & 1.00 & 0.03\\
			Sample MVO      & 6.47 & 11.99 & 0.54 & 7.64  & 27.73 & 0.86 & 3.07 & 0.37\\
			Shrinkage MVO   & 6.55 & 11.63 & 0.56 & 7.52  & 27.56 & 0.68 & 2.56 & 0.31\\
			CVaR$_{0.50}$   & 6.75 & 12.31 & 0.55 & 7.29  & 31.44 & 1.59 & 3.09 & 0.39\\
			CVaR$_{0.90}$   & 4.58 & 13.56 & 0.34 & 8.32  & 35.45 & 2.12 & 4.39 & 0.42\\
			Adaptive SRM    & 7.00 & 12.14 & 0.58 & 7.61  & 26.94 & 0.94 & 3.01 & 0.37\\
			\bottomrule
		\end{tabular}
	}
\end{table}

The stability diagnostics are broadly favorable for Adaptive SRM. Its average gross exposure is slightly below that of sample MVO, while its turnover is only modestly higher. In contrast, the fixed-level CVaR rules are considerably more sensitive to the choice of tail probability: CVaR$_{0.50}$ has annual turnover almost twice that of sample MVO, while CVaR$_{0.90}$ has the highest turnover and gross exposure among the optimized rules.

Downside-risk measures tell a similar story. Adaptive SRM has a maximum drawdown of $26.94\%$, compared with $27.73\%$ for sample MVO, and monthly expected shortfall of $7.61\%$, compared with $7.64\%$ for MVO. CVaR$_{0.50}$ produces somewhat lower expected shortfall but a larger maximum drawdown and substantially higher turnover, while CVaR$_{0.90}$ performs less favorably on both risk and stability measures.

Turning to conventional return-performance measures, Adaptive SRM earns an annualized mean excess return of $7.00\%$ with volatility of $12.14\%$, corresponding to a Sharpe ratio of $0.58$. Sample MVO produces $6.47\%$, $11.99\%$, and $0.54$, respectively. Linear shrinkage reduces MVO volatility and turnover and attains a Sharpe ratio of $0.56$. These differences are modest, as would be expected given that the theoretical contribution concerns estimation efficiency rather than guaranteed dominance in realized returns.

The fixed-CVaR comparison illustrates the statistical importance of criterion selection. Although different CVaR levels identify the same population portfolio under the elliptical working model, their empirical minimizers can have substantially different sampling behavior. CVaR$_{0.90}$, in particular, has the lowest mean return and Sharpe ratio and the highest gross exposure and turnover among the optimized rules.

The comparison is reasonably stable across broad subsamples. Adaptive SRM's Sharpe ratios are $0.22$, $1.11$, and $0.50$ during 2000--2009, 2010--2019, and 2020--2026, respectively, compared with $0.15$, $1.08$, and $0.48$ for sample MVO. A paired circular block bootstrap with 12-month blocks gives an Adaptive-SRM-minus-MVO Sharpe-ratio difference of $0.037$, with a $95\%$ percentile interval of $[-0.002,0.078]$. The corresponding volatility difference is $0.15$ percentage points, with interval $[-0.01,0.30]$, and the monthly expected-shortfall difference is $-0.03$ percentage points, with interval $[-0.31,0.25]$.

The data therefore do not establish statistically decisive out-of-sample dominance. Rather, they show that Adaptive SRM remains competitive with sample and shrinkage MVO while avoiding the instability associated with an arbitrarily chosen high CVaR level. Industry portfolios are diversified test assets rather than directly traded securities, the exercise ignores transaction costs, and the rolling return distribution is neither independent nor exactly elliptical. Moreover, the theoretical ordering concerns first-order covariance of estimated weights, whereas realized returns, Sharpe ratios, and drawdowns also reflect model misspecification and time variation. We therefore interpret the empirical results as evidence that the adaptive procedure is operational and empirically competitive, rather than as a universal dominance claim.

Taken together, the numerical results support three implications of the theory. First, the first-order covariance and Gaussian approximations are accurate at moderate sample sizes. Second, learning the radial distribution and selecting the spectral criterion imposes little detectable cost relative to the oracle rule, consistent with oracle adaptivity. Third, criterion selection matters materially for weight estimation under the heavy-tailed designs considered here, while providing little first-order gain over MVO under Gaussianity. The empirical application further shows that the resulting adaptive rule can be implemented in rolling portfolio problems and remains competitive with conventional portfolio estimators.

\section{Conclusion}
\label{sec:conclusion}

This paper studies a statistical question that arises when a family of
empirical criteria shares a common population minimizer: how should the
criterion itself be chosen to estimate that common target efficiently,
and can the efficient criterion be learned from the data without
sacrificing first-order efficiency? Spectral-risk portfolio estimation
provides a natural setting for these questions. Under elliptically
contoured returns, spectral risk measures identify the same population
efficient portfolio for a fixed target return, while their empirical
minimizers can have substantially different sampling distributions.
The spectral measure can therefore be viewed not only as part of a risk
criterion, but also as an object that can be selected for statistical
efficiency.

We first develop a general asymptotic theory for empirical spectral-risk
minimization under estimated linear constraints. The theory accommodates
the nonsmooth empirical CVaR representation, mixtures across CVaR levels,
and sampling error in the feasible set, and allows weakly binding
inequalities to generate non-Gaussian stochastic quadratic-program
limits. For the canonical target-return problem under elliptical returns,
the resulting covariance comparison has a particularly simple structure:
\[
\Sigma_{\mathrm{SRM}}(m)
=
\Sigma_{\mathrm{common}}
+
I(m;c)\Sigma_{\perp},
\qquad
\Sigma_{\perp}\succeq0.
\]
The population portfolio problem enters the criterion-dependent component
only through the scalar index $c$, while all dependence on the spectral
measure is summarized by $I(m;c)$. Choosing an efficient empirical
criterion therefore reduces to an optimization problem over probability
measures.

For each fixed endpoint restriction $u$, we characterize the unique
efficiency-optimal spectral measure $m_{*,u}^c$. The optimizer has no
singular component in the interior of the admissible interval, so no
single interior CVaR level attains the minimum first-order covariance.
As the endpoint restriction is relaxed, the best covariance bound over
the spectral-risk class is no larger than the sample mean-variance
benchmark. Under Gaussian returns, the optimized spectral criterion
approaches the mean-variance efficiency benchmark; under heavy-tailed
models such as the Laplace example, spectral criterion selection can
strictly reduce the sampling variability of estimated portfolio weights.
Thus, even when mean-variance optimization identifies the correct
population portfolio, its usual sample implementation need not be the
most efficient estimator of that portfolio.

The efficiency-optimal criterion depends on the unknown return
distribution, so an oracle characterization alone would not provide an
implementable procedure. We therefore develop a fully data-adaptive rule
that estimates the radial distribution, learns the efficiency-optimal
spectral measure, and estimates the portfolio weights using the same
observations. Despite this additional layer of adaptation,
\[
\sqrt{T}
\left(
\hat{\ww}_{\hat m}(\mu_0)-\ww_*(\mu_0)
\right)
\]
has the same first-order limiting distribution as the infeasible estimator
that knows $m_{*,u}^c$ in advance. Hence, learning both the
infinite-dimensional nuisance distribution and the estimating criterion
is first-order costless for estimation of the portfolio weights; no sample
splitting is required. The constrained radial NPMLE developed in the paper
provides one concrete implementation satisfying the conditions needed for
this oracle-adaptivity result.

The numerical evidence reflects these theoretical distinctions. The
first-order covariance approximation is accurate at moderate sample
sizes, and the sampling distributions of the oracle and data-adaptive
rules are nearly indistinguishable in the simulations. Optimizing the
spectral criterion offers little gain over sample mean-variance
optimization under Gaussian returns, but materially reduces weight
estimation error in the heavy-tailed designs considered here. In the
empirical application to U.S. industry portfolios, the adaptive rule is
stable and competitive with conventional portfolio estimators, while the
performance of fixed-level CVaR rules is substantially more sensitive to
the chosen confidence level. These empirical results are best interpreted
as evidence for the usefulness of data-adaptive criterion selection rather
than as a claim of universal investment-performance dominance.

Several directions remain open. The present theory keeps the number of
assets fixed while the sample size diverges. When the number of assets is
large relative to the available return history, estimation of the return
distribution, the constraints, and the criterion itself all interact with
high-dimensional regularization, and it is not clear whether the
fixed-dimensional efficiency ordering persists. Outside the elliptical
setting, changing the spectral measure can also change the population
portfolio, so criterion selection must balance the investment objective
against statistical precision rather than exploiting a common estimand.
Extensions to dependent returns and a more systematic finite-sample theory
for data-adaptive criterion selection are additional directions for future
work. More broadly, the results suggest that whenever multiple empirical
criteria identify a common population target, selecting and learning the
criterion itself can provide a useful additional dimension for improving
statistical efficiency.

\bibliographystyle{plainnat}
\bibliography{refs.bib}
	
\newpage

\appendix

\section{Portfolios with a Risk-Free Asset}
\label{app:risk_free}

This appendix shows that the main asymptotic, efficiency, and
data-adaptivity results continue to hold when investors may borrow and
lend freely at the risk-free rate. Because returns are measured in excess
of the risk-free rate, the budget constraint on risky-asset weights is
removed and only the target expected excess-return constraint remains.
The resulting calculations parallel those in the main text, so we collect
them here.

\subsection{General Asymptotic Theory}
\label{app:risk_free_general}

The population problem is
\[
\ww_*^f(\mu_0)
=
\arg\min_{\ww\in\R^N}
\rho_m(-r_{\ww}),
\qquad
\mbox{subject to}
\quad
\ww^\top\bmu=\mu_0,
\]
with empirical counterpart
\[
\hw^f(\mu_0)
=
\arg\min_{\ww\in\R^N}
\rho_m(-\hat r_{\ww,T}),
\qquad
\mbox{subject to}
\quad
\ww^\top\hat{\bmu}=\mu_0.
\]
We assume $\mu_0\neq0$ to exclude the trivial solution
$\ww_*^f=\bzero$.

Write $\ww_*^f=\ww_*^f(\mu_0)$. Let $\eta_s^f\in\R$ denote the
Lagrange multiplier associated with the expected-return constraint, so
that
\[
\nabla_{\ww}\rho_m(-r_{\ww_*^f})
+
\eta_s^f\bmu
=
0.
\]
Define
\[
\begin{split}
	H_*^f
	&=
	\nabla_{\ww}^2\rho_m(-r_{\ww_*^f}),
	\\
	V_*^f(\rr)
	&=
	\nabla_{\ww}U_m(\ww_*^f;\rr)
	+
	\eta_s^f\rr,
	\\
	r_*^f(\rr)
	&=
	(\ww_*^f)^\top\rr,
	\\
	K_*^f
	&=
	\begin{bmatrix}
		H_*^f & \bmu\\
		\bmu^\top & 0
	\end{bmatrix}.
\end{split}
\]

This problem is obtained from
Theorem~\ref{thm:normality_general_risk} by retaining only the estimated
expected-return equality constraint. Since there are no inequality
constraints, the limiting stochastic quadratic program reduces to an
equality-constrained Gaussian problem.

\begin{cor}[Portfolio with a risk-free asset]
	\label{prop:normality_risk_free_constraint}
	Suppose $\ww_*^f$ is unique, Assumptions~\ref{ass:regularity_density} and \ref{ass:measure_support} hold, the population objective is level-bounded on its constraint set, and $K_*^f$ is
	nonsingular. Let $\hw^f(\mu_0)$ be any measurable solution of the
	empirical problem. Then
	\[
	\hw^f(\mu_0)
	\pto
	\ww_*^f(\mu_0).
	\]
	If, in addition,
	\[
	\sup_{(\ww,v)\in K\times\R}
	p_{-r_{\ww}}(v)<\infty
	\]
	for every compact set
	$K\subset\R^N\backslash\{\bzero\}$, then
	\[
	\sqrt{T}
	\left(
	\hw^f(\mu_0)-\ww_*^f(\mu_0)
	\right)
	\wto
	\normal\left(
	0,\Sigma^f(\mu_0,\rho_m)
	\right),
	\]
	where
	\[
	\Sigma^f(\mu_0,\rho_m)
	=
	\var\left(
	\begin{bmatrix}
		\id_N & \bzero
	\end{bmatrix}
	(K_*^f)^{-1}
	\begin{bmatrix}
		V_*^f(\rr)\\
		r_*^f(\rr)
	\end{bmatrix}
	\right).
	\]
\end{cor}

\subsection{Elliptical Returns and Covariance Decomposition}
\label{app:risk_free_elliptic}

We next specialize the preceding result to the elliptical model of
\Cref{sec:elliptic}. The statistical structure is the same as in the
canonical target-return problem: the spectral-risk and sample
mean-variance covariance matrices share a common component and differ
only through the scalar coefficients $I(m;c)$ and $J(c)$.

Define
\[
c^f
:=
\frac{\sgn(\mu_0)}
{\|\bSigma^{-1/2}\bmu\|_2},
\]
and
\[
P_\perp^f
:=
\bSigma^{-1}
-
\frac{
	\bSigma^{-1}\bmu\bmu^\top\bSigma^{-1}
}{
	\bmu^\top\bSigma^{-1}\bmu
}.
\]
The matrix $P_\perp^f$ is positive semidefinite and describes the
covariance geometry in directions orthogonal, in the
$\bSigma$-induced geometry, to the expected-return constraint.

\begin{prop}
	\label{prop:risk_free_lowdim}
	Suppose the conditions of
	\Cref{prop:normality_risk_free_constraint} hold and
	\Cref{ass:elliptic_return} holds. Assume additionally that
	$\mu_0\neq0$ and $\bmu\neq\bzero$. Then
	\begin{equation}
		\label{eq:w_star_risk_free}
		\ww_*^f(\mu_0)
		=
		\frac{\mu_0}
		{\bmu^\top\bSigma^{-1}\bmu}
		\bSigma^{-1}\bmu.
	\end{equation}
	Moreover,
	\[
	\hw^f(\mu_0)
	\xrightarrow{p}
	\ww_*^f(\mu_0),
	\]
	and
	\[
	\sqrt{T}
	\left(
	\hw^f(\mu_0)-\ww_*^f(\mu_0)
	\right)
	\xrightarrow{d}
	\normal\left(
	0,\Sigma^f(\mu_0,\rho_m)
	\right),
	\]
	where
	\[
	\begin{split}
		\Sigma^f(\mu_0,\rho_m)
		=
		&\,
		\frac{
			\E[\lambda^2]\mu_0^2
		}{
			(\bmu^\top\bSigma^{-1}\bmu)^3
		}
		\bSigma^{-1}\bmu\bmu^\top\bSigma^{-1}
		\\
		&+
		\frac{
			\mu_0^2
		}{
			\rho_m(\lambda Z)^2
			(\bmu^\top\bSigma^{-1}\bmu)
		}
		\left(
		\int_{[0,1]^2}
		\frac{
			A^f(\alpha_1,\alpha_2)
		}{
			(1-\alpha_1)(1-\alpha_2)
		}
		\d m(\alpha_1)\d m(\alpha_2)
		\right)
		P_\perp^f,
	\end{split}
	\]
	with
	\[
	\begin{split}
		A^f(\alpha_1,\alpha_2)
		=
		&\,
		\E\left[
		\lambda^2
		\Phi\left(
		-\frac{1}{\lambda}
		q_{\alpha_1\vee\alpha_2}(\lambda Z)
		\right)
		\right]
		\\
		&-
		(1-\alpha_1)\eta_s^f
		\E\left[
		\lambda^2
		\Phi\left(
		-\frac{1}{\lambda}
		q_{\alpha_2}(\lambda Z)
		\right)
		\right]
		\\
		&-
		(1-\alpha_2)\eta_s^f
		\E\left[
		\lambda^2
		\Phi\left(
		-\frac{1}{\lambda}
		q_{\alpha_1}(\lambda Z)
		\right)
		\right]
		\\
		&+
		(1-\alpha_1)(1-\alpha_2)
		(\eta_s^f)^2
		\E[\lambda^2],
	\end{split}
	\]
	where
	\[
	\eta_s^f
	=
	1-c^f\rho_m(\lambda Z).
	\]
	Here $Z\sim\normal(0,1)$ is independent of $\lambda$.
\end{prop}

\paragraph{Comparison with sample mean-variance optimization.}

The sample mean-variance estimator is
\[
\hw_{\rm MV}^f(\mu_0)
=
\frac{\mu_0}
{\hat{\bmu}^\top
	\hat{\bS}^{-1}
	\hat{\bmu}}
\hat{\bS}^{-1}\hat{\bmu}.
\]

\begin{prop}
	\label{prop:asymptotic_variance_mv_risk_free}
	Suppose the conditions of
	\Cref{prop:risk_free_lowdim} hold and
	$\E[\lambda^4]<\infty$. Then
	\[
	\hw_{\rm MV}^f(\mu_0)
	\xrightarrow{p}
	\ww_*^f(\mu_0),
	\]
	and
	\[
	\sqrt{T}
	\left(
	\hw_{\rm MV}^f(\mu_0)-\ww_*^f(\mu_0)
	\right)
	\xrightarrow{d}
	\normal\left(
	0,\Sigma_{\rm MV}^f(\mu_0)
	\right),
	\]
	where
	\[
	\begin{split}
		\Sigma_{\rm MV}^f(\mu_0)
		=
		&\,
		\frac{
			\E[\lambda^2]\mu_0^2
		}{
			(\bmu^\top\bSigma^{-1}\bmu)^3
		}
		\bSigma^{-1}\bmu\bmu^\top\bSigma^{-1}
		\\
		&+
		\frac{\mu_0^2}
		{\bmu^\top\bSigma^{-1}\bmu}
		\left(
		\frac{\E[\lambda^4]}
		{\E[\lambda^2]^2}
		+
		\frac{\E[\lambda^2]}
		{\bmu^\top\bSigma^{-1}\bmu}
		\right)
		P_\perp^f.
	\end{split}
	\]
\end{prop}

The two covariance matrices therefore have exactly the same scalar
reduction as in the main text.

\begin{cor}[Scalar efficiency reduction with a risk-free asset]
	\label{cor:compare_variance_risk_free}
	Under the conditions of the preceding two propositions, define
	\[
	\Sigma_{\mathrm{common}}^f(\mu_0)
	:=
	\frac{
		\E[\lambda^2]\mu_0^2
	}{
		(\bmu^\top\bSigma^{-1}\bmu)^3
	}
	\bSigma^{-1}\bmu\bmu^\top\bSigma^{-1}
	\]
	and
	\[
	\Sigma_\perp^f(\mu_0)
	:=
	\frac{\mu_0^2}
	{\bmu^\top\bSigma^{-1}\bmu}
	P_\perp^f.
	\]
	Then
	\[
	\Sigma^f (\mu_0,\rho_m) := \Sigma_{\rm SRM}^f(m)
	=
	\Sigma_{\mathrm{common}}^f(\mu_0)
	+
	I(m;c^f)\Sigma_\perp^f(\mu_0),
	\]
	whereas
	\[
	  \Sigma_{\rm MV}^f (\mu_0) := \Sigma_{\rm MV}^f
	=
	\Sigma_{\mathrm{common}}^f(\mu_0)
	+
	J(c^f)\Sigma_\perp^f(\mu_0).
	\]
	Consequently,
	\[
	I(m_1;c^f)\le I(m_2;c^f)
	\quad\Longrightarrow\quad
	\Sigma_{\rm SRM}^f(m_1)
	\preceq
	\Sigma_{\rm SRM}^f(m_2).
	\]
	Thus, the risk-free formulation leads to the same scalar
	criterion-selection problem as the canonical target-return formulation:
	\[
	\min_m I(m;c^f).
	\]
\end{cor}

The risk-free extension therefore changes the portfolio geometry, and
hence the scalar index entering the efficiency functional, but not the
form of the criterion-selection problem. Under the normalization
$\E[\lambda^2]=1$ used for the data-adaptive procedure and for a positive
target return,
\[
c^f
=
\frac{1}
{\|\bSigma^{-1/2}\bmu\|_2}
\]
is the reciprocal of the maximum population Sharpe ratio.

\subsection{Data-Adaptive Extension}
\label{app:risk_free_adaptive}

The data-adaptive construction of \Cref{sec:feasible_rule} extends
directly to the risk-free formulation. Let $m_{*,u}^{c^f}$ denote the
fixed-$u$ efficiency-optimal spectral measure corresponding to $c^f$.
Let $\hat c^f$ denote the plug-in estimator of $c^f$, and define
$\hat m^f$ by minimizing the corresponding estimated efficiency
functional over $\mathscr{P}([u,1-u])$.

\begin{cor}[Oracle adaptivity with a risk-free asset]
	\label{cor:risk_free_oracle_adaptivity}
	Suppose the conditions of
	\Cref{prop:normality_risk_free_constraint} hold, suppose
	\Cref{ass:elliptic_return} holds with $\lambda>0$ almost surely and
	$\E[\lambda^2]=1$, and suppose
	\[
	W_2(\hat P_\lambda,P_\lambda)\pto0.
	\]
	Then
	\[
	\hat{\ww}_{\hat m^f}^f(\mu_0)
	\pto
	\ww_*^f(\mu_0),
	\]
	and
	\[
	\sqrt{T}
	\left(
	\hat{\ww}_{\hat m^f}^f(\mu_0)
	-
	\ww_*^f(\mu_0)
	\right)
	\wto
	\normal\left(
	0,
	\Sigma^f\left(
	\mu_0,
	\rho_{m_{*,u}^{c^f}}
	\right)
	\right).
	\]
\end{cor}

Thus, estimating the radial distribution and learning the
efficiency-optimal spectral measure from the same observations is also
first-order costless in the risk-free formulation. The adaptive estimator
attains the same first-order covariance as the infeasible fixed-$u$
oracle rule without sample splitting.

\section{Proofs for Section~\ref{sec:low-dim}}\label{sec:proofs_sec3}

\subsection{Proof of \cref{thm:population_derivatives}}
\begin{proof}
	By definition, we have
	\begin{equation*}
		R_{\alpha} (\ww, z) = \, \E [L_{\alpha} (\ww, z; \rr)] = \E \left[ z + \frac{1}{1 - \alpha} ( - \ww^\top \rr - z )_+ \right].
	\end{equation*}
	It is straightforward to see that \cref{eq:pop_gradient} directly follows from \cref{ass:regularity_density} (i) and (ii), Fubini's theorem and Lebesgue's differentiation theorem. We next prove \cref{eq:pop_hessian}, using \cref{lem:general_diff} from \cref{sec:auxiliary_sec3}.
	
	To prove \cref{eq:pop_hessian}, we first note that \cref{ass:regularity_density} (i) directly implies $\partial_z^2 R_{\alpha} (\ww, z) = p_{- r_{\ww}} (z) / (1 - \alpha)$. For the other statements, we apply \cref{lem:general_diff} for $g(\rr) = 1$ or $\rr$, and $h(\rr, \ww) = - r_{\ww} - z$. It is straightforward to verify that the conditions of this lemma are satisfied under \cref{ass:regularity_density}. We thus obtain that
	\begin{equation*}
		\begin{split}
			\nabla_{\ww}^2 R_{\alpha} (\ww, z) = \, & \frac{1}{1 - \alpha} p_{-r_{\ww}-z} (0) \E \left[ \rr \rr^\top \big\vert - r_{\ww} - z = 0 \right], \\
			\nabla_{\ww} \partial_z R_{\alpha} (\ww, z) = \, & \frac{1}{1 - \alpha} p_{-r_{\ww} - z} (0) \E \left[ \rr \vert - r_{\ww} - z = 0 \right],
		\end{split}
	\end{equation*}
	completing the proof.
\end{proof}

\subsection{Proof of \cref{thm:direct_differentiate_cvar}}
\begin{proof}
	By definition and \cref{thm:var_rep}, we know that
	\begin{equation*}
		z (\ww, \alpha) = \arg \min_{z \in \R} R_{\alpha} (\ww, z), \quad \operatorname{CVaR}_{\alpha} (-r_{\ww}) = R_{\alpha} (\ww, z (\ww, \alpha) ).
	\end{equation*}
	By the envelope theorem and \cref{thm:population_derivatives},
	\begin{align*}
		\nabla_{\ww} \operatorname{CVaR}_{\alpha} (-r_{\ww}) = \, \nabla_{\ww} R_{\alpha} (\ww, z (\ww, \alpha) ) = - \frac{1}{1 - \alpha} \E \left[ \rr \cdot \bone \left\{-r_{\ww} \ge z (\ww, \alpha) \right\} \right].
	\end{align*}
	Further, using the implicit function theorem, we get
	\begin{equation*}
		\nabla_{\ww} z (\ww, \alpha) = \, - \left( \partial_z^2 R_{\alpha} (\ww, z(\ww, \alpha)) \right)^{-1} \nabla_{\ww} \partial_z R_{\alpha} (\ww, z (\ww, \alpha)).
	\end{equation*}
	Therefore,
	\begin{align*}
		& \nabla_{\ww}^2 \operatorname{CVaR}_{\alpha} (-r_{\ww}) \\
		= \, & \nabla_{\ww}^2 R_{\alpha} (\ww, z (\ww, \alpha)) + \partial_z \nabla_{\ww} R_{\alpha} (\ww, z (\ww, \alpha)) \nabla_{\ww} z (\ww, \alpha)^\top \\
		= \, & \nabla_{\ww}^2 R_{\alpha} (\ww, z (\ww, \alpha)) - \left( \partial_z^2 R_{\alpha} (\ww, z(\ww, \alpha)) \right)^{-1} \nabla_{\ww} \partial_z R_{\alpha} (\ww, z (\ww, \alpha)) \nabla_{\ww} \partial_z R_{\alpha} (\ww, z (\ww, \alpha))^\top \\
		\stackrel{(i)}{=} \, & \frac{1}{1 - \alpha} p_{-r_{\ww}} (z (\ww, \alpha)) \operatorname{Cov} \left( \rr \big\vert -r_{\ww} = z (\ww, \alpha) \right),
	\end{align*}
	where $(i)$ follows from \cref{thm:population_derivatives}.
	This completes the proof.
\end{proof}

\subsection{Proof of \cref{thm:normality_general_risk}}
\begin{proof}
	We first establish the consistency of $\hw$, by using a proof strategy similar to that of Corollary 3.2.3 (ii) of \cite{wellner2013weak}. We begin with proving that the empirical risk $\rho_m (- \hat{r}_{\ww, T})$ uniformly converges to the population risk $\rho_m (-r_{\ww})$ over any compact set $K$ in probability. To prove this claim, we define $\hat{z}_T (\ww, \alpha)$ as the $\alpha$-quantile of $- \hat{r}_{\ww, T}$. Using Markov's inequality, for any $\alpha \in \supp(m) \subset [u, 1-u]$, we deduce that
	\begin{equation*}
		\vert \hat{z}_T (\ww, \alpha) \vert \le \, \frac{1}{T \cdot \min(\alpha, 1 - \alpha)} \sum_{i=1}^{T} \vert \ww^\top \rr_i \vert \le \frac{\norm{\ww}_2}{Tu} \sum_{i=1}^{T} \norm{\rr_i}_2.
	\end{equation*}
	By the law of large numbers, $(1/T) \sum_{i=1}^{T} \norm{\rr_i}_2 \to \E [\norm{\rr}_2]$ in probability, which means that
	\begin{equation}\label{eq:quantile_uniform_bounded}
		\sup_{(\ww, \alpha) \in K \times \supp(m)} \vert \hat{z}_T (\ww, \alpha) \vert = \, O_p (1)
	\end{equation}
	for any compact set $K \subset \R^N$. Further, for $(\ww, z) \in \R^N \times \R$, define
	\begin{equation*}
		\hat{R}_{\alpha} (\ww, z) = \, \hat{\E}_T [L_{\alpha} (\ww, z; \rr)] = \frac{1}{T} \sum_{i=1}^{T} L_{\alpha} (\ww, z; \rr_i).
	\end{equation*}
	By \cref{thm:var_rep}, we know that
	\begin{equation*}
		\operatorname{CVaR}_{\alpha} (- \hat{r}_{\ww, T}) = \, \hat{R}_{\alpha} \big( \ww, \hat{z}_T (\ww, \alpha) \big) = \min_{z \in \R} \hat{R}_{\alpha} \big( \ww, z \big).
	\end{equation*}
	\cref{eq:quantile_uniform_bounded} then implies that with high probability,
	\begin{equation*}
		\operatorname{CVaR}_{\alpha} (- \hat{r}_{\ww, T}) = \min_{z \in [-M, M]} \hat{R}_{\alpha} ( \ww, z )
	\end{equation*}
	for all $(\ww, \alpha) \in K \times \supp(m)$ and some constant $M > 0$. Recalling that $z(\ww, \alpha)$ is the $\alpha$-quantile of $-r_{\ww} = - \ww^\top \rr$, we similarly deduce that
	\begin{equation*}
		\operatorname{CVaR}_{\alpha} (- r_{\ww}) = \min_{z \in [-M, M]} R_{\alpha} ( \ww, z )
	\end{equation*}
	for all $(\ww, \alpha) \in K \times \supp(m)$. Now since $L_{\alpha}$ is a Lipschitz function of $r_{\ww} = \ww^\top \rr$ and $z$, we can apply \cite[Theorem 3]{van2000preservation} to conclude that the function class
	\begin{equation*}
		\left\{ L_{\alpha} (\ww, z; \rr) \vert (\ww, z, \alpha) \in K \times [-M, M] \times \supp(m) \right\}
	\end{equation*}
	is $P$-Glivenko-Cantelli. Therefore,
	\begin{equation*}
		\sup_{(\ww, z, \alpha) \in K \times [-M, M] \times \supp(m)} \left\vert \hat{R}_{\alpha} ( \ww, z ) - R_{\alpha} ( \ww, z ) \right\vert \stackrel{p}{\to} 0
	\end{equation*}
	as $T \to \infty$, which further implies that
	\begin{equation*}
		\sup_{(\ww, \alpha) \in K \times \supp(m)} \left\vert \operatorname{CVaR}_{\alpha} (- \hat{r}_{\ww, T}) - \operatorname{CVaR}_{\alpha} (- r_{\ww}) \right\vert \stackrel{p}{\to} 0,
	\end{equation*}
	and consequently $\sup_{\ww \in K} \vert \rho_m (- \hat{r}_{\ww, T}) - \rho_m (- r_{\ww}) \vert \stackrel{p}{\to} 0$, which proves our claim. The law of large numbers and the linear independence constraint qualification in \cref{ass:identification_constraint} imply that $\hat{\cC}_T$ converges locally in probability to $\cC$ in the Fell topology (cf. \cite{molchanov2005theory}). The compact near-optimal level set in \cref{ass:identification_constraint} makes the sequence of empirical minimizers tight. Local uniform convergence of the convex objectives, convergence of the feasible sets, and uniqueness of the population minimizer therefore yield $\hw\pto\ww_*$.
	
	We next derive the limiting distribution of $\sqrt{T} (\hw - \ww_*)$. To this end, we apply Theorem 5 in \cite{knight1999epi}, an extension of the argmax continuous mapping theorem \cite[Theorem 3.2.2]{wellner2013weak} that is able to handle random constraints. Define $\hat{h}_{c, s} = \sqrt{T} ( \hw - \ww_* )$, and
	\begin{equation*}
		\hat{M}_T (h) := \, T \Big( \rho_m \big( - \hat{r}_{\ww_* + h /\sqrt{T}, T } \big) - \rho_m \big( - \hat{r}_{\ww_*, T } \big) \Big).
	\end{equation*}
	By definition, $\hat{h}_{c, s}$ is a measurable minimizer of $\hat{M}_T$ over $\hat{\cC}_{T, h}$, or equivalently,
	\begin{equation*}
		\hat{h}_{c, s} = \, \arg \min_{h \in \R^N} \big\{ \hat{M}_T \cdot \bone_{\hat{\cC}_{T, h}} + \infty \cdot \bone_{\hat{\cC}_{T, h}^c} \big\},
	\end{equation*}
	where $\hat{\cC}_{T, h}$ is obtained by rescaling $\hat{\cC}_T$:
	\begin{equation}\label{eq:defn_hat_CTH}
		\begin{split}
			\hat{\cC}_{T, h} := \, \Big\{ & h \in \R^N : \, F_c^{(e)} h = 0, \,\, F_c^{(i)} h \le \sqrt{T} (b_c^{(i)} - F_c^{(i)} \ww_* ), \\
			& \hat{\E}_T [F_s^{(e)} (\rr)] h = - \mathbb{G}_T (F_s^{(e)}) \ww_*, \,\, \hat{\E}_T [F_s^{(i)} (\rr)] h \le \sqrt{T} (b_s^{(i)} - \hat{\E}_T [F_s^{(i)} (\rr)] \ww_* ) \Big\}.
		\end{split}
	\end{equation}
	It then suffices to figure out the limiting process of $\hat{M}_T$ and the limiting set of $\hat{\cC}_{T, h}$, which we proceed one by one.
	
	\paragraph{Reduction from $\hat{M}_T$ to $\widetilde{M}_T$.} We first note that for all $\alpha \in (0, 1)$,
	\begin{align*}
		\operatorname{CVaR}_{\alpha} ( -\hat{r}_{\ww, T} ) = \, & \hat{R}_{\alpha} \big( \ww, \hat{z}_T (\ww, \alpha) \big) = \hat{\E}_T \left[ L_{\alpha} \left( \ww, \hat{z}_T (\ww, \alpha); \rr \right) \right].
	\end{align*}
	Using the spectral representation, we obtain that
	\begin{align*}
		\hat{M}_T (h) = \, \int_{0}^{1} T \, \hat{\E}_T \left[ L_{\alpha} \left( \ww_* + h / \sqrt{T}, \hat{z}_T (\ww_* + h / \sqrt{T}, \alpha); \rr \right) - L_{\alpha} \left( \ww_*, \hat{z}_T (\ww_*, \alpha); \rr \right) \right] \d m (\alpha).
	\end{align*}
	Define
	\begin{align*}
		& \widetilde{M}_T (h) = \, \int_{0}^{1} T \, \hat{\E}_T \left[ L_{\alpha} \left( \ww_* + h / \sqrt{T}, z (\ww_* + h / \sqrt{T}, \alpha); \rr \right) - L_{\alpha} \left( \ww_*, z (\ww_*, \alpha); \rr \right) \right] \d m (\alpha) \\
		= \, & T \, \hat{\E}_T \left[ \int_{0}^{1} L_{\alpha} \left( \ww_* + h / \sqrt{T}, z (\ww_* + h / \sqrt{T}, \alpha); \rr \right) \d m (\alpha) - \int_{0}^{1} L_{\alpha} \left( \ww_*, z (\ww_*, \alpha); \rr \right) \d m (\alpha) \right],
	\end{align*}
	where we recall that
	\begin{equation*}
		z (\ww, \alpha) = \, \arg \min_{z \in \R} \left\{ z + \frac{1}{1 - \alpha} \E \left[ \left( -r_{\ww} - z \right)_+ \right] \right\}
	\end{equation*}
	is the $\alpha$-quantile of $-r_{\ww} = - \ww^\top \rr$. We next prove that for any $R > 0$,
	\begin{equation}\label{eq:first_convergence_M}
		\sup_{\norm{h}_2 \le R} \vert \hat{M}_T (h) - \widetilde{M}_T (h) \vert \to 0 \quad \mbox{in probability}.
	\end{equation}

	To this end, note that for any fixed $\ww$ and $\alpha$, we have
	\begin{align*}
		& \hat{\E}_T \left[ L_{\alpha} \left( \ww, \hat{z}_T (\ww, \alpha); \rr \right) \right] - \hat{\E}_T \left[ L_{\alpha} \left( \ww, z (\ww, \alpha); \rr \right) \right] \\
		= \, & \hat{z}_T (\ww, \alpha) - z (\ww, \alpha) + \frac{1}{1 - \alpha} \int_{\hat{z}_T (\ww, \alpha)}^{z (\ww, \alpha)} \hat{\E}_T \left[ \bone_{-r_{\ww} \ge t} \right] \d t \\
		= \, & \frac{1}{1 - \alpha} \int_{z (\ww, \alpha)}^{\hat{z}_T (\ww, \alpha)} \left( \hat{\E}_T \left[ \bone_{-r_{\ww} \le t} \right] - \alpha \right) \d t.
	\end{align*}
	Define $\mathbb{Q}_T (\ww, t) = \mathbb{G}_T (\bone_{-r_{\ww} \le t}) = \sqrt{T} (\hat{\E}_T [ \bone_{-r_{\ww} \le t} ] - \P (-r_{\ww} \le t) )$. Then, \cref{lem:indicator_donsker} implies that
	\begin{equation*}
		\mathbb{Q}_T (\ww, t) \wto \mathbb{Q} (\ww, t) \quad \mbox{in} \,\, \ell^{\infty} (K \times \R)
	\end{equation*}
	for any compact set $K$, where $\mathbb{Q} (\ww, t) = \mathbb{G} (\bone_{-r_{\ww} \le t})$ is the corresponding $\P$-Brownian bridge indexed by $(\ww, t)$. Since our goal is proving $\sup_{\norm{h}_2 \le R} \vert \hat{M}_T (h) - \widetilde{M}_T (h) \vert \stackrel{p}{\to} 0$, using an extended version of Skorokhod's representation theorem, we may assume without loss of generality that $\mathbb{Q}_T$ almost surely converges to $\mathbb{Q}$ in $\ell^{\infty} (K \times \R)$. By definition, we know that $\hat{z}_T (\ww, \alpha)$ satisfies
	\begin{equation}\label{eq:solve_emp_quantile}
		\mathbb{Q}_T (\ww, \hat{z}_T (\ww, \alpha)) = \, \sqrt{T} \left( \alpha - \P (-r_{\ww} \le \hat{z}_T (\ww, \alpha) ) \right) + O (1/ \sqrt{T}),
	\end{equation}
	which further implies that
	\begin{align*}
		& \P (-r_{\ww} \le z (\ww, \alpha)) - \P (-r_{\ww} \le \hat{z}_T (\ww, \alpha)) \\
		= \, & \alpha - \P (-r_{\ww} \le \hat{z}_T (\ww, \alpha)) = O_p (1 / \sqrt{T}).
	\end{align*}
	Since by our assumption, $p_{-r_{\ww}} (\cdot)$ is continuous and always positive, we deduce that $\hat{z}_T (\ww, \alpha) - z (\ww, \alpha) = O_p (1 / \sqrt{T})$. Combining this estimate with \cref{eq:solve_emp_quantile}, it follows that
	\begin{equation*}
		\hat{z}_T (\ww, \alpha) - z (\ww, \alpha) = \, - \frac{1}{\sqrt{T}} \frac{\mathbb{Q} (\ww, z(\ww, \alpha))}{p_{-r_{\ww}} (z(\ww, \alpha))} + o_p ( 1 / \sqrt{T} ).
	\end{equation*}
	Further, we note that the $o_p ( 1 / \sqrt{T} )$ term is uniformly small in $(\ww, \alpha)$, since $\mathbb{Q}_T$ converges uniformly in $\mathbb{Q}$ and $z (\ww, \alpha)$ is uniformly bounded for $\ww \in K$ and $\alpha \in [u, 1 - u]$. Namely, defining
	\begin{equation*}
		\tilde{z}_T (\ww, \alpha) = \, z (\ww, \alpha) - \frac{1}{\sqrt{T}} \frac{\mathbb{Q} (\ww, z(\ww, \alpha))}{p_{-r_{\ww}} (z(\ww, \alpha))},
	\end{equation*}
	we have
	\begin{equation}\label{eq:unif_conv_emp_quantile}
		\sup_{(\ww, \alpha) \in K \times [u, 1 - u]} \sqrt{T} \left\vert \hat{z}_T (\ww, \alpha) - \tilde{z}_T (\ww, \alpha) \right\vert \stackrel{p}{\to} 0 \quad \mbox{as} \,\, T \to \infty.
	\end{equation}
	Using \cref{eq:unif_conv_emp_quantile}, as well as the facts that $\{ \mathbb{Q}_T \}_{T=1}^{\infty}$ is a tight sequence in $\ell^{\infty} (K \times \R)$, and $ \P ( -r_{\ww} \le t ) = \alpha$ at $t = z (\ww, \alpha)$, we deduce that
	\begin{equation}\label{eq:approximate_emp_quantile}
		\sqrt{T} \int_{\hat{z}_T (\ww, \alpha)}^{\tilde{z}_T (\ww, \alpha)} \mathbb{Q}_T (\ww, t) \d t + T \int_{\hat{z}_T (\ww, \alpha)}^{\tilde{z}_T (\ww, \alpha)} \left( \P \left( -r_{\ww} \le t \right) - \alpha \right) \d t        
	\end{equation}
	converges to $0$ in probability, uniformly over $(\ww, \alpha) \in K \times [u, 1 - u]$. 
	
	In what follows, we show that
	\begin{equation}\label{eq:conv_emp_quantile_process}
		\sup_{(\ww, \alpha) \in K \times [u, 1 - u]} \left\vert T \left( \hat{\E}_T \left[ L_{\alpha} \left( \ww, \hat{z}_T (\ww, \alpha); \rr \right) \right] - \hat{\E}_T \left[ L_{\alpha} \left( \ww, z (\ww, \alpha); \rr \right) \right] \right) - A (\ww, \alpha) \right\vert \stackrel{p}{\to} 0
	\end{equation}
	as $T \to \infty$, where
	\begin{equation*}
		A (\ww, \alpha) = \, - \frac{1}{2 (1 - \alpha)} \frac{\mathbb{Q} (\ww, z(\ww, \alpha))^2}{p_{-r_{\ww}} (z (\ww, \alpha))}.
	\end{equation*}
	To this end, we write
	\begin{align*}
		& T \left( \hat{\E}_T \left[ L_{\alpha} \left( \ww, \hat{z}_T (\ww, \alpha); \rr \right) \right] - \hat{\E}_T \left[ L_{\alpha} \left( \ww, z (\ww, \alpha); \rr \right) \right] \right) \\
		= \, & \frac{T}{1 - \alpha} \int_{z (\ww, \alpha)}^{\hat{z}_T (\ww, \alpha)} \left( \hat{\E}_T \left[ \bone_{-r_{\ww} \le t} \right] - \alpha \right) \d t \\
		= \, & \frac{\sqrt{T}}{1 - \alpha} \int_{z (\ww, \alpha)}^{\hat{z}_T (\ww, \alpha)} \mathbb{Q}_T (\ww, t) \d t + \frac{T}{1 - \alpha} \int_{z (\ww, \alpha)}^{\hat{z}_T (\ww, \alpha)} \left( \P \left( -r_{\ww} \le t \right) - \alpha \right) \d t.
	\end{align*}
	Due to \cref{eq:approximate_emp_quantile}, and the fact that $\mathbb{Q}_T$ converges uniformly to $\mathbb{Q}$, it suffices to consider
	\begin{equation*}
		\frac{\sqrt{T}}{1 - \alpha} \int_{z (\ww, \alpha)}^{\tilde{z}_T (\ww, \alpha)} \mathbb{Q} (\ww, t) \d t + \frac{T}{1 - \alpha} \int_{z (\ww, \alpha)}^{\tilde{z}_T (\ww, \alpha)} \left( \P \left( -r_{\ww} \le t \right) - \alpha \right) \d t,
	\end{equation*}
	which converges to $A(\ww, \alpha)$ since $\mathbb{Q}$ and $p_{-r_{\ww}} (\cdot)$ are continuous. This proves \cref{eq:conv_emp_quantile_process}.
	
	The proof of \cref{eq:first_convergence_M} is then concluded, using \cref{eq:conv_emp_quantile_process} and continuity of $\mathbb{Q}$, $p_{-r_{\ww}} (\cdot)$ and $z (\ww, \alpha)$, which follows from \cref{ass:regularity_density} and the inverse function theorem.
	
	\paragraph{Asymptotics of $\widetilde{M}_T$.} Since $\hat{M}_T$ and $\widetilde{M}_T$ are convex and are asymptotically equivalent uniformly on compact sets by \eqref{eq:first_convergence_M}, it suffices to consider minimizing $\widetilde{M}_T$ over $\hat{\cC}_{T, h}$. Recall the definition of $U_m (\ww; \rr)$ in \cref{sec:low-dim}; then
	\begin{align*}
		\widetilde{M}_T (h) = \, & T \, \hat{\E}_T \left[ U_m \big( \ww_* + h / \sqrt{T}; \rr \big) - U_m \big( \ww_*; \rr \big) \right] \\
		:= \, & \widetilde{M}_T^{(1)} (h) + \sqrt{T} h^\top \E \big[ \nabla_{\ww} U_m(\ww_* ; \rr ) \big],
	\end{align*}
	where we define
	\begin{equation*}
		\widetilde{M}_T^{(1)} (h) = \, T \left( \hat{\E}_T \left[ U_m \big( \ww_* + h / \sqrt{T}; \rr \big) - U_m \big( \ww_*; \rr \big) \right] - \frac{h^\top}{\sqrt{T}} \E \big[ \nabla_{\ww} U_m(\ww_* ; \rr ) \big] \right).
	\end{equation*}

	As for $\widetilde{M}_T^{(1)}$, consider the function class
	\begin{equation*}
		\left\{ \ww \mapsto U_m (\ww; \rr) = \int_{0}^{1} L_{\alpha} \left( \ww, z (\ww, \alpha); \rr \right) \d m (\alpha): \ww \in K \right\}.
	\end{equation*}
	We will show that this mapping is Lipschitz in $\ww$, with an $L^2$-integrable Lipschitz constant (depending on $\rr$). By definition, we have for any $\ww_1, \ww_2 \in K$:
	\begin{align*}
		& \left\vert \int_{0}^{1} L_{\alpha} \left( \ww_1, z (\ww_1, \alpha); \rr \right) \d m (\alpha) - \int_{0}^{1} L_{\alpha} \left( \ww_2, z (\ww_2, \alpha); \rr \right) \d m (\alpha) \right\vert \\
		\le \, & \frac{1 + u}{u} \sup_{\alpha \in [u, 1 - u]} \left\vert z (\ww_1, \alpha) - z (\ww_2, \alpha) \right\vert + \frac{1}{u} \norm{\rr}_2 \cdot \norm{\ww_1 - \ww_2}_2 \\
		\le \, & \left( \frac{1+u}{u} \sup_{(\ww, \alpha) \in K \times [u, 1-u]} \frac{\norm{\nabla_{\ww} \P (-r_{\ww} \le z(\ww, \alpha) )}_2 }{p_{-r_{\ww}} (z (\ww, \alpha))} + \frac{1}{u} \norm{\rr }_2 \right) \cdot \norm{\ww_1 - \ww_2}_2.
	\end{align*}
	Using the local empirical-process expansion in \cite[Section 3.2]{wellner2013weak} (see also \cite[Theorem 12.6]{sen2018gentle}), we know that for any $R > 0$, $\widetilde{M}_T^{(1)} (h) \wto M^{(1)} (h)$ in $\ell^{\infty} (\{ h: \norm{h}_2 \le R \})$, where
	\begin{equation*}
		M^{(1)} (h) = \, h^\top \mathbb{G} \left( \nabla_{\ww} U_m (\ww_*; \rr ) \right) + \frac{1}{2} h^\top \nabla_{\ww}^2 \E \left[ U_m (\ww_*; \rr ) \right] h
	\end{equation*}
	consists of the non-Lagrangian terms in the definition of $M$.
	
	As for the remainder $\sqrt{T} h^\top \E \big[ \nabla_{\ww} U_m(\ww_* ; \rr ) \big]$, we note that the KKT conditions of the convex optimization problem~\eqref{eq:pop_rho_risk_general_linear} imply
	\begin{align*}
		& \E \big[ \nabla_{\ww} U_m(\ww_* ; \rr ) \big] = \, \nabla_{\ww} \E \big[ U_m(\ww_* ; \rr ) \big] = \nabla_{\ww} \rho_m \big( - r_{\ww_*} \big) \\
		= \, & - \left( F_{c}^{(e) \top} \eta_c^{(e)} + F_{c, +}^{(i, *) \top} \eta_{c, +}^{(i, *)} + \E [F_s^{(e)} (\rr)]^\top \eta_s^{(e)} + \E [F_{s, +}^{(i, *)} (\rr)]^\top \eta_{s, +}^{(i, *)} \right).
	\end{align*}
	Recalling the definition of $\hat{\cC}_{T, h}$ from \cref{eq:defn_hat_CTH}, we deduce that
	\begin{align*}
		\sqrt{T} h^\top \E \big[ \nabla_{\ww} U_m(\ww_* ; \rr ) \big] = \, - \sqrt{T} \left( \big( F_{c, +}^{(i, *)} h \big)^\top \eta_{c, +}^{(i, *)} + \big( \E [F_s^{(e)} (\rr)] h \big)^\top \eta_s^{(e)} + \big( \E [F_{s, +}^{(i, *)} (\rr)] h \big)^\top \eta_{s, +}^{(i, *)} \right).
	\end{align*}
	Further, on $\hat{\cC}_{T, h}$ we have $F_{c, +}^{(i, *)} h \le 0$. Since our goal is to minimize $\widetilde{M}_T$, we must enforce the constraint $F_{c, +}^{(i, *)} h = 0$ to avoid blow-up of the target function. For the second term, from the definition of $\hat{\cC}_{T, h}$ we know that
	\begin{align*}
		- \sqrt{T} \big( \E [F_s^{(e)} (\rr)] h \big)^\top \eta_s^{(e)} = \, \sqrt{T} \ww_*^\top \mathbb{G}_T \big( F_s^{(e)} \big)^\top \eta_s^{(e)} + h^\top \mathbb{G}_T \big( F_s^{(e)} \big)^\top \eta_s^{(e)},
	\end{align*}
	where the first term on the right hand side does not depend on $h$, and the limit of the second term corresponds to the term $h^\top \mathbb{G} (F_s^{(e)} (\rr)^\top \eta_s^{(e)})$ in \cref{eq:defn_M_func}. Using the same argument, we can show that the constraints in the definition of $\hat{\cC}_{T, h}$ corresponding to $(F_{s, +}^{(i,*)}, b_{s, +}^{(i,*)})$ must be satisfied with equality, and justify the presence of the term $h^\top \mathbb{G} (F_{s, +}^{(i, *)} (\rr)^\top \eta_{s, +}^{(i, *)})$ in \cref{eq:defn_M_func}.
	
	\paragraph{Asymptotics of $\hat{\cC}_{T, h}$.} Finally, we are in position to derive the limiting set of $\hat{\cC}_{T, h}$. It is straightforward to see that the equality constraints converge in law to their population versions as $T \to \infty$. For the inequality constraints, we note that by \cref{defn:asym_dist_low_dim}, if $(F, b) \in (F_c^{(i, *)}, b_c^{(i, *)})$, then $b - F \ww_* = 0$. Otherwise, $b - F \ww_* > 0$ and consequently $\sqrt{T} (b - F \ww_*) \to \infty$ as $T \to \infty$. Similarly, if $(F, b) \in (F_s^{(i, *)}, b_s^{(i, *)})$, then $b - \E [F(\rr)] \ww_* = 0$ and
	\begin{equation*}
		\sqrt{T} (b - \hat{\E}_T [F(\rr)] \ww_*) = \, -\mathbb{G}_T (F) \ww_* \stackrel{d}{\to} - \mathbb{G} (F) \ww_* 
	\end{equation*}
	as $T \to \infty$. Otherwise, $b - \E [F(\rr)] \ww_* > 0$ and $\sqrt{T} (b - \hat{\E}_T [F(\rr)] \ww_*) \stackrel{p}{\to} \infty$. Recalling that the inequality constraints corresponding to $F_{c, +}^{(i,*)}$ and $F_{s, +}^{(i,*)}$ must be saturated with equalities, we therefore define
	\begin{equation*}
		\begin{split}
			\cC_h := \, \Big\{ h \in \R^N : \, & \, F_c^{(e)} h = 0, \,\, F_{c, +}^{(i, *)} h = 0, \,\, F_{c, 0}^{(i, *)} h \le 0, \,\, \E [F_s^{(e)} (\rr)] h = - \mathbb{G} (F_s^{(e)}) \ww_*, \\ 
			& \E [F_{s, +}^{(i, *)} (\rr)] h = - \mathbb{G} (F_{s, +}^{(i, *)}) \ww_*, \,\, \E [F_{s, 0}^{(i, *)} (\rr)] h \le - \mathbb{G} (F_{s, 0}^{(i, *)}) \ww_* \Big\}.
		\end{split}
	\end{equation*}
	The difference between any two points in $\cC_h$ lies in the subspace $\cT_*$ defined before \Cref{ass:second_order_regularity}. The positive-curvature condition in \Cref{ass:second_order_regularity} therefore makes $M$ strictly convex on $\cC_h$, so $h_{c,s}$ is its unique minimizer, i.e.,
	\begin{equation*}
		h_{c, s} = \, \arg \min_{h \in \R^N} \big\{ M \cdot \bone_{\cC_h} + \infty \cdot \bone_{\cC_h^c} \big\}.
	\end{equation*}
	Combining our arguments, it follows that
	\begin{equation*}
		\hat{M}_T \cdot \bone_{\hat{\cC}_{T, h}} + \infty \cdot \bone_{\hat{\cC}_{T, h}^c} \,\, \mbox{epi-converges to} \,\, M \cdot \bone_{\cC_h} + \infty \cdot \bone_{\cC_h^c}
	\end{equation*}
	in distribution as $T \to \infty$, as per the definition in Sections 2 and 3 of \cite{knight1999epi}. Invoking \cite[Theorem 5(b)]{knight1999epi}, we know that $\hat{h}_{c, s} \wto h_{c, s}$. Finally, the ``Furthermore'' part follows directly from the delta method. This completes the proof of \cref{thm:normality_general_risk}.
\end{proof}

\subsection{Proof of \Cref{prop:normality_risky_constraints}}
\label{app:proof_target_return}

\begin{proof}
	The result is a direct specialization of
	\Cref{thm:normality_general_risk}. Take the deterministic budget
	constraint and stochastic expected-return equality constraint by setting
	\[
	F_c^{(e)}=\bone^\top,
	\qquad
	b_c^{(e)}=1,
	\qquad
	F_s^{(e)}(\rr)=\rr^\top,
	\qquad
	b_s^{(e)}=\mu_0,
	\]
	with no inequality constraints. Convexity gives $H_*\succeq0$,
	and nonsingularity of $K_*$ implies that $\bmu^{(1)}$ has full
	column rank and $H_*$ is positive definite on
	$\{d:\bmu^{(1)\top}d=0\}$. Indeed, a nonzero tangent vector
	with $d^\top H_*d=0$ would satisfy $H_*d=0$ and produce a
	null vector of $K_*$. Together with uniqueness, level-boundedness,
	and $\ww_*\neq\bzero$ from the budget constraint, these facts
	verify Assumptions~\ref{ass:identification_constraint} and
	\ref{ass:second_order_regularity}. The population KKT system is
	\[
	\nabla_{\ww}\rho_m(-r_{\ww_*})
	+
	\eta_s\bmu
	+
	\eta_c\bone
	=
	0,
	\qquad
	\ww_*^\top\bmu=\mu_0,
	\quad
	\ww_*^\top\bone=1.
	\]
	Under the additional density condition, the limiting stochastic
	quadratic program in \Cref{defn:asym_dist_low_dim} is an
	equality-constrained Gaussian quadratic program. Ordering the
	multipliers by expected return and then budget, its KKT system is
	\[
	\begin{bmatrix}
		H_* & \bmu^{(1)}\\
		\bmu^{(1)\top} & \bzero
	\end{bmatrix}
	\begin{bmatrix}
		h\\
		\xi_s\\
		\xi_c
	\end{bmatrix}
	=
	-
	\begin{bmatrix}
		\mathbb{G}\!\left(V_*(\rr)\right)\\
		\mathbb{G}\!\left(r_*(\rr)\right)\\
		0
	\end{bmatrix},
	\]
	where
	\[
	V_*(\rr)
	=
	\nabla_{\ww}U_m(\ww_*;\rr)+\eta_s\rr,
	\qquad
	r_*(\rr)=\ww_*^\top\rr.
	\]
	Nonsingularity of $K_*$ therefore gives
	\[
	h
	=
	-
	\begin{bmatrix}
		\id_N & \bzero
	\end{bmatrix}
	K_*^{-1}
	\begin{bmatrix}
		\mathbb{G}\!\left(V_*(\rr)\right)\\
		\mathbb{G}\!\left(r_*(\rr)\right)\\
		0
	\end{bmatrix}.
	\]
	The sign has no effect on the covariance, yielding
	\[
	\Sigma(\mu_0,\rho_m)
	=
	\var\left(
	\begin{bmatrix}
		\id_N & \bzero
	\end{bmatrix}
	K_*^{-1}
	\begin{bmatrix}
		V_*(\rr)\\
		r_*(\rr)\\
		0
	\end{bmatrix}
	\right).
	\]
	Consistency and the stated limiting distribution now follow directly
	from \Cref{thm:normality_general_risk}.
\end{proof}

\subsection{Auxiliary Lemmas}\label{sec:auxiliary_sec3}
\begin{lem}\label{lem:general_diff}
	Let $g(\rr)$ and $h(\rr, \ww)$ be two measurable functions, such that $g(\rr)$ and $g(\rr) \nabla_{\ww} h(\rr, \ww)$ are integrable, and the mapping $\ww \mapsto \E [ \vert g (\rr) \nabla_{\ww} h(\rr, \ww) \vert ]$ is locally integrable. Further, assume that the mappings
	\begin{equation*}
		(\ww, u) \mapsto p_{h(\rr, \ww)} (u) \,\, \mbox{and} \,\, (\ww, u) \mapsto \E \left[ g(\rr) \nabla_{\ww} h(\rr, \ww) \vert h(\rr, \ww ) = u \right] 
	\end{equation*}
	are continuous. Then, the function
	\begin{equation*}
		\psi(\ww) = \, \E \left[ g(\rr ) \bone \left\{ h(\rr, \ww ) \ge 0 \right\} \right]
	\end{equation*}
	is differentiable, and
	\begin{equation}\label{eq:indicator_gradient}
		\nabla_{\ww} \psi(\ww) = \, p_{h(\rr, \ww)} (0) \E \left[ g(\rr) \nabla_{\ww} h(\rr, \ww) \vert h(\rr, \ww) = 0 \right],
	\end{equation}
	provided that the right hand side is continuous in $\ww$.
\end{lem}

\begin{proof}
	For any $\veps > 0$, we define
	\begin{equation*}
		\psi_{\veps} (\ww) = \E \left[ g(\rr ) \Phi \left( \frac{h(\rr, \ww)}{\veps} \right) \right],
	\end{equation*}
	where $\Phi$ is the standard Gaussian CDF. Since $g(\rr )$ is integrable, by dominated convergence theorem we know that $\lim_{\veps \to 0} \psi_{\veps} (\ww ) = \psi (\ww)$ for all $\ww \in \R^N$. Note that for any fixed $\veps > 0$ and $\ww_0, \ww_1 \in \R^N $:
	\begin{align*}
		\psi_{\veps} (\ww_1) - \psi_{\veps} (\ww_0) = \, & \E \left[ g(\rr ) \Phi \left( \frac{h(\rr, \ww_1 )}{\veps} \right) - g(\rr ) \Phi \left( \frac{h(\rr, \ww_0 )}{\veps} \right) \right] \\
		= \, & \E \left[ \int_{0}^{1} \left\langle \frac{1}{\veps} \phi \left( \frac{h(\rr, \ww_t)}{\veps} \right) g(\rr ) \nabla_{\ww } h(\rr, \ww_t ), \ww_1 - \ww_0 \right\rangle \d t \right],
	\end{align*}
	where $\phi$ is the standard Gaussian PDF, and $\ww_t = (1 - t) \ww_0 + t \ww_1$ denotes a linear interpolation path between $\ww_0$ and $\ww_1$. Now since $\Phi$ and $\phi$ are uniformly bounded, and by our assumption that $g(\rr ) \nabla_{\ww} h(\rr, \ww )$ is jointly integrable in $(\rr, \ww)$ over a neighborhood of the line segment connecting $\ww_0$ and $\ww_1$, we can apply Fubini's theorem to obtain that
	\begin{align*}
		\psi_{\veps} (\ww_1) - \psi_{\veps} (\ww_0) = \, & \int_{0}^{1} \left\langle \E \left[ \frac{1}{\veps} \phi \left( \frac{h(\rr, \ww_t)}{\veps} \right) g(\rr) \nabla_{\ww} h(\rr, \ww_t ) \right], \ww_1 - \ww_0 \right\rangle \d t.
	\end{align*}
	To compute the $\veps \to 0$ limit of the above term, we notice that for any fixed $\ww$,
	\begin{align*}
		& \E \left[ \frac{1}{\veps} \phi \left( \frac{h(\rr, \ww )}{\veps} \right) g(\rr) \nabla_{\ww} h(\rr, \ww ) \right] \\
		= \, & \E \left[ \frac{1}{\veps} \phi \left( \frac{h(\rr, \ww )}{\veps} \right) \E \left[ g(\rr ) \nabla_{\ww } h(\rr, \ww) \vert h(\rr, \ww) \right] \right] \\
		= \, & \int_{\R} \frac{1}{\veps} \phi \left( \frac{u}{\veps} \right) p_{h(\rr, \ww)} (u) \E \left[ g(\rr) \nabla_{\ww} h(\rr, \ww) \vert h(\rr, \ww) = u \right] \d u,
	\end{align*}
	thus leading to (again by Fubini's theorem)
	\begin{align*}
		& \psi_{\veps} (\ww_1) - \psi_{\veps} (\ww_0) \\
		= \, & \int_{\R} \frac{1}{\veps} \phi \left( \frac{u}{\veps} \right) \times \left\langle \int_{0}^{1} p_{h(\rr, \ww_t)} (u) \E \left[ g(\rr) \nabla_{\ww} h(\rr, \ww_t) \vert h(\rr, \ww_t) = u \right] \d t , \ww_1 - \ww_0 \right\rangle \d u.
	\end{align*}
	Under our assumption, the mapping $(\ww, u) \mapsto p_{h(\rr, \ww )} (u) \E \left[ g(\rr ) \nabla_{\ww} h(\rr, \ww) \vert h(\rr, \ww ) = u \right]$ is continuous. Therefore, the function
	\begin{equation*}
		\left\langle \int_{0}^{1} p_{h(\rr, \ww_t)} (u) \E \left[ g(\rr) \nabla_{\ww} h(\rr, \ww_t) \vert h(\rr, \ww_t) = u \right] \d t , \ww_1 - \ww_0 \right\rangle
	\end{equation*}
	is continuous in $u$. Consequently, as $\veps \to 0$ we have
	\begin{align*}
		\psi(\ww_1) - \psi(\ww_0) = \, \int_{0}^{1} \left\langle p_{h(\rr, \ww_t)} (0) \E \left[ g(\rr ) \nabla_{\ww} h(\rr, \ww_t) \vert h(\rr, \ww_t) = 0 \right], \ww_1 - \ww_0 \right\rangle \d t,
	\end{align*}
	which implies \cref{eq:indicator_gradient}.
	This completes the proof of \cref{lem:general_diff}.
\end{proof}

\begin{lem}\label{lem:indicator_donsker}
	Let $K \subset \R^N \backslash \{ \bzero \}$ be any compact set, and assume that
	\begin{equation}\label{eq:bounded_density_assumption}
		\sup_{(\ww, u) \in K \times \R } p_{-r_{\ww}} (u) < \infty.
	\end{equation}
	Then, the function class
	\begin{equation*}
		\cF_K = \, \left\{ \bone_{-r_{\ww} \le t}: \ww \in K, t \in \R \right\}
	\end{equation*}
	indexed by $(\ww, t) \in K \times \R$ is $P$-Donsker.
\end{lem}

\begin{proof}
	We use the bracketing Donsker criterion in \cite[Section 2.5.2]{wellner2013weak} (see also \cite[Theorem 11.3]{sen2018gentle}) to show that $\cF_K$ is $P$-Donsker. First, it is easy to see that the constant function $\bone$ is a square-integrable envelope of $\cF_K$. Further, we define
	\begin{equation*}
		J_{[ \, ]}\left(a, \mathcal{F}_K, L^2 (P)\right):=\int_0^a \sqrt{\log N_{[ \, ]}\left(\veps, \mathcal{F}_K \cup\{0\}, L^2 (P)\right)} \, \d \veps,
	\end{equation*}
	where $N_{[ \, ]}$ denotes the bracketing number. It suffices to show that $J_{[ \, ]} (1, \mathcal{F}_K, L^2 (P) ) < \infty$. We prove this by estimating the bracketing number $N_{[ \, ]} ( \veps, \mathcal{F}_K \cup\{0\}, L^2 (P) )$ for all $\veps > 0$. For some $\delta > 0$ (to be determined later), let $N (\delta, K) = \{ \ww_1, \cdots, \ww_M \}$ be a $\delta$-net of $K$ (in terms of $\ell^2$-norm). Then, we know that
	\begin{equation*}
		M = \left\vert N (\delta, K) \right\vert \le \left( \frac{C_K}{\delta} \right)^N
	\end{equation*}
	for some constant $C_K$ only depending on $K$.
	Denote $J = \lceil 1 / \sqrt{\delta} \rceil$, and for each $i \in [M]$, let 
	\begin{equation*}
		- \infty = t_i (0) < t_i (1) < \cdots < t_i (J) = + \infty
	\end{equation*}
	be a sequence of increasing quantiles such that
	\begin{equation*}
		\P \left( - \rr^\top \ww_i \le t_i (j) \right) - \P \left( - \rr^\top \ww_i \le t_i (j - 1) \right) \le \sqrt{\delta}, \quad \forall j \in [J].
	\end{equation*}
	We are now in position to construct the brackets. For $i \in [M]$ and $j \in [J]$, define
	\begin{equation*}
		l_{i, j} (\rr) = \, \bone_{ - \rr^\top \ww_i \le t_i (j - 1) - \norm{\rr}_2 \delta}, \quad u_{i, j} (\rr) = \, \bone_{ - \rr^\top \ww_i \le t_i (j) + \norm{\rr}_2 \delta}.
	\end{equation*}
	We first show that
	\begin{equation}\label{eq:bracket_inclusion_condition}
		\cF_K \subset \cup_{i=1}^{M} \cup_{j=1}^{J} [l_{i, j} (\rr), u_{i, j} (\rr)]. 
	\end{equation}
	To this end, note that for any $\ww \in K$, there exists a $\ww_i \in N (\delta, K)$ such that $\norm{\ww - \ww_i}_2 \le \delta$. Hence, the Cauchy--Schwarz inequality implies that
	\begin{equation*}
		-\rr^\top \ww_i - \norm{\rr}_2 \delta \le -\rr^\top \ww \le -\rr^\top \ww_i + \norm{\rr}_2 \delta.
	\end{equation*}
	For any $t \in \R$, there exists some $j \in [J]$ such that $t \in [t_i (j-1), t_i (j)]$. Therefore,
	\begin{align*}
		\bone_{-\rr^\top \ww \le t} \le \, & \bone_{-\rr^\top \ww_i - \norm{\rr}_2 \delta \le t_i (j)} = u_{i, j} (\rr), \\
		\bone_{-\rr^\top \ww \le t} \ge \, & \bone_{-\rr^\top \ww_i + \norm{\rr}_2 \delta \le t_i (j - 1)} = l_{i, j} (\rr),
	\end{align*}
	which proves \cref{eq:bracket_inclusion_condition}. Further, we have for all $i, j$:
	\begin{align*}
		\norm{u_{i, j} - l_{i, j}}_{L^2 (P)}^2 = \, & \P \left( t_i (j - 1) - \norm{\rr}_2 \delta < -\rr^\top \ww_i \le t_i (j) + \norm{\rr}_2 \delta \right) \\
		\le \, & \P \left( \norm{\rr}_2 > \frac{1}{\sqrt{\delta}} \right) + \P \left( t_i (j - 1) - \sqrt{\delta} < -\rr^\top \ww_i \le t_i (j) + \sqrt{\delta} \right) \\
		\stackrel{(i)}{\le} \, & \delta \E \left[ \norm{\rr}_2^2 \right] + 2 \bar p_K \sqrt{\delta} + \P \left( t_i (j - 1) < -\rr^\top \ww_i \le t_i (j) \right) = O (\sqrt{\delta}),
	\end{align*}
	where $(i)$ follows from Markov's inequality and our assumption~\eqref{eq:bounded_density_assumption}, and in the above display we denote $\bar p_K = \sup_{(\ww, u) \in K \times \R } p_{-r_{\ww}} (u)$. We thus obtain that
	\begin{equation*}
		\norm{u_{i, j} - l_{i, j}}_{L^2 (P)} = O(\delta^{1/4}).
	\end{equation*}
	Choosing $\delta \propto \veps^4$, it follows that
	\begin{equation*}
		\log N_{[ \, ]} ( \veps, \mathcal{F}_K \cup\{0\}, L^2 (P) ) \le \, \log M + \log J \lesssim N \log \frac{1}{\delta} \lesssim N \log \frac{1}{\veps},
	\end{equation*}
	which implies $J_{[ \, ]} (1, \mathcal{F}_K, L^2 (P) ) < \infty$. This completes the proof of \cref{lem:indicator_donsker}.
\end{proof}

\section{Proofs for Section~\ref{sec:elliptic}}

\subsection{Proof of \Cref{prop:no_risk_free_lowdim}}
\label{app:proof_srm_risky}

\begin{proof}
	We first prove \cref{eq:w_star_no_risk_free}. Recall that
	\begin{equation*}
		\ww_*(\mu_0)
		=\arg\min_{\ww\in\R^N}\ww^\top\bSigma\ww,
		\quad\mbox{subject to}\quad
		\ww^\top\bmu=\mu_0,\quad \ww^\top\bone=1
	\end{equation*}
	does not depend on the choice of spectral measure $m$. Indeed, the
	above quadratic program can be solved explicitly, and we get
	\begin{equation*}
		\ww_*(\mu_0)
		=\bSigma^{-1}\bmu^{(1)}
		\left(\bmu^{(1)\top}\bSigma^{-1}\bmu^{(1)}\right)^{-1}
		\mu_0^{(1)},
	\end{equation*}
	which proves \cref{eq:w_star_no_risk_free}. For future convenience,
	write $\ww_*=\ww_*(\mu_0)$ and denote
	\begin{equation*}
		s_*=\|\bSigma^{1/2}\ww_*\|_2,\qquad
		b_*=
		\bSigma^{-1}\bmu^{(1)}
		\left(\bmu^{(1)\top}\bSigma^{-1}\bmu^{(1)}\right)^{-1}
		\begin{bmatrix}1\\0\end{bmatrix}.
	\end{equation*}
	In particular, $b_*b_*^\top=P_\mu$ and
	$P_\perp\bSigma\ww_*=0$.

	Next, applying \cref{prop:normality_risky_constraints} yields
	consistency and asymptotic normality of $\hw(\mu_0)$.
	To check its density bound, \cref{ass:regularity_density} and the
	mixture representation imply $\lambda>0$ a.s. and
	$\E[1/\lambda]<\infty$: the mixture density at the center of any
	nonzero portfolio is
	$\phi(0)\E[1/\lambda]/\|\bSigma^{1/2}\ww\|_2$.
	Consequently, for every compact $K\subset\R^N\setminus\{\bzero\}$,
	\[
	\sup_{(\ww,t)\in K\times\R}p_{-r_{\ww}}(t)
	\le
	\frac{\phi(0)\E[1/\lambda]}
	{\inf_{\ww\in K}\|\bSigma^{1/2}\ww\|_2}<\infty.
	\]
	To compute the asymptotic variance, it suffices to determine the law
	of $h_{c,s}$, defined in \cref{defn:asym_dist_low_dim}.
	The constraints here are $\hat{\E}_T[\rr]^\top\ww=\mu_0$ and
	$\bone^\top\ww=1$. Therefore, $h_{c,s}$ is the unique minimizer of
	the stochastic convex optimization problem
	\begin{equation}\label{eq:risky_h}
	\begin{split}
		&\mbox{minimize}\quad
		M(h)=h^\top\mathbb G\left(
		\nabla_{\ww}U_m(\ww_*;\rr)+\eta_s\rr\right)
		+\frac12h^\top\nabla_{\ww}^2\rho_m(-r_{\ww_*})h,\\
		&\mbox{subject to}\quad
		\bmu^\top h=-\mathbb G(\ww_*^\top\rr),\qquad
		\bone^\top h=0.
	\end{split}
	\end{equation}
	The Lagrange multipliers $\eta_s$ and $\eta_c$ are determined by
	\begin{equation*}
		0=\nabla_{\ww}\rho_m(-r_{\ww_*})
		+\eta_s\bmu+\eta_c\bone.
	\end{equation*}
	We then compute these related quantities. Observe that
	\begin{align*}
		\rho_m(-r_{\ww})
		&=-\ww^\top\bmu+\|\bSigma^{1/2}\ww\|_2\rho_m(\lambda Z),\\
		\nabla_{\ww}\rho_m(-r_{\ww})
		&=-\bmu+\rho_m(\lambda Z)
		\frac{\bSigma\ww}{\|\bSigma^{1/2}\ww\|_2},\\
		\nabla_{\ww}^2\rho_m(-r_{\ww})
		&=\frac{\rho_m(\lambda Z)}{\sqrt{\ww^\top\bSigma\ww}}
		\left(\bSigma-
		\frac{\bSigma\ww\ww^\top\bSigma}{\ww^\top\bSigma\ww}\right),
	\end{align*}
	where $Z\sim\normal(0,1)$ is independent of $\lambda$. This leads to
	\begin{equation*}
		\eta_s
		=1-\frac{\rho_m(\lambda Z)}{s_*}
		[\,1\ \ 0\,]
		\left(\bmu^{(1)\top}\bSigma^{-1}\bmu^{(1)}\right)^{-1}
		\mu_0^{(1)}.
	\end{equation*}
	Substituting $\ww_*$ into the Hessian formula yields
	\begin{equation*}
		\nabla_{\ww}^2\rho_m(-r_{\ww_*})
		=\frac{\rho_m(\lambda Z)}{s_*}
		\left(\bSigma-
		\frac{\bSigma\ww_*\ww_*^\top\bSigma}{s_*^2}\right).
	\end{equation*}

	Further, by definition of $U_m$, it follows that
	\begin{align*}
		\nabla_{\ww}U_m(\ww;\rr)
		={}&\int_0^1\left(
		\nabla_{\ww}L_\alpha(\ww,z(\ww,\alpha);\rr)
		+\partial_zL_\alpha(\ww,z(\ww,\alpha);\rr)
		\nabla_{\ww}z(\ww,\alpha)\right)\d m(\alpha)\\
		={}&\int_0^1\left(
		-\frac{\rr}{1-\alpha}\bone_{-r_{\ww}\ge z(\ww,\alpha)}
		+\left(1-\frac{\bone_{-r_{\ww}\ge z(\ww,\alpha)}}{1-\alpha}\right)
		\nabla_{\ww}z(\ww,\alpha)\right)\d m(\alpha),
	\end{align*}
	where
	\begin{equation*}
		\nabla_{\ww}z(\ww,\alpha)
		=-\left(\partial_z^2R_\alpha(\ww,z(\ww,\alpha))\right)^{-1}
		\nabla_{\ww}\partial_zR_\alpha(\ww,z(\ww,\alpha))
		=-\E[\rr\mid-r_{\ww}=z(\ww,\alpha)]
	\end{equation*}
	by \cref{thm:population_derivatives}. Denote
	$\vv=\bSigma^{1/2}\ww_*/s_*$, so that $\|\vv\|_2=1$.
	Keeping the mean and scale of the portfolio explicit, we have
	\begin{equation*}
		r_{\ww_*}=\mu_0+s_*\lambda\vv^\top\zz,
		\qquad
		z(\ww_*,\alpha)=-\mu_0+s_* q_\alpha(\lambda Z).
	\end{equation*}
	We thus deduce that
	\begin{align*}
		\nabla_{\ww}z(\ww_*,\alpha)
		&=-\E[\rr\mid-r_{\ww_*}=z(\ww_*,\alpha)]\\
		&=-\bmu-\E\left[
		\lambda\bSigma^{1/2}\zz
		\mid\lambda\vv^\top\zz=-q_\alpha(\lambda Z)\right]\\
		&=-\bmu+q_\alpha(\lambda Z)\bSigma^{1/2}\vv
		=-\bmu+q_\alpha(\lambda Z)\frac{\bSigma\ww_*}{s_*},
	\end{align*}
	where the third equality follows by conditioning on $\lambda$ and
	using Gaussian conditional expectation. It then follows that
	\begin{align*}
		&\mathbb G\left(\nabla_{\ww}U_m(\ww_*;\rr)+\eta_s\rr\right)\\
		={}&\mathbb G\bigg(
		\int_0^1-\frac{\bone_{-\lambda\vv^\top\zz\ge q_\alpha(\lambda Z)}}{1-\alpha}
		\bSigma^{1/2}\big(\lambda\zz+q_\alpha(\lambda Z)\vv\big)
		\d m(\alpha)+\eta_s\lambda\bSigma^{1/2}\zz\bigg)\\
		={}&\mathbb G\bigg(
		\int_0^1-\frac{\bSigma^{1/2}V_\alpha(\lambda\zz,\vv)}{1-\alpha}
		\d m(\alpha)\bigg),
	\end{align*}
	where deterministic terms disappear under the centered Gaussian
	process $\mathbb G$, and we define
	\begin{equation*}
		V_\alpha(\lambda\zz,\vv)
		=\bone_{-\lambda\vv^\top\zz\ge q_\alpha(\lambda Z)}
		\big(\lambda\zz+q_\alpha(\lambda Z)\vv\big)
		-(1-\alpha)\eta_s\lambda\zz.
	\end{equation*}
	Finally, by the preceding expression for $r_{\ww_*}$,
	\begin{equation*}
		\mathbb G(\ww_*^\top\rr)=s_*\mathbb G(\lambda\vv^\top\zz).
	\end{equation*}

	On the constraint set in \cref{eq:risky_h}, the quantity
	$h^\top\bSigma\ww_*\ww_*^\top\bSigma h$ is constant, since
	$\bSigma\ww_*$ lies in the span of $\bmu$ and $\bone$.
	We may therefore drop the corresponding rank-one term from the
	quadratic objective and multiply the objective by the positive
	constant $s_*/\rho_m(\lambda Z)$. Based on the above calculations,
	the quadratic program defining $h$ can thus be equivalently written as
	\begin{equation*}
	\begin{split}
		&\mbox{minimize}\quad
		M(h)=\frac{s_*}{\rho_m(\lambda Z)}h^\top
		\mathbb G\bigg(\int_0^1
		-\frac{\bSigma^{1/2}V_\alpha(\lambda\zz,\vv)}{1-\alpha}
		\d m(\alpha)\bigg)+\frac12h^\top\bSigma h,\\
		&\mbox{subject to}\quad
		\bmu^\top h=-s_*\mathbb G(\lambda\vv^\top\zz),\qquad
		\bone^\top h=0.
	\end{split}
	\end{equation*}
	This optimization problem can be solved explicitly, and we obtain that
	\begin{equation*}
	\begin{split}
		h_{c,s}
		={}&-\frac{s_*}{\rho_m(\lambda Z)}P_\perp
		\mathbb G\bigg(\int_0^1
		-\frac{\bSigma^{1/2}V_\alpha(\lambda\zz,\vv)}{1-\alpha}
		\d m(\alpha)\bigg)\\
		&-s_*b_*\mathbb G(\lambda\vv^\top\zz)
		:=\mathrm I+\mathrm{II}.
	\end{split}
	\end{equation*}

	We are now in position to compute $\Sigma(\mu_0,\rho_m)$, the
	covariance matrix of $h_{c,s}$. To this end, we first show that
	$\mathrm I$ and $\mathrm{II}$ are uncorrelated. It suffices to prove that
	\begin{equation}\label{eq:risky_score_cross_direction}
		\E\big[\lambda\vv^\top\zz V_\alpha(\lambda\zz,\vv)\big]
		\propto\vv,\qquad \alpha\in(0,1).
	\end{equation}
	For any vector $\vv_\perp$ orthogonal to $\vv$, we have
	\begin{align*}
		&\vv_\perp^\top
		\E\big[\lambda\vv^\top\zz V_\alpha(\lambda\zz,\vv)\big]\\
		={}&\E\big[
		\lambda\vv^\top\zz
		\big(\bone_{-\lambda\vv^\top\zz\ge q_\alpha(\lambda Z)}
		-(1-\alpha)\eta_s\big)
		\lambda\vv_\perp^\top\zz\big]=0,
	\end{align*}
	where the last equality follows by conditioning on $\lambda$ and
	using independence of $\vv^\top\zz$ and $\vv_\perp^\top\zz$.
	This proves \eqref{eq:risky_score_cross_direction}. Since
	\begin{equation*}
		P_\perp\bSigma^{1/2}\vv
		=\frac{1}{s_*}P_\perp\bSigma\ww_*=\bzero,
	\end{equation*}
	it follows that $\operatorname{Cov}(\mathrm I,\mathrm{II})=0$, and hence
	\begin{equation*}
		\Sigma(\mu_0,\rho_m)
		=\operatorname{Var}(\mathrm I)+\operatorname{Var}(\mathrm{II}).
	\end{equation*}
	Next, we compute these two variances separately. By straightforward
	calculation,
	\begin{equation*}
		\operatorname{Var}(\mathrm{II})
		=\E[\lambda^2]s_*^2b_*b_*^\top
		=\E[\lambda^2]s_*^2P_\mu.
	\end{equation*}
	To calculate $\operatorname{Var}(\mathrm I)$, note that for any
	$\alpha\in(0,1)$,
	\begin{equation*}
		\E[V_\alpha(\lambda\zz,\vv)]
		=\left((1-\alpha)q_\alpha(\lambda Z)
		-\E\left[\lambda\phi\left(
		\frac{q_\alpha(\lambda Z)}\lambda\right)\right]\right)\vv
		=:V(\alpha)\vv,
	\end{equation*}
	where $\phi$ is the standard normal density. Further, conditional
	Gaussian moments give, for any $\alpha_1,\alpha_2\in(0,1)$,
	\begin{align*}
		\E[V_{\alpha_1}(\lambda\zz,\vv)
		V_{\alpha_2}(\lambda\zz,\vv)^\top]
		&=B(\alpha_1,\alpha_2)\vv\vv^\top
		+A(\alpha_1,\alpha_2)\id_N,\\
		A(\alpha_1,\alpha_2)
		&=\E\left[\lambda^2\Phi\left(
		-\frac{q_{\alpha_1\vee\alpha_2}(\lambda Z)}\lambda\right)\right]\\
		&\quad-(1-\alpha_1)\eta_s
		\E\left[\lambda^2\Phi\left(
		-\frac{q_{\alpha_2}(\lambda Z)}\lambda\right)\right]\\
		&\quad-(1-\alpha_2)\eta_s
		\E\left[\lambda^2\Phi\left(
		-\frac{q_{\alpha_1}(\lambda Z)}\lambda\right)\right]
		+(1-\alpha_1)(1-\alpha_2)\eta_s^2\E[\lambda^2],
	\end{align*}
	where $\Phi$ denotes the standard normal distribution function. The
	exact expression for $B$ is immaterial, as we will soon see.
	Combining these results, we deduce that
	\begin{align*}
		&\operatorname{Cov}\big(V_{\alpha_1}(\lambda\zz,\vv),
		V_{\alpha_2}(\lambda\zz,\vv)\big)\\
		={}&A(\alpha_1,\alpha_2)\id_N
		+\big(B(\alpha_1,\alpha_2)-V(\alpha_1)V(\alpha_2)\big)
		\vv\vv^\top.
	\end{align*}
	It follows that
	\begin{align*}
		&\operatorname{Var}\left(
		\mathbb G\bigg(\int_0^1
		-\frac{\bSigma^{1/2}V_\alpha(\lambda\zz,\vv)}{1-\alpha}
		\d m(\alpha)\bigg)\right)\\
		={}&\left(\int_{[0,1]^2}
		\frac{A(\alpha_1,\alpha_2)}{(1-\alpha_1)(1-\alpha_2)}
		\d m(\alpha_1)\d m(\alpha_2)\right)\bSigma\\
		&+\left(\int_{[0,1]^2}
		\frac{B(\alpha_1,\alpha_2)-V(\alpha_1)V(\alpha_2)}
		{(1-\alpha_1)(1-\alpha_2)}
		\d m(\alpha_1)\d m(\alpha_2)\right)
		\frac{\bSigma\ww_*\ww_*^\top\bSigma}{s_*^2}.
	\end{align*}
	Using $P_\perp\bSigma\ww_*=0$ and
	$P_\perp\bSigma P_\perp=P_\perp$, we finally obtain
	\begin{align*}
		\operatorname{Var}(\mathrm I)
		&=\frac{s_*^2}{\rho_m(\lambda Z)^2}
		\left(\int_{[0,1]^2}
		\frac{A(\alpha_1,\alpha_2)}{(1-\alpha_1)(1-\alpha_2)}
		\d m(\alpha_1)\d m(\alpha_2)\right)
		P_\perp\bSigma P_\perp\\
		&=\frac{s_*^2}{\rho_m(\lambda Z)^2}
		\left(\int_{[0,1]^2}
		\frac{A(\alpha_1,\alpha_2)}{(1-\alpha_1)(1-\alpha_2)}
		\d m(\alpha_1)\d m(\alpha_2)\right)P_\perp.
	\end{align*}
	Adding $\operatorname{Var}(\mathrm I)$ and
	$\operatorname{Var}(\mathrm{II})$, and substituting
	$s_*^2=\ww_*(\mu_0)^\top\bSigma\ww_*(\mu_0)$, gives the stated
	covariance formula.
\end{proof}

\subsection{Proof of \Cref{prop:asymptotic_variance_mv}}

\begin{proof}
The sample mean-variance estimator is
\[
\hw_{\rm MV}(\mu_0)
=
\hat{\bS}^{-1}\hat{\bmu}^{(1)}
\left(\hat{\bmu}^{(1)\top}\hat{\bS}^{-1}\hat{\bmu}^{(1)}\right)^{-1}
\mu_0^{(1)}.
\]
Consistency follows from the laws of large numbers for $\hat{\bmu}$ and
$\hat{\bS}$ and continuity of this mapping at
$(\bmu,\E[\lambda^2]\bSigma)$. The assumed rank condition and positive
definiteness of $\bSigma$ ensure that the required inverses exist in a
neighborhood of these population values. Since $\E[\lambda^4]<\infty$,
the joint central limit theorem applies to $\hat{\bmu}$ and $\hat{\bS}$.
Differentiating the preceding mapping gives the influence function
\[
\psi_{\rm MV}(\rr)
=
-b_*\ww_*^\top\varepsilon
+\gamma_s^{(e)}P_{\perp}\varepsilon
-\frac{P_{\perp}
\left(\varepsilon\varepsilon^\top-\E[\lambda^2]\bSigma\right)\ww_*}
{\E[\lambda^2]},
\]
where $\varepsilon=\rr-\bmu$, $\ww_*=\ww_*(\mu_0)$, and
\[
b_*
=
\bSigma^{-1}\bmu^{(1)}
\left(\bmu^{(1)\top}\bSigma^{-1}\bmu^{(1)}\right)^{-1}
\begin{bmatrix}1\\0\end{bmatrix}.
\]
Here $P_{\perp}$ and $\gamma_s^{(e)}$ are defined in
\Cref{sec:elliptic}. Put $s_*^2=\ww_*^\top\bSigma\ww_*$. The identities
$P_{\perp}\bSigma\ww_*=0$, $P_{\perp}\bSigma P_{\perp}=P_{\perp}$,
and $b_*b_*^\top=P_{\mu}$ imply that the first two terms of
$\psi_{\rm MV}$ have zero cross-covariance. Symmetry makes all third
central moments vanish, so they also have zero cross-covariance with
the final term. Conditional Gaussian fourth moments give
\[
\var\!\left(P_{\perp}\varepsilon\varepsilon^\top\ww_*\right)
=\E[\lambda^4]s_*^2P_{\perp}.
\]
Consequently, the mean-zero influence function has covariance
\[
\var\bigl(\psi_{\rm MV}(\rr)\bigr)
=
\E[\lambda^2]s_*^2P_{\mu}
+\left(
\frac{\E[\lambda^4]}{\E[\lambda^2]^2}s_*^2
+\E[\lambda^2](\gamma_s^{(e)})^2
\right)P_{\perp}.
\]
The delta method therefore yields the asserted asymptotic normality
and covariance formula.
\end{proof}


\section{Proofs for \cref{sec:optimal_srm}}

\subsection{Proof of \cref{prop:unique_minimizer}}
\begin{proof}
	Fix $c\in\R$ and write $K=[u,1-u]$. Recall the change of variable $p=1-\alpha$ used in the proof of \cref{thm:SRM_optimality}. Let $\mu$ be the pushforward of $m$ under this change of variable, and retain the notation $H_\mu$ introduced there. Since $K$ is invariant under the map $\alpha\mapsto1-\alpha$, we have $\mu\in\cuP(K)$. We further write
	\begin{equation*}
		\rho_\mu=\rho_m(\lambda Z)=\int_K\frac{\cU(p)}{p}\d\mu(p).
	\end{equation*}
	Because $\lambda>0$ almost surely, $\cU(p)>0$ for every $p\in(0,1)$, and therefore $\rho_\mu>0$.
	
	We now make a second change of measure. Define $\nu=\mu/\rho_\mu$ and let
	\begin{equation*}
		\cV
		=\left\{\nu\in\mathcal{M}_+(K):
		\int_K\frac{\cU(p)}{p}\d\nu(p)=1\right\},
	\end{equation*}
	where $\mathcal{M}_+(K)$ denotes the set of finite nonnegative Borel measures on $K$. Clearly, $\nu\in\cV$. Moreover, the mapping $\mu\mapsto\nu$ is a bijection from $\cuP(K)$ to $\cV$. Indeed, for any $\nu\in\cV$, setting $\mu=\nu/\nu(K)$ gives
	\begin{equation*}
		\rho_\mu
		=\int_K\frac{\cU(p)}{p}\d\mu(p)
		=\frac{1}{\nu(K)},
	\end{equation*}
	and hence $\mu/\rho_\mu=\nu$.
	
	For a finite signed measure $\nu$ on $K$, define
	\begin{equation*}
		(\mathcal{A}\nu)(t)=H_\nu(t)-\nu(K), \qquad t\in(0,1).
	\end{equation*}
	Since $H_\nu=H_\mu/\rho_\mu$ and $\nu(K)=1/\rho_\mu$, the representation \eqref{eq:rewrite_I} derived in the proof of \cref{thm:SRM_optimality} yields
	\begin{equation*}
		I(m;c)=\mathcal{J}_c(\nu)
		:=\int_0^1\left[c+(\mathcal{A}\nu)(t)\right]^2\d\cT(t).
	\end{equation*}
	The set $\cV$ is convex. It therefore suffices to show that $\mathcal{J}_c$ is strictly convex on $\cV$. To this end, we first prove that $\mathcal{A}$ is injective as a map from finite signed measures on $K$ into $L^2(\d\cT)$.
	
	Using the notation from the proof of \cref{thm:SRM_optimality}, we know that $q(s)<q(t)$ whenever $0<s<t<1$, because the distribution function of $\lambda Z$ is continuous and strictly increasing. It follows that
	\begin{equation*}
		\cT(t)-\cT(s)
		=\E\left[\lambda^2\left\{
		\Phi\left(\frac{q(t)}{\lambda}\right)
		-\Phi\left(\frac{q(s)}{\lambda}\right)
		\right\}\right]>0.
	\end{equation*}
	Thus, the Stieltjes measure $\d\cT$ assigns positive mass to every nonempty open subinterval of $(0,1)$.
	
	Suppose now that $\delta$ is a finite signed measure on $K$ such that $\mathcal{A}\delta=0$ in $L^2(\d\cT)$. For $t\in(1-u,1)$, we have $H_\delta(t)=0$, so that
	\begin{equation*}
		(\mathcal{A}\delta)(t)=-\delta(K).
	\end{equation*}
	Since $\d\cT((1-u,1))>0$, we must have $\delta(K)=0$. Therefore, $H_\delta=0$ $\d\cT$-almost everywhere on $(0,1)$.
	
	Let $\xi$ be the finite signed measure on $K$ defined by $\d\xi(p)=p^{-1}\d\delta(p)$. Then
	\begin{equation*}
		H_\delta(t)=\xi([t,1]), \qquad t\in(0,1).
	\end{equation*}
	The function $t\mapsto\xi([t,1])$ is left-continuous. If it were nonzero at some $t_0\in(0,1)$, it would remain bounded away from zero on an interval immediately to the left of $t_0$. This is impossible because $H_\delta=0$ $\d\cT$-almost everywhere and $\d\cT$ is non-zero on every non-empty open interval. Hence, $H_\delta(t)=0$ for every $t\in(0,1)$. In particular,
	\begin{equation*}
		\xi([a,b))=H_\delta(a)-H_\delta(b)=0,
		\qquad 0<a<b<1,
	\end{equation*}
	and $\xi(\{1-u\})=H_\delta(1-u)=0$. The half-open intervals in the preceding display generate the Borel $\sigma$-field on $[u,1-u)$; together with the last identity, they therefore imply that $\xi=0$ on $K$. Consequently, $\delta=0$, and $\mathcal{A}$ is injective.
	
	Finally, let $\nu_1,\nu_2\in\cV$ be distinct and let $\theta\in(0,1)$. By the linearity of $\mathcal{A}$,
	\begin{align*}
		\mathcal{J}_c\big(\theta\nu_1+(1-\theta)\nu_2\big)
		=\, &\theta\mathcal{J}_c(\nu_1)
		+(1-\theta)\mathcal{J}_c(\nu_2)\\
		&-\theta(1-\theta)
		\int_0^1\left[\mathcal{A}(\nu_1-\nu_2)(t)\right]^2\d\cT(t).
	\end{align*}
	The final integral is strictly positive by the injectivity of $\mathcal{A}$. Hence, $\mathcal{J}_c$ is strictly convex on $\cV$. Since a minimizer of $I(m;c)$ over $\cuP(K)$ exists, the two changes of measure above imply that this minimizer must be unique. This completes the proof of \cref{prop:unique_minimizer}.
\end{proof}

\subsection{Proof of \cref{thm:continuous_optimal_measure}}
\begin{proof}
	Fix $c \in \R$, we already know that a minimizer $m_{*,u}^c$ of $I (m; c)$ exists. Denote
	\begin{equation*}
		m_{\cT} = \int_{0}^{1} \frac{\cT(1 - \alpha)}{1 - \alpha} \d m_{*,u}^c (\alpha), \quad m_{\cU} = \int_{0}^{1} \frac{\cU(1 - \alpha)}{1 - \alpha} \d m_{*,u}^c (\alpha).
	\end{equation*}
	Then, by \cref{eq:simplified_I_m_c} we know that
	\begin{equation*}
		\begin{split}
			& m_{*,u}^c = \arg \min_{m \in \cuP([u, 1-u])} \left\{ \int_{[0, 1]^2} \frac{\cT\bigl(1-(\alpha_1 \vee \alpha_2)\bigr)}{(1 - \alpha_1) (1 - \alpha_2)} \, \d m (\alpha_1) \d m (\alpha_2) \right\}, \\ & \mbox{subject to} \,\, \int_{0}^{1} \frac{\cT(1 - \alpha)}{1 - \alpha} \d m (\alpha) = m_{\cT}, \,\, \int_{0}^{1} \frac{\cU(1 - \alpha)}{1 - \alpha} \d m (\alpha) = m_{\cU}.
		\end{split}
	\end{equation*}
	Using calculus of variations and the method of Lagrange multipliers, we know that there exist constants $\nu_{\cT}, \nu_{\cU}, \nu_0 \in \R$, such that, defining
	\begin{equation*}
		\Gamma_* (x) = \, \int_{0}^{1} \frac{\cT\bigl(1-(x \vee y)\bigr)}{(1 - x) (1 - y)} \, \d m_{*,u}^c (y) + \nu_{\cT} \frac{\cT(1 - x)}{1 - x} + \nu_{\cU} \frac{\cU(1 - x)}{1 - x},
	\end{equation*}
	we have
	\begin{equation}\label{eq:FOC}
		\Gamma_* (x) \ge \, \nu_0, \,\, \forall x \in [u, 1-u], \quad \Gamma_* (x) = \, \nu_0, \,\, \forall x \in \supp (m_{*,u}^c).
	\end{equation}

	We first show that $m_{*,u}^c$ cannot have a point mass in $(u, 1-u)$.
	Assume by contradiction that $m_{*,u}^c (\{ x \}) > 0$ for some $x \in (u, 1-u)$. Then, using dominated convergence theorem, we get that:
	\begin{equation}\label{eq:derivative_Gamma}
		\begin{split}
			\Gamma_*' (x^{-}) = \, & \frac{1}{(1 - x)^2} \left( - (1 - x) \int_{0}^{1} \frac{\cT' (1 - x) \bone_{x > y} }{1 - y} \d m_{*,u}^c (y) + \int_{0}^{1} \frac{\cT\bigl(1-(x \vee y)\bigr)}{1 - y} \d m_{*,u}^c (y) \right) \\
			& + \nu_{\cT} \frac{\d}{\d x} \left( \frac{\cT(1 - x)}{1 - x} \right) + \nu_{\cU} \frac{\d}{\d x} \left( \frac{\cU(1-x)}{1-x} \right), \\
			\Gamma_*' (x^{+}) = \, & \frac{1}{(1 - x)^2} \left( - (1 - x) \int_{0}^{1} \frac{\cT' (1 - x) \bone_{x \ge y} }{1 - y} \d m_{*,u}^c (y) + \int_{0}^{1} \frac{\cT\bigl(1-(x \vee y)\bigr)}{1 - y} \d m_{*,u}^c (y) \right) \\
			& + \nu_{\cT} \frac{\d}{\d x} \left( \frac{\cT(1 - x)}{1 - x} \right) + \nu_{\cU} \frac{\d}{\d x} \left( \frac{\cU(1-x)}{1-x} \right),
		\end{split}
	\end{equation}
	which leads to
	\begin{equation*}
		\Gamma_*' (x^{-}) - \Gamma_*' (x^{+}) = \, m_{*,u}^c \big( \{ x \} \big) \frac{\cT'(1 - x)}{ (1 - x)^2} > 0,
	\end{equation*}
	where $\cT' (1 - x) > 0$ follows directly from our assumption $\E [1 / \lambda] < \infty$ by straightforward calculation.
	However, \cref{eq:FOC} implies that $\Gamma_*' (x^{-}) \le 0 \le \Gamma_*' (x^{+})$, a contradiction. 
	This proves that $m_{*,u}^c$ is nonatomic on $(u, 1-u)$ and can have point masses only at the endpoints $u$ and $1-u$.
	
	We then prove that $m_{*,u}^c$ must be absolutely continuous, i.e., $m_{*,u}^c$ can not have any singular component on $(u, 1-u)$. It suffices to consider any closed subinterval $[a, b] \subset (u, 1-u)$ and the restriction $m_{*,u}^c \vert_{[a, b]}$, which we still denote by $m_{*,u}^c$ for notational simplicity. Since $m_{*,u}^c$ has no atoms in $(u,1-u)$, its interior support has no isolated points: an isolated support point would necessarily carry positive mass. Hence every interior support point is an accumulation point of the support, and \eqref{eq:FOC} implies that $\Gamma_*'(x)=0$ for all $x\in\supp(m_{*,u}^c)\cap(u,1-u)$. By \cref{eq:derivative_Gamma}, we know that
	\begin{equation*}
		\nu_0 = \, \cT' (1 - x) \cdot \int_{0}^{x} \frac{\d m_{*,u}^c (y)}{1 - y} + \nu_{\cT} \cT' (1 - x) + \nu_{\cU} \cU' (1 - x),
	\end{equation*}
	which leads to
	\begin{equation}\label{eq:density_condition}
		\int_{0}^{x} \frac{ \d m_{*,u}^c (y) }{1 - y} = \,  \frac{\nu_0}{\cT' (1 - x)} - \nu_{\cT} - \nu_{\cU} \frac{\cU' (1 - x)}{\cT' (1 - x)} := V (x), \quad \forall x \in \supp \big( m_{*,u}^c \big).
	\end{equation}
	To show that $m_{*,u}^c$ is absolutely continuous with respect to the Lebesgue measure, we will prove that for any $\veps > 0$, there exists some $\delta > 0$, such that for any disjoint intervals $\{ [x_i, y_i] \}_{i=1}^{n}$:
	\begin{equation}\label{eq:def_ac}
		\sum_{i=1}^{n} (y_i - x_i) < \delta \implies \sum_{i=1}^{n} m_{*,u}^c \big( [x_i, y_i] \big) < \veps.
	\end{equation}
	%
	%
	%
	
	By definition, there exists a constant $C_u > 0$, such that the function $V(x)$ is $C_u$-Lipschitz on $[u, 1-u]$. Further, we can assume without loss of generality that $x_i, y_i \in \supp (m_{*,u}^c)$ for each $i \in [n]$. \cref{eq:density_condition} then implies that  
	\begin{equation*}
		C_{u} (y_i - x_i) \ge \, V (y_i) - V (x_i) = \int_{x_i}^{y_i} \frac{ \d m_{*,u}^c (y) }{1 - y} \ge m_{*,u}^c \left( [x_i, y_i] \right).
	\end{equation*}
	The above calculation further implies that
	\begin{equation*}
		\sum_{i=1}^{n} m_{*,u}^c \big( [x_i, y_i] \big) \le C_{u} \sum_{i=1}^{n} (y_i - x_i) < C_{u} \delta.
	\end{equation*}
	Of course, we can choose $\delta$ to be so small that $C_{u} \delta \le \veps$, which finally leads to \cref{eq:def_ac}. This completes the proof of \cref{thm:continuous_optimal_measure}.
\end{proof}

\subsection{Proof of \cref{thm:SRM_optimality}}
\begin{proof}
	We assume without loss of generality that $\lambda > 0$ almost surely, and that $\E [\lambda^4] < \infty$. Otherwise, $J (c) = + \infty$ and our claim automatically holds. For future convenience, we denote
	$$
	L_2 = \E [\lambda^2], \quad L_4 = \E [\lambda^4],
	$$
	and write $q (\alpha)$ as a shorthand for $q_{\alpha} (\lambda Z)$. Denote $ F(x):=\mathbb P(X\le x) $ as the CDF of $X=\lambda Z$. Since $\lambda>0$ almost surely and $Z$ is independent of $\lambda$, we know that
	\[
	F(x)=\E \left[\Phi\left(\frac{x}{\lambda}\right)\right].
	\]
	As a consequence, for any $ \alpha \in (0,1)$, we have
	\begin{equation}\label{eq:F_and_Phi}
		\alpha = F (q(\alpha)) = \E \left[\Phi\left(\frac{q(\alpha)}{\lambda}\right)\right].
	\end{equation}
	By symmetry, we also have $q(1 - \alpha) = - q(\alpha)$.
	
	According to the definition of expectation, we get that
	\begin{equation}\label{eq:property_q}
		\begin{split}
			& \int_0^1 q(\alpha) \d \alpha = \mathbb E[X] = \mathbb E[\lambda]\mathbb E[Z] = 0, \\
			& \int_0^1 q(\alpha)^2 \d \alpha = \mathbb E[X^2] = \mathbb E[\lambda^2]\mathbb E[Z^2] = L_2.
		\end{split}
	\end{equation}
	Further, computing the first-order derivative of $\cU$ and using \cref{eq:F_and_Phi} gives that
	\begin{equation}\label{eq:representation_U}
		\cU(\alpha) = \int_{0}^{\alpha} -q(t) \d t.
	\end{equation}
	By definition of $\cT$, we deduce that for any integrable function $\psi$:
	\[
	\int_0^1 \psi(q(\alpha)) \d\cT(\alpha) = \mathbb E[\lambda^2\psi(X)] = \mathbb E[\lambda^2\psi(\lambda Z)].
	\]
	In particular,
	\begin{align*}
		& \int_0^1 \d\cT(\alpha)=\mathbb E[\lambda^2]=L_2, \\
		& \int_0^1 q(\alpha) \d\cT(\alpha) = \mathbb E[\lambda^2 X] = \mathbb E[\lambda^3 Z] = \mathbb E[\lambda^3]\mathbb E[Z] = 0, \\
		& \int_0^1 q(\alpha)^2 \d\cT(\alpha) = \mathbb E[\lambda^2 X^2] = \mathbb E[\lambda^2(\lambda Z)^2] = \mathbb E[\lambda^4]\mathbb E[Z^2] = L_4.
	\end{align*}
	
	Now for any $m \in \cuP([0,1])$, let $\mu$ be its pushforward under the change of variable $p = 1 - \alpha$, and define for $t \in (0, 1)$:
	\[
	H_\mu(t) :=
	\int_{t}^{1} \frac{1}{p} \d \mu(p).
	\]
	We can therefore recast the three terms appearing in $I(m;c)$ in terms of $\mu$. First, using the representation~\eqref{eq:representation_U} and Fubini's theorem, we get
	\begin{equation*}
		\begin{split}
			& \rho_\mu = \int_0^1 \frac{\cU(p)}{p} \d \mu(p) = \int_0^1 \frac{1}{p}
			\left(\int_0^p -q(t) \d t \right) \d \mu(p) \\
			= \, & \int_0^1 -q(t) \left(
			\int_{t}^{1} \frac{1}{p} \d \mu(p)
			\right) \d t = \int_0^1 -q(t) H_\mu(t) \d t.
		\end{split}
	\end{equation*}
	Second, noting that $\cT(0) = 0$, direct calculation shows that
	\[
	\int_0^1 \frac{\cT(p)}{p} \d \mu(p) = \int_0^1
	\frac{1}{p} \left( \int_{0}^{p} \d\cT(t) \right)
	\d \mu(p) = \int_0^1 H_\mu(t) \d\cT(t).
	\]
	Third, using the identity
	\[
	\cT(p\wedge s)
	=
	\int_0^1 \mathbf 1_{\{t\le p\}}\mathbf 1_{\{t\le s\}} \d\cT(t),
	\]
	we obtain that
	\[
	\iint_{[0,1]^2}
	\frac{\cT(p\wedge s)}{ps}
	\d \mu(p) \d \mu(s)
	=
	\int_0^1
	\left(
	\int_{[t,1]}\frac{1}{p} \d \mu(p)
	\right)^2
	\d\cT(t) = \int_0^1 H_\mu(t)^2 \d\cT(t).
	\]
	Based on these calculations, define $\widetilde I(\mu;c):=I(m;c)$, where $m$ is the inverse pushforward of $\mu$ under $p=1-\alpha$. Then
	\begin{equation}\label{eq:rewrite_I}
		\widetilde I(\mu;c)
		=
		\frac{1}{\rho_\mu^2}
		\int_0^1
		\left[
		H_\mu(t)-\bigl(1-c\rho_\mu\bigr)
		\right]^2
		\d\cT(t) = \int_0^1
		\left[
		c+\frac{H_\mu(t)-1}{\rho_\mu}
		\right]^2
		\d\cT(t).
	\end{equation}
	
	To complete the proof of \cref{thm:SRM_optimality}, it suffices to construct a sequence of probability measures $\mu_n$ such that $\widetilde I(\mu_n;c) \to J(c)$ as $n \to \infty$. To this end, we define for $p \in (0, 1)$:
	\[
	s_0(p):=-\frac{q(p)}{L_2}.
	\]
	For each $n\in\mathbb N$, define the truncated version of $s_0$ on $[-n, n]$ as:
	\[
	a_n(p) := \max\{-n,\min\{s_0(p),n\}\}.
	\]
	Since $q$ is nondecreasing, we know that $s_0=-q/L_2$ is nonincreasing, and hence $a_n$ is also nonincreasing. Define
	\[
	H_n(p):=1+\frac{a_n(p)}{n},
	\]
	we know that $H_n$ is nonincreasing and bounded: $0\le H_n(p)\le 2$. Further since $a_n$ is an odd function, we have $\int_0^1 H_n(p) \d p=1$.
	
	Since $H_n$ is nonnegative and nonincreasing, there exists a finite Borel measure $\xi_n$ on $[0,1]$, such that
	\[
	\xi_n([p,1])=H_n(p),
	\qquad \forall p \in (0,1).
	\]
	Define the measure $\mu_n$ on $[0,1]$ by
	\[
	\d \mu_n(p) := p \d \xi_n (p),
	\]
	we claim that $\mu_n$ is a probability measure. Indeed, by Fubini's theorem,
	\[
	\mu_n([0,1])
	=
	\int_{0}^{1} p \d \xi_n (p)
	=
	\int_{0}^{1}
	\left(\int_0^p \d t\right)
	\d \xi_n(p)
	=
	\int_0^1 \xi_n([t,1]) \d t
	=
	\int_0^1 H_n(t) \d t
	=
	1.
	\]
	Further, it is easy to verify that $H_n = H_{\mu_n}$, as
	\[
	H_{\mu_n} (t) = \int_{t}^{1} \frac{1}{p} \d \mu_n(p)
	=
	\int_{t}^{1} \d \xi_n(p)
	=
	\xi_n([t,1])
	=
	H_n(t).
	\]
	Therefore, the representation of $\widetilde I$ in \cref{eq:rewrite_I} yields that
	\[
	\widetilde I(\mu_n;c)
	=
	\int_0^1
	\left[
	c+\frac{H_n(t)-1}{\rho_n}
	\right]^2
	\d\cT(t), \qquad \rho_n
	:=
	\int_0^1 -q(t)H_n(t) \d t.
	\]
	
	By direct calculation, we know that
	\[
	\widetilde I(\mu_n;c)
	=
	\int_0^1
	\left(
	c+\frac{a_n(t)}{b_n}
	\right)^2
	\d\cT(t), \qquad b_n:=\int_0^1 -q(t)a_n(t) \d t.
	\]
	We now compute the limit of this expression. Since $a_n(t)\to s_0(t)$ pointwise and
	\[
	|a_n(t)|\le |s_0(t)|,
	\]
	applying dominated convergence theorem yields that as $n \to \infty$:
	\begin{align*}
		b_n = \int_0^1 -q(t)a_n(t) \d t = L_2\int_0^1 s_0(t)a_n(t) \d t \to L_2\int_0^1 s_0(t)^2 \d t = \frac{1}{L_2}\int_0^1 q(t)^2 \d t = 1,
	\end{align*}
	where the last equality follows from \cref{eq:property_q}. Using again the dominated convergence theorem, we deduce that as $n \to \infty$:
	\[
	\widetilde I(\mu_n;c)
	=
	\int_0^1
	\left(
	c+\frac{a_n(t)}{b_n}
	\right)^2
	\d\cT(t)
	\to
	\int_0^1
	\left(
	c+s_0(t)
	\right)^2
	\d\cT(t).
	\]
	Since $s_0(t)=-q(t)/L_2$, we obtain
	\begin{align*}
		& \int_0^1 \left( c+s_0(t) \right)^2 \d\cT(t) = \int_0^1 \left( c-\frac{q(t)}{L_2} \right)^2 \d\cT(t) \\
		= \, & c^2\int_0^1 \d\cT(t) - \frac{2c}{L_2}\int_0^1 q(t) \d\cT(t) + \frac{1}{L_2^2}\int_0^1 q(t)^2 \d\cT(t).
	\end{align*}
	Using the identities established above,
	\[
	\int_0^1 \d\cT(t)=L_2,
	\qquad
	\int_0^1 q(t) \d\cT(t)=0,
	\qquad
	\int_0^1 q(t)^2 \d\cT(t)=L_4,
	\]
	we finally conclude that
	\[
	\int_0^1
	\left(
	c-\frac{q(t)}{L_2}
	\right)^2
	\d\cT(t)
	=
	c^2L_2+\frac{L_4}{L_2^2} = J (c).
	\]
	This completes the proof.
\end{proof}

\subsection{Proof of \cref{prop:optimal_measure_gaussian}}
\begin{proof}
	We already know that a minimizer $m_{*,u}$ exists. It suffices to verify the first-order condition for minimizing $I$. Namely, there exists a constant $C > 0$ such that:
	\begin{equation*}
		\int_{[u, 1-u]} \frac{\alpha \wedge \beta}{1-(\alpha \wedge \beta)} \d m_{*,u}(\beta) = C \frac{\cU(1-\alpha)}{1-\alpha}, \quad \forall \alpha \in [u, 1-u].
	\end{equation*}
	To this end, we define the integral transform $\mathcal{I}(\alpha) = \int_{[u, 1-u]} \frac{\alpha \wedge \beta}{1-(\alpha \wedge \beta)} \d m_{*,u}(\beta)$. Splitting the integral at $\alpha$, we have:
	\begin{equation*}
		\mathcal{I}(\alpha) = \int_{[u, \alpha)} \frac{\beta}{1-\beta} \d m_{*,u}(\beta) + \frac{\alpha}{1-\alpha} \int_{[\alpha, 1-u]} \d m_{*,u}(\beta).
	\end{equation*}
	Differentiating with respect to $\alpha$, we obtain that
	\begin{equation*}
		\mathcal{I}'(\alpha) = \frac{\d}{\d\alpha}\left( \frac{\alpha}{1-\alpha} \right) \int_{[\alpha, 1-u]} \d m_{*,u}(\beta) = \frac{1}{(1-\alpha)^2} m_{*,u}([\alpha, 1-u]), \quad \forall \alpha \in (u, 1 - u).
	\end{equation*}
	By definition of $m_{*,u}$, we get that
	\begin{equation*}
		m_{*,u}([\alpha, 1-u]) = \frac{u}{\cU(1-u)} \left[ \cU(1-\alpha) - (1-\alpha)\cU'(1-\alpha) \right].
	\end{equation*}
	Substituting this into our expression for $\mathcal{I}'(\alpha)$ yields:
	\begin{equation*}
		\mathcal{I}'(\alpha) = \frac{u}{\cU(1-u)} \frac{\cU(1-\alpha) - (1-\alpha)\cU'(1-\alpha)}{(1-\alpha)^2} = \frac{u}{\cU(1-u)} \frac{\d}{\d\alpha} \left[ \frac{\cU(1-\alpha)}{1-\alpha} \right].
	\end{equation*}
	Integrating this relation, we thus obtain that
	\begin{equation*}
		\mathcal{I}(\alpha) = \frac{u}{\cU(1-u)} \frac{\cU(1-\alpha)}{1-\alpha} + c_0.
	\end{equation*}
	To determine the integration constant $c_0$, we evaluate $\mathcal{I}(\alpha)$ at the lower boundary $\alpha = u$. Because $m_{*,u}$ is a probability measure, its total mass is 1:
	\begin{equation*}
		\mathcal{I}(u) = \frac{u}{1-u} \int_{[u, 1-u]} \d m_{*,u}(\beta) = \frac{u}{1-u}.
	\end{equation*}
	Simultaneously, evaluating the right-hand side of our integrated expression at $u$ yields:
	\begin{equation*}
		\frac{u}{\cU(1-u)} \frac{\cU(1-u)}{1-u} = \frac{u}{1-u}.
	\end{equation*}
	This implies $c_0 = 0$. Therefore, $\mathcal{I}(\alpha) = \frac{u}{\cU(1-u)} \frac{\cU(1-\alpha)}{1-\alpha}$ for all $\alpha \in [u, 1-u]$, which satisfies the first-order optimality condition with constant $C = \frac{u}{\cU(1-u)} > 0$. Because the functional is strictly convex, this condition is both necessary and sufficient, proving $m_{*,u}$ is the unique minimizer. Finally, the expression for the minimum value $I(m_{*,u}; c)$ follows by direct calculation. This completes the proof.
\end{proof}

\section{Proofs for Section~\ref{sec:feasible_rule}}

\subsection{Proof of \cref{prop:plugin_optimal_measure}}
\begin{proof}
	We first establish the consistency of $\hat{c}$. Under the normalization $\E[\lambda^2]=1$, we have $\E[\rr]=\bmu$ and $\var(\rr)=\bSigma$. Since $\E[\|\rr\|_2^2]<\infty$, the strong law of large numbers gives
	\begin{equation}\label{eq:consistency_mu_sigma_feasible}
		\hat{\bmu}\stackrel{\mathrm{a.s.}}{\longrightarrow}\bmu,
		\qquad
		\hat{\bSigma}\stackrel{\mathrm{a.s.}}{\longrightarrow}\bSigma.
	\end{equation}
	Since $\bSigma$ is positive definite, \eqref{eq:consistency_mu_sigma_feasible} implies that $\hat{\bSigma}$ is positive definite with probability tending to one. As usual, quantities involving its inverse may be defined arbitrarily on the complementary event without affecting any of the conclusions below.
	Write
\begin{equation*}
		\hat{\bmu}^{(1)}=[\hat{\bmu}\,\,\bone].
	\end{equation*}
	Since $\hat{\bSigma}=\hat{\bS}$, the sample mean-variance portfolio $\hw_{\rm MV}(\mu_0)$ defined in \eqref{eq:empirical_mvo_risky} is
	\begin{equation*}
		\hw_{\rm MV}(\mu_0)
		=
		\hat{\bSigma}^{-1}\hat{\bmu}^{(1)}
		\left(
		\hat{\bmu}^{(1)\top}\hat{\bSigma}^{-1}\hat{\bmu}^{(1)}
		\right)^{-1}
		\mu_0^{(1)}.
	\end{equation*}
	The empirical counterpart of the scalar index $c=c(\mu_0)$ is then
	\begin{equation*}
		\hat{c}
		=
		\frac{
			[1\,\,0]
			\left(
			\hat{\bmu}^{(1)\top}\hat{\bSigma}^{-1}\hat{\bmu}^{(1)}
			\right)^{-1}
			\mu_0^{(1)}
		}
		{\|\hat{\bSigma}^{1/2}\hw_{\rm MV}(\mu_0)\|_2}.
	\end{equation*}
	Because $\bSigma$ is positive definite and $\bmu^{(1)}$ has full column rank, the matrix $\bmu^{(1)\top}\bSigma^{-1}\bmu^{(1)}$ is positive definite. Moreover, $\ww_*(\mu_0)\ne\bzero$ because $\bone^\top\ww_*(\mu_0)=1$, so $\|\bSigma^{1/2}\ww_*(\mu_0)\|_2>0$. It follows from \eqref{eq:consistency_mu_sigma_feasible} that the empirical matrices are nonsingular with probability tending to one and that $\hw_{\rm MV}(\mu_0)\pto\ww_*(\mu_0)$. Another application of the continuous mapping theorem therefore gives $\hat{c}\pto c$ in this case as well.
	
	We next study the quantities depending on the radial distribution. We begin with a deterministic continuity argument. Let $Q_n,Q\in\mathscr{P}_2((0,\infty))$ satisfy $W_2(Q_n,Q)\to0$. Let $\Lambda_n\sim Q_n$ and $\Lambda\sim Q$. By the coupling characterization of the $W_2$ metric, these random variables may be defined on a common probability space so that
	\begin{equation}\label{eq:radial_L2_coupling}
		\E\left[|\Lambda_n-\Lambda|^2\right]\to0.
	\end{equation}
	Denote the distribution functions of $\Lambda_n Z$ and $\Lambda Z$ by $F_n$ and $F$, respectively. Thus,
	\begin{equation*}
		F_n(x)=\E\left[\Phi\left(\frac{x}{\Lambda_n}\right)\right],
		\qquad
		F(x)=\E\left[\Phi\left(\frac{x}{\Lambda}\right)\right].
	\end{equation*}
	By \eqref{eq:radial_L2_coupling}, $\Lambda_n\to\Lambda$ in probability. For each fixed $x\in\R$, the function $a\mapsto\Phi(x/a)$ is bounded and continuous on $(0,\infty)$. It follows that $\Phi(x/\Lambda_n)\to\Phi(x/\Lambda)$ in probability; boundedness then implies convergence in $L^1$, and hence $F_n(x)\to F(x)$. Moreover, since $\Lambda>0$ almost surely, $F$ is continuous and strictly increasing. Indeed, for any $x<y$,
	\begin{equation*}
		F(y)-F(x)
		=
		\E\left[
		\Phi\left(\frac{y}{\Lambda}\right)
		-
		\Phi\left(\frac{x}{\Lambda}\right)
		\right]>0.
	\end{equation*}
	Since the pointwise limit $F$ is continuous, the convergence of the distribution functions is uniform:
	\begin{equation}\label{eq:uniform_cdf_radial}
		\sup_{x\in\R}|F_n(x)-F(x)|\to0.
	\end{equation}

	Let $q_n(\alpha)$ and $q(\alpha)$ denote the $\alpha$-quantiles of $\Lambda_nZ$ and $\Lambda Z$, respectively. Fix $K=[u,1-u]$. The function $q$ is continuous on $K$. For any $\veps>0$, strict monotonicity of $F$ and compactness of $K$ imply that
	\begin{equation*}
		\delta_{\veps}
		:=
		\min_{\alpha\in K}
		\min\left\{
		\alpha-F(q(\alpha)-\veps),
		F(q(\alpha)+\veps)-\alpha
		\right\}>0.
	\end{equation*}
	By \eqref{eq:uniform_cdf_radial}, for all sufficiently large $n$ and every $\alpha\in K$,
	\begin{equation*}
		F_n(q(\alpha)-\veps)<\alpha
		<F_n(q(\alpha)+\veps).
	\end{equation*}
	Since $Q_n((0,\infty))=1$, the function $F_n$ is also continuous and strictly increasing. The preceding display therefore implies
	\begin{equation}\label{eq:uniform_quantile_radial}
		\sup_{\alpha\in K}|q_n(\alpha)-q(\alpha)|\to0.
	\end{equation}
	
	Define, for $x\in\R$,
	\begin{equation*}
		\tau_n(x)
		=
		\E\left[\Lambda_n^2\Phi\left(\frac{x}{\Lambda_n}\right)\right],
		\qquad
		\upsilon_n(x)
		=
		\E\left[\Lambda_n\phi\left(\frac{x}{\Lambda_n}\right)\right],
	\end{equation*}
	and let $\tau$ and $\upsilon$ be the corresponding functions defined using $\Lambda$. By the Cauchy--Schwarz inequality,
	\begin{equation*}
		\E\left[|\Lambda_n^2-\Lambda^2|\right]
		\le
		\|\Lambda_n-\Lambda\|_2
		\left(\|\Lambda_n\|_2+\|\Lambda\|_2\right)
		\to0.
	\end{equation*}
	Consequently, $\{\Lambda_n^2:n\ge1\}$ is uniformly integrable and
	\begin{equation*}
		\E[\Lambda_n^2]\to\E[\Lambda^2].
	\end{equation*}
	Since
	\begin{equation*}
		0\le
		\Lambda_n^2\Phi\left(\frac{x}{\Lambda_n}\right)
		\le\Lambda_n^2,
		\qquad
		0\le
		\Lambda_n\phi\left(\frac{x}{\Lambda_n}\right)
		\le\frac{\Lambda_n}{\sqrt{2\pi}},
	\end{equation*}
	the continuous mapping theorem gives convergence in probability of both integrands to their counterparts defined using $\Lambda$. The first envelope above is uniformly integrable, while $\{\Lambda_n:n\ge1\}$ is uniformly integrable because it is bounded in $L^2$. Vitali's convergence theorem therefore yields
	\begin{equation}\label{eq:pointwise_T_U_radial}
		\tau_n(x)\to\tau(x),
		\qquad
		\upsilon_n(x)\to\upsilon(x)
		\qquad
		\mbox{for every }x\in\R.
	\end{equation}
	The convergence is uniform on compact subsets of $\R$. To see this, the mean value theorem gives
	\begin{align*}
		|\tau_n(x)-\tau_n(y)|
		&\le
		\frac{\E[\Lambda_n]}{\sqrt{2\pi}}|x-y|,\\
		|\upsilon_n(x)-\upsilon_n(y)|
		&\le
		\|\phi'\|_{\infty}|x-y|.
	\end{align*}
	The first Lipschitz constants are uniformly bounded by \eqref{eq:radial_L2_coupling}. Hence, \eqref{eq:pointwise_T_U_radial}, together with a finite covering argument on any compact interval, gives uniform convergence on that interval.
	
	By definition,
	\begin{equation*}
		\cT_{Q_n}(\alpha)=\tau_n(q_n(\alpha)),
		\qquad
		\cU_{Q_n}(\alpha)=\upsilon_n(q_n(\alpha)),
	\end{equation*}
	with analogous identities for $\cT_Q$ and $\cU_Q$. The uniform quantile convergence in \eqref{eq:uniform_quantile_radial} places all the relevant quantiles in a common compact interval. Combining the preceding uniform convergence with the Lipschitz bounds therefore yields
	\begin{equation}\label{eq:uniform_T_U_radial}
		\sup_{\alpha\in K}|\cT_{Q_n}(\alpha)-\cT_Q(\alpha)|\to0,
		\qquad
		\sup_{\alpha\in K}|\cU_{Q_n}(\alpha)-\cU_Q(\alpha)|\to0.
	\end{equation}
	
	We now apply this deterministic conclusion to the random measure $\hat{P}_{\lambda}$. To justify this step formally, consider an arbitrary subsequence. Since $W_2(\hat{P}_{\lambda},P_{\lambda})\pto0$, there exists a further subsequence along which $W_2(\hat{P}_{\lambda},P_{\lambda})\to0$ almost surely. Applying \eqref{eq:uniform_quantile_radial} and \eqref{eq:uniform_T_U_radial} pathwise along this further subsequence shows that
	\begin{equation}\label{eq:plugin_T_U_consistency}
		\sup_{\alpha\in K}|\hat{\cT}(\alpha)-\cT(\alpha)|\pto0,
		\qquad
		\sup_{\alpha\in K}|\hat{\cU}(\alpha)-\cU(\alpha)|\pto0,
	\end{equation}
	and
	\begin{equation}\label{eq:plugin_L2_consistency}
		\hat{L}_2
		:=
		\hat{\E}_{\lambda}[\lambda^2]
		\pto
		\E[\lambda^2]
		=:L_2.
	\end{equation}
	The conclusion follows from the subsequence characterization of convergence in probability. Notice that this argument does not require $\hat{P}_{\lambda}$ to be independent of $(\hat{\bmu},\hat{\bSigma})$.
	
	It remains to translate \eqref{eq:plugin_T_U_consistency} into uniform convergence of the objective function. For $\alpha,\beta\in K$, define
	\begin{equation*}
		r(\alpha)
		=
		\frac{\cU(1-\alpha)}{1-\alpha},
		\qquad
		s(\alpha)
		=
		\frac{\cT(1-\alpha)}{1-\alpha},
	\end{equation*}
	and
	\begin{equation*}
		k(\alpha,\beta)
		=
		\frac{\cT(1-\max\{\alpha,\beta\})}
		{(1-\alpha)(1-\beta)}.
	\end{equation*}
	Let $\hat{r},\hat{s}$, and $\hat{k}$ be their plug-in counterparts. Since $1-\alpha\ge u$ on $K$, \eqref{eq:plugin_T_U_consistency} implies
	\begin{equation}\label{eq:plugin_kernel_consistency}
		\|\hat{r}-r\|_{\infty}\pto0,
		\qquad
		\|\hat{s}-s\|_{\infty}\pto0,
		\qquad
		\|\hat{k}-k\|_{\infty}\pto0,
	\end{equation}
	where the last norm is taken over $K^2$.
	
	For $m\in\cuP(K)$, write
	\begin{equation*}
		R(m)=\int_K r(\alpha)\d m(\alpha),
		\qquad
		S(m)=\int_K s(\alpha)\d m(\alpha),
	\end{equation*}
	and
	\begin{equation*}
		H(m)
		=
		\int_{K^2}k(\alpha,\beta)
		\d m(\alpha)\d m(\beta).
	\end{equation*}
	Define $\hat{R}(m),\hat{S}(m)$, and $\hat{H}(m)$ analogously. We then have
	\begin{align}\label{eq:uniform_integrated_kernel_consistency}
		\sup_{m\in\cuP(K)}|\hat{R}(m)-R(m)|
		&\le\|\hat{r}-r\|_{\infty},\nonumber\\
		\sup_{m\in\cuP(K)}|\hat{S}(m)-S(m)|
		&\le\|\hat{s}-s\|_{\infty},\\
		\sup_{m\in\cuP(K)}|\hat{H}(m)-H(m)|
		&\le\|\hat{k}-k\|_{\infty}.\nonumber
	\end{align}
	Moreover, $R(m)=\rho_m(\lambda Z)$. Since $\lambda>0$ almost surely, $\cU(p)>0$ for every $p\in(0,1)$. The continuity of $\cU$ and compactness of $K$ therefore give
	\begin{equation*}
		r_0:=\min_{\alpha\in K}r(\alpha)>0.
	\end{equation*}
	It follows that $R(m)\ge r_0$ for every $m\in\cuP(K)$. By \eqref{eq:plugin_kernel_consistency}, with probability tending to one,
	\begin{equation}\label{eq:uniform_plugin_risk_lower_bound}
		\inf_{m\in\cuP(K)}\hat{R}(m)\ge\frac{r_0}{2}.
	\end{equation}
	
	Recall from \eqref{eq:simplified_I_m_c} that
	\begin{equation*}
		I(m;c)
		=
		\Psi\big(R(m),S(m),H(m),L_2,c\big),
	\end{equation*}
	where
	\begin{equation*}
		\Psi(x,y,z,\ell,a)
		=
		\frac{z-2(1-ax)y+(1-ax)^2\ell}{x^2}.
	\end{equation*}
	The same representation holds for $\hat{I}(m;\hat{c})$ with the plug-in quantities. In view of \eqref{eq:plugin_L2_consistency}, \eqref{eq:uniform_integrated_kernel_consistency}, and \eqref{eq:uniform_plugin_risk_lower_bound}, all the arguments of $\Psi$ lie, with probability tending to one, in a fixed compact subset of its domain on which $x$ is bounded away from zero. The function $\Psi$ is uniformly continuous on this set. Together with $\hat{c}\pto c$, this proves that
	\begin{equation}\label{eq:uniform_plugin_objective_consistency}
		\sup_{m\in\cuP(K)}
		|\hat{I}(m;\hat{c})-I(m;c)|
		\pto0.
	\end{equation}
	
	Finally, we prove the consistency of the minimizer. Endow $\cuP(K)$ with the bounded-Lipschitz metric $d_{\rm BL}$, which metrizes weak convergence. Since $K$ is compact, $\cuP(K)$ is compact under this metric. The functions $r$ and $s$ are continuous on $K$, and $k$ is continuous on $K^2$. It follows from the preceding representation and the lower bound $R(m)\ge r_0$ that $I(\cdot;c)$ is continuous on $\cuP(K)$. By \cref{prop:unique_minimizer}, its unique minimizer is $m_{*,u}^c$. The same argument applies pathwise to $\hat{I}(\cdot;\hat{c})$. Since $\cuP(K)$ is compact and $\hat{I}(\cdot;\hat{c})$ is continuous, its argmin set is nonempty and compact. Measurability of the criterion and a measurable selection theorem allow $\hat{m}$ to be chosen as a measurable minimizer. Finite-sample uniqueness is not required.
	
	Fix $\veps>0$ and define
	\begin{equation*}
		\mathcal{B}_{\veps}
		=
		\left\{
		m\in\cuP(K):d_{\rm BL}(m,m_{*,u}^c)\ge\veps
		\right\}.
	\end{equation*}
	If $\mathcal{B}_{\veps}$ is nonempty, its compactness, continuity of $I(\cdot;c)$, and uniqueness of $m_{*,u}^c$ imply that
	\begin{equation*}
		\zeta_{\veps}
		:=
		\inf_{m\in\mathcal{B}_{\veps}}
		\left\{I(m;c)-I(m_{*,u}^c;c)\right\}>0.
	\end{equation*}
	On the event
	\begin{equation*}
		\sup_{m\in\cuP(K)}
		|\hat{I}(m;\hat{c})-I(m;c)|
		<\frac{\zeta_{\veps}}{3},
	\end{equation*}
	no minimizer of $\hat{I}(\cdot;\hat{c})$ can belong to $\mathcal{B}_{\veps}$. Indeed, for every $m\in\mathcal{B}_{\veps}$,
	\begin{equation*}
		\hat{I}(m;\hat{c})
		>
		I(m_{*,u}^c;c)+\frac{2\zeta_{\veps}}{3}
		>
		\hat{I}(m_{*,u}^c;\hat{c}).
	\end{equation*}
	Consequently, \eqref{eq:uniform_plugin_objective_consistency} gives
	\begin{equation*}
		\P\left(d_{\rm BL}(\hat{m},m_{*,u}^c)\ge\veps\right)\to0.
	\end{equation*}
	This establishes $\hat{m}\stackrel{w}{\to}m_{*,u}^c$ in probability and concludes the proof of \cref{prop:plugin_optimal_measure}.
\end{proof}

\subsection{Proof of \cref{thm:feasible_asymptotic_normality}}
\begin{proof}
	Write
	\begin{equation*}
		m_0=m_{*,u}^c,
		\qquad
		\ww_*=\ww_*(\mu_0).
	\end{equation*}
	By \cref{prop:plugin_optimal_measure}, we know that
	\begin{equation}\label{eq:feasible_measure_weak_convergence}
		\hat{m}\stackrel{w}{\to}m_0
		\qquad\mbox{in probability}.
	\end{equation}
	We first establish the consistency of the empirical estimator. The consistency argument in the proof of \cref{thm:normality_general_risk} applies uniformly over $m\in\cuP([u,1-u])$. Indeed, for any compact set $K\subset\R^N$,
	\begin{align*}
		&\sup_{m\in\cuP([u,1-u])}\sup_{\ww\in K}
		\left|
		\rho_m(-\hat{r}_{\ww,T})-\rho_m(-r_{\ww})
		\right|\\
		&\qquad\le
		\sup_{\alpha\in[u,1-u]}\sup_{\ww\in K}
		\left|
		\operatorname{CVaR}_{\alpha}(-\hat{r}_{\ww,T})
		-
		\operatorname{CVaR}_{\alpha}(-r_{\ww})
		\right|
		\pto0.
	\end{align*}
	Under elliptical returns, we have
	\begin{equation}\label{eq:population_risk_uniform_measure}
		\rho_m(-r_{\ww})
		=
		-\ww^\top\bmu
		+
		\rho_m(\lambda Z)\norm{\bSigma^{1/2}\ww}_2.
	\end{equation}
	Moreover, the argument leading to \eqref{eq:uniform_plugin_risk_lower_bound} shows that
	\begin{equation*}
		\inf_{m\in\cuP([u,1-u])}\rho_m(\lambda Z)>0.
	\end{equation*}
	Thus, on the population constraint set, every objective in \eqref{eq:population_risk_uniform_measure} has the same unique minimizer $\ww_*$. The convergence of the empirical constraint set is the same as in the proof of \cref{thm:normality_general_risk}. Uniform separation of $\ww_*$ on compact sets, together with coercivity of $\norm{\bSigma^{1/2}\ww}_2$ in \eqref{eq:population_risk_uniform_measure}, therefore gives
	\begin{equation}\label{eq:feasible_weight_consistency}
		\hat{\ww}_{\hat{m}}(\mu_0)\pto\ww_*(\mu_0).
	\end{equation}
	This proves consistency.
	
	We next derive the limiting distribution of $\sqrt{T}(\hat{\ww}_{\hat{m}}(\mu_0)-\ww_*)$. As in the proof of \cref{thm:normality_general_risk}, we use the argmax theorem for random convex constraints in \cite[Theorem~5]{knight1999epi}. Define
	\begin{equation*}
		\hat{h}_T
		=
		\sqrt{T}\left(
		\hat{\ww}_{\hat{m}}(\mu_0)-\ww_*
		\right).
	\end{equation*}
	and the local objective
	\begin{equation*}
		\hat{M}_T(h)
		=
		T\left\{
		\rho_{\hat{m}}\left(
		-\hat{r}_{\ww_*+h/\sqrt{T},T}
		\right)
		-
		\rho_{\hat{m}}\left(
		-\hat{r}_{\ww_*,T}
		\right)
		\right\}.
	\end{equation*}
	By definition, $\hat{h}_T$ minimizes $\hat{M}_T$ over the rescaled empirical constraint set
	\begin{equation}\label{eq:feasible_local_constraint_set}
		\hat{\cC}_{T, h}
		=
		\left\{
		h\in\R^N:
		\bone^\top h=0,\quad
		\hat{\E}_T[\rr]^\top h
		=
		-\ww_*^\top\mathbb{G}_T(\rr)
		\right\}.
	\end{equation}
	It therefore suffices to determine the limiting process of $\hat{M}_T$ and the limiting set of $\hat{\cC}_{T, h}$.
	
	\paragraph{Reduction from $\hat{M}_T$ to $\widetilde{M}_T$.}
	We first make the same empirical-quantile reduction as in the proof of \cref{thm:normality_general_risk}. For $\alpha\in[u,1-u]$, let
	\begin{equation*}
		U_{\alpha}(\ww;\rr)
		=
		L_{\alpha}\big(\ww,z(\ww,\alpha);\rr\big),
	\end{equation*}
	and define
	\begin{equation}\label{eq:feasible_tilde_M}
		\widetilde{M}_T(h)
		=
		\int_{[u,1-u]}
		T\hat{\E}_T\left[
		U_{\alpha}\left(\ww_*+\frac{h}{\sqrt{T}};\rr\right)
		-
		U_{\alpha}(\ww_*;\rr)
		\right]
		\d\hat{m}(\alpha).
	\end{equation}
	The proof of \eqref{eq:first_convergence_M} is already uniform in $\alpha\in[u,1-u]$. Since $\hat{m}$ is a probability measure supported on $[u,1-u]$, integrating the uniform remainder against $\hat{m}$ does not increase it. Consequently, for every $R>0$,
	\begin{equation}\label{eq:feasible_quantile_reduction}
		\sup_{\norm{h}_2\le R}
		\left|
		\hat{M}_T(h)-\widetilde{M}_T(h)
		\right|
		\pto0.
	\end{equation}
	Notice that this step neither uses \eqref{eq:feasible_measure_weak_convergence} nor requires $\hat{m}$ to be independent of the observations.
	
	\paragraph{Asymptotics of $\widetilde{M}_T$.}
	Since $\hat{M}_T$ and $\widetilde{M}_T$ are convex and have the same asymptotic local behavior by \eqref{eq:feasible_quantile_reduction}, it suffices to study $\widetilde{M}_T$ over $\hat{\cC}_{T, h}$. This is the only part that does not follow directly from the proof of \cref{thm:normality_general_risk}: after integrating $U_\alpha$ against $\hat{m}$, the resulting function is data-dependent. We therefore retain $\alpha$ as an index until after taking the empirical-process limit.
	
	Let
	\begin{equation*}
		a_{\alpha}
		=
		\operatorname{CVaR}_{\alpha}(\lambda Z),
		\qquad
		s_*=\norm{\bSigma^{1/2}\ww_*}_2,
	\end{equation*}
	and set
	\begin{equation*}
		H_0
		=
		\frac{1}{s_*}
		\left(
		\bSigma
		-
		\frac{\bSigma\ww_*\ww_*^\top\bSigma}{s_*^2}
		\right).
	\end{equation*}
	Ellipticity gives, for every $\alpha\in[u,1-u]$,
	\begin{equation}\label{eq:single_cvar_elliptic_objective}
		\E\left[U_\alpha(\ww;\rr)\right]
		=
		\operatorname{CVaR}_{\alpha}(-r_{\ww})
		=
		-\ww^\top\bmu
		+
		a_\alpha\norm{\bSigma^{1/2}\ww}_2.
	\end{equation}
	It follows that
	\begin{equation*}
		\nabla_{\ww}^2\E\left[U_\alpha(\ww_*;\rr)\right]
		=
		a_\alpha H_0.
	\end{equation*}
	Let $\eta_{s,\alpha}$ denote the multiplier associated with the expected-return constraint. By \eqref{eq:w_star_no_risk_free} and the definition of $c=c(\mu_0)$, there exists a scalar $d$ such that
	\begin{equation*}
		\frac{\bSigma\ww_*}{s_*}
		=
		c\bmu+d\bone.
	\end{equation*}
	It follows from \eqref{eq:single_cvar_elliptic_objective} that
	\begin{equation*}
		\eta_{s,\alpha}=1-ca_\alpha.
	\end{equation*}
	The KKT conditions for the population problem based on $\operatorname{CVaR}_{\alpha}$ can therefore be written as
	\begin{equation}\label{eq:single_cvar_kkt}
		\E\left[\nabla_{\ww}U_\alpha(\ww_*;\rr)\right]
		+
		\eta_{s,\alpha}\bmu
		+
		\eta_{c,\alpha}\bone
		=
		0
	\end{equation}
	for some budget-constraint multiplier $\eta_{c,\alpha}\in\R$.
	
	In analogy with $\widetilde{M}_T^{(1)}$ in the proof of \cref{thm:normality_general_risk}, define its $\alpha$-indexed version by
	\begin{equation*}
		\begin{split}
			\widetilde{M}_T^{(1)}(\alpha,h)
			=
			T\bigg\{
			&\hat{\E}_T\left[
			U_\alpha\left(\ww_*+\frac{h}{\sqrt{T}};\rr\right)
			-
			U_\alpha(\ww_*;\rr)
			\right]\\
			&-
			\frac{1}{\sqrt{T}}
			h^\top\E\left[\nabla_{\ww}U_\alpha(\ww_*;\rr)\right]
			\bigg\}.
		\end{split}
	\end{equation*}
	Thus, \eqref{eq:feasible_tilde_M} can be written as
	\begin{equation}\label{eq:feasible_tilde_M_decomposition}
		\widetilde{M}_T(h)
		=
		\int_{[u,1-u]}
		\left\{
		\widetilde{M}_T^{(1)}(\alpha,h)
		+
		\sqrt{T}h^\top
		\E\left[\nabla_{\ww}U_\alpha(\ww_*;\rr)\right]
		\right\}
		\d\hat{m}(\alpha).
	\end{equation}
	We now establish the local expansion uniformly in $(\alpha,h)$. We claim that, for every $R>0$,
	\begin{equation}\label{eq:feasible_alpha_process_limit}
		\widetilde{M}_T^{(1)}(\alpha,h)
		\wto
		M^{(1)}(\alpha,h)
		:=
		h^\top\mathbb{G}
		\left(
		\nabla_{\ww}U_\alpha(\ww_*;\rr)
		\right)
		+
		\frac{a_\alpha}{2}h^\top H_0h
	\end{equation}
	in $\ell^\infty([u,1-u]\times\{h:\norm{h}_2\le R\})$.
	
	To verify this claim, first note from \eqref{eq:single_cvar_elliptic_objective} that a second-order Taylor expansion, uniformly over $\alpha\in[u,1-u]$ and bounded $h$, gives
	\begin{align*}
		T\E\left[
		U_\alpha\left(\ww_*+\frac{h}{\sqrt{T}};\rr\right)
		-
		U_\alpha(\ww_*;\rr)
		\right]
		=
		\sqrt{T}h^\top\E\left[\nabla_{\ww}U_\alpha(\ww_*;\rr)\right]
		+
		\frac{a_\alpha}{2}h^\top H_0h
		+
		o(1).
	\end{align*}
	For the empirical-process part, write $z_\alpha=z(\ww_*,\alpha)$. As in the proof of \cref{prop:no_risk_free_lowdim},
	\begin{equation}\label{eq:single_cvar_score}
		\nabla_{\ww}U_\alpha(\ww_*;\rr)
		=
		-\frac{\rr}{1-\alpha}
		\bone_{\{-r_{\ww_*}\ge z_\alpha\}}
		+
		\left(
		1-\frac{1}{1-\alpha}
		\bone_{\{-r_{\ww_*}\ge z_\alpha\}}
		\right)
		\nabla_{\ww}z(\ww_*,\alpha).
	\end{equation}
	The functions $\alpha\mapsto z_\alpha$ and $\alpha\mapsto\nabla_{\ww}z(\ww_*,\alpha)$ are continuous on $[u,1-u]$. In particular, they are bounded there. The coordinate classes generated by \eqref{eq:single_cvar_score} are weighted half-line indicator classes with an $L^2(P)$ envelope bounded by $C(1+\norm{\rr}_2)$. They are therefore $P$-Donsker. Moreover, continuity of $z_\alpha$, the absence of atoms in $-r_{\ww_*}$, and dominated convergence imply
	\begin{equation*}
		\norm{
			\nabla_{\ww}U_\alpha(\ww_*;\cdot)
			-
			\nabla_{\ww}U_\beta(\ww_*;\cdot)
		}_{L^2(P)}
		\longrightarrow0
		\qquad\mbox{as }\beta\to\alpha.
	\end{equation*}
	Hence
	\begin{equation}\label{eq:single_cvar_score_donsker}
		\mathbb{G}_T
		\left(
		\nabla_{\ww}U_\alpha(\ww_*;\rr)
		\right)
		\wto
		\mathbb{G}
		\left(
		\nabla_{\ww}U_\alpha(\ww_*;\rr)
		\right)
		\quad\mbox{in }\ell^\infty([u,1-u])^N,
	\end{equation}
	where the limiting Gaussian process has continuous sample paths.
	
	To obtain the local expansion jointly in $\alpha$, fix a compact neighborhood $\cW_*$ of $\ww_*$. By the implicit function theorem and \cref{ass:regularity_density}, the map $(\ww,\alpha)\mapsto z(\ww,\alpha)$ is continuously differentiable on $\cW_*\times[u,1-u]$, after shrinking $\cW_*$ if necessary, and its derivatives are bounded there. Consequently, the class
	\begin{equation*}
		\left\{
		U_\alpha(\ww;\cdot):
		(\ww,\alpha)\in \cW_*\times[u,1-u]
		\right\}
	\end{equation*}
	is Lipschitz in its finite-dimensional index with a square-integrable envelope of the form $C(1+\norm{\rr}_2)$. The same preservation argument used in the proof of \cref{thm:normality_general_risk} therefore shows that this class is $P$-Donsker. Its asymptotic equicontinuity, together with the uniform $L^2(P)$ differentiability of $U_\alpha(\ww;\cdot)$ at $\ww_*$, yields
	\begin{align}\label{eq:uniform_stochastic_differentiability}
		\sup_{\substack{\alpha\in[u,1-u]\\\norm{h}_2\le R}}
		\bigg|
		&\sqrt{T}\mathbb{G}_T\left[
		U_\alpha\left(\ww_*+\frac{h}{\sqrt{T}};\rr\right)
		-
		U_\alpha(\ww_*;\rr)
		\right]\nonumber\\
		&-
		h^\top\mathbb{G}_T
		\left(
		\nabla_{\ww}U_\alpha(\ww_*;\rr)
		\right)
		\bigg|
		\pto0.
	\end{align}
	Indeed, the difference quotients converge to \eqref{eq:single_cvar_score} uniformly in $L^2(P)$. Pointwise differentiability holds outside the events $\{-r_{\ww_*}=z_\alpha\}$, each of which has probability zero. To verify uniformity, first restrict to $\{\norm{\rr}_2\le M\}$. The smooth part of the remainder is uniform by continuous differentiability of $z$. The nonsmooth part can be nonzero only when $-r_{\ww_*}$ lies in an interval of length $O(T^{-1/2})$ around $z_\alpha$; the densities are uniformly bounded on the relevant compact set, so the resulting $L^2(P)$ remainder converges to zero uniformly in $\alpha$. The contribution from $\{\norm{\rr}_2>M\}$ is bounded uniformly in $T$ and $\alpha$ by a constant multiple of $\E[(1+\norm{\rr}_2)^2\bone_{\{\norm{\rr}_2>M\}}]$, which tends to zero as $M\to\infty$. Combining the deterministic Taylor expansion, \eqref{eq:single_cvar_score_donsker}, and \eqref{eq:uniform_stochastic_differentiability} proves \eqref{eq:feasible_alpha_process_limit}.
	
	We next incorporate the estimated expected-return constraint. Since $\hat{\E}_T[\rr]=\bmu+\mathbb{G}_T(\rr)/\sqrt{T}$, every $h\in\hat{\cC}_{T, h}$ satisfies
	\begin{equation}\label{eq:population_local_constraint_identity}
		\bmu^\top h
		=
		-\ww_*^\top\mathbb{G}_T(\rr)
		-
		\frac{1}{\sqrt{T}}
		h^\top\mathbb{G}_T(\rr).
	\end{equation}
	Combining \eqref{eq:single_cvar_kkt}, \eqref{eq:feasible_local_constraint_set}, and \eqref{eq:population_local_constraint_identity}, we obtain, for $h\in\hat{\cC}_{T, h}$,
	\begin{equation}\label{eq:feasible_gradient_reduction}
		\sqrt{T}h^\top
		\E\left[\nabla_{\ww}U_\alpha(\ww_*;\rr)\right]
		=
		C_T(\alpha)
		+
		\eta_{s,\alpha} h^\top\mathbb{G}_T(\rr),
		\qquad
		C_T(\alpha)
		=
		\sqrt{T}\eta_{s,\alpha}
		\ww_*^\top\mathbb{G}_T(\rr).
	\end{equation}
	The term $C_T(\alpha)$ does not depend on $h$ and therefore has no effect on the minimizer. Define
	\begin{equation*}
		\psi_\alpha(\rr)
		=
		\nabla_{\ww}U_\alpha(\ww_*;\rr)
		+
		\eta_{s,\alpha}\rr
	\end{equation*}
	and
	\begin{equation*}
		\overline{M}_T^{(1)}(\alpha,h)
		=
		\widetilde{M}_T^{(1)}(\alpha,h)
		+
		\eta_{s,\alpha} h^\top\mathbb{G}_T(\rr).
	\end{equation*}
	Since $\alpha\mapsto a_\alpha$ is continuous, the class $\{\psi_\alpha:\alpha\in[u,1-u]\}$ is also $P$-Donsker and $L^2(P)$-continuous. More precisely, adjoining the finitely many coordinate functions of $\rr$ to the score class in \eqref{eq:single_cvar_score_donsker}, multiplying them by the bounded continuous coefficients $\eta_{s,\alpha}$, and taking sums preserve the Donsker property. Thus, all the empirical processes appearing below converge jointly, and
	\begin{equation}\label{eq:centered_alpha_process_limit}
		\overline{M}_T^{(1)}(\alpha,h)
		\wto
		\overline{M}^{(1)}(\alpha,h)
		:=
		h^\top\mathbb{G}(\psi_\alpha)
		+
		\frac{a_\alpha}{2}h^\top H_0h
	\end{equation}
	in $\ell^\infty([u,1-u]\times\{h:\norm{h}_2\le R\})$. The limiting process has continuous sample paths.
	
	\paragraph{Integration against the estimated measure.}
	We can now integrate against the random measure. Put
	\begin{equation*}
		\overline{M}_T(h)
		=
		\int_{[u,1-u]}
		\overline{M}_T^{(1)}(\alpha,h)
		\d\hat{m}(\alpha).
	\end{equation*}
	By \eqref{eq:feasible_tilde_M_decomposition} and \eqref{eq:feasible_gradient_reduction}, for every $h\in\hat{\cC}_{T, h}$,
	\begin{equation}\label{eq:centered_random_measure_objective}
		\widetilde{M}_T(h)
		=
		\overline{M}_T(h)
		+
		\int_{[u,1-u]} C_T(\alpha)\d\hat{m}(\alpha).
	\end{equation}
	The last term is independent of $h$. Hence $\widetilde{M}_T$ and $\overline{M}_T$ have the same minimizer over $\hat{\cC}_{T, h}$.
	Because $m_0$ is deterministic, \eqref{eq:feasible_measure_weak_convergence} and \eqref{eq:centered_alpha_process_limit} imply
	\begin{equation*}
		\left(
		\overline{M}_T^{(1)},\hat{m}
		\right)
		\wto
		\left(
		\overline{M}^{(1)},m_0
		\right).
	\end{equation*}
	No independence is needed for this joint convergence. Moreover, integration is continuous at every pair $(f,m)$ for which $f$ is continuous on the compact set $[u,1-u]\times\{h:\norm{h}_2\le R\}$ and $m\in\cuP([u,1-u])$. To see this, suppose that $\norm{f_n-f}_\infty\to0$ and $m_n\stackrel{w}{\to}m$. Then
	\begin{align*}
		&\sup_{\norm{h}_2\le R}
		\left|
		\int_{[u,1-u]} f_n(\alpha,h)\d m_n(\alpha)
		-
		\int_{[u,1-u]} f(\alpha,h)\d m(\alpha)
		\right|\\
		&\qquad\le
		\norm{f_n-f}_\infty\\
		&\qquad\quad+
		\sup_{\norm{h}_2\le R}
		\left|
		\int_{[u,1-u]} f(\alpha,h)\d(m_n-m)(\alpha)
		\right|.
	\end{align*}
	The first term converges to zero. For the second, the collection $\{f(\cdot,h):\norm{h}_2\le R\}$ is compact in $C([u,1-u])$, because $f$ is uniformly continuous on the compact product set. A finite $\veps$-net, followed by weak convergence for each member of the net, shows that the second term also converges to zero. The continuous mapping theorem therefore gives
	\begin{equation*}
		\overline{M}_T(h)
		\wto
		M(h)
		:=
		h^\top\mathbb{G}
		\left(
		\int_{[u,1-u]}\psi_\alpha\,\d m_0(\alpha)
		\right)
		+
		\frac{1}{2}h^\top H_*h
	\end{equation*}
	in $\ell^\infty(\{h:\norm{h}_2\le R\})$. Here $H_*$ is the Hessian defined in \cref{sec:low-dim}. In the present elliptical setting,
	\begin{equation*}
		H_*
		=
		\nabla_{\ww}^2\rho_{m_0}(-r_{\ww_*})
		=
		\left(
		\int_{[u,1-u]}a_\alpha\d m_0(\alpha)
		\right)H_0
		=
		\rho_{m_0}(\lambda Z)H_0.
	\end{equation*}
	
	To identify the limiting process, let $U_{m_0}$ denote the function defined in \cref{sec:low-dim} for the fixed spectral measure $m_0$, and recall that $\eta_s$ denotes the corresponding expected-return multiplier. Differentiation under the integral is justified by the square-integrable envelope above, while integration of \eqref{eq:single_cvar_kkt} identifies the multiplier. Therefore,
	\begin{align*}
		\int_{[u,1-u]}\psi_\alpha(\rr)\d m_0(\alpha)
		&=
		\nabla_{\ww}U_{m_0}(\ww_*;\rr)
		+
		\eta_s\rr,\\
		\eta_s
		&=
		\int_{[u,1-u]}\eta_{s,\alpha}\d m_0(\alpha)
		=
		1-c\rho_{m_0}(\lambda Z),
	\end{align*}
	so the function inside $\mathbb{G}$ is precisely $V_*(\rr)$ in the notation of the canonical target-return problem in \cref{sec:low-dim}. Since $\alpha\mapsto\psi_\alpha$ is continuous in $L^2(P)$, its Bochner integral is well defined. The Brownian bridge is a continuous linear map from $L^2(P)$ into $L^2$ of its underlying probability space, modulo constant functions. It follows that
	\begin{equation*}
		\int_{[u,1-u]}\mathbb{G}(\psi_\alpha)\d m_0(\alpha)
		=
		\mathbb{G}\left(
		\int_{[u,1-u]}\psi_\alpha\d m_0(\alpha)
		\right)
	\end{equation*}
	in $L^2$. Thus, $M$ is exactly the local quadratic process obtained in the proof of \cref{thm:normality_general_risk} for the fixed oracle measure $m_0$.
	
	\paragraph{Asymptotics of $\hat{\cC}_{T, h}$.}
	Finally, we derive the limiting set of $\hat{\cC}_{T, h}$. As in the corresponding step of the proof of \cref{thm:normality_general_risk}, \eqref{eq:feasible_local_constraint_set} and the multivariate central limit theorem imply that $\hat{\cC}_{T, h}$ converges jointly with the preceding processes to
	\begin{equation*}
		\cC_h
		=
		\left\{
		h\in\R^N:
		\bone^\top h=0,\quad
		\bmu^\top h
		=
		-\ww_*^\top\mathbb{G}(\rr)
		\right\}.
	\end{equation*}
	This is the same finite-dimensional constraint-set argument as in the proof of \cref{thm:normality_general_risk}; the full-column-rank assumption on $\bmu^{(1)}$ guarantees that the limiting affine constraint has the required rank. The matrix $H_*$ is positive semidefinite with null space $\operatorname{span}\{\ww_*\}$. Since $\bone^\top\ww_*=1$, it is positive definite on the tangent space
	\begin{equation*}
		\left\{
		h:\bone^\top h=0,\ \bmu^\top h=0
		\right\}.
	\end{equation*}
	Hence the limiting constrained quadratic problem has a unique minimizer and is coercive along the feasible tangent space. To connect the auxiliary process back to the original objective, set
	\begin{equation*}
		D_T
		=
		\int_{[u,1-u]}C_T(\alpha)\d\hat{m}(\alpha).
	\end{equation*}
	Equations~\eqref{eq:feasible_quantile_reduction} and \eqref{eq:centered_random_measure_objective} imply, for every $R>0$,
	\begin{equation*}
		\sup_{\substack{h\in\hat{\cC}_{T, h}\\\norm{h}_2\le R}}
		\left|
		\hat{M}_T(h)-D_T-\overline{M}_T(h)
		\right|
		\pto0.
	\end{equation*}
	The quantity $D_T$ is independent of $h$ and hence does not affect the minimizer. According to the notation of \cref{defn:asym_dist_low_dim}, the unique minimizer of $M$ over $\cC_h$ is
	\begin{equation*}
		h_{c, s}
		=
		\arg\min_{h\in\cC_h}M(h).
	\end{equation*}
	Combining the preceding arguments, we obtain
	\begin{equation*}
		(\hat{M}_T-D_T)\cdot\bone_{\hat{\cC}_{T, h}}
		+
		\infty\cdot\bone_{\hat{\cC}_{T, h}^c}
		\quad\mbox{epi-converges to}\quad
		M\cdot\bone_{\cC_h}
		+
		\infty\cdot\bone_{\cC_h^c}
	\end{equation*}
	in distribution. Invoking \cite[Theorem~5(b)]{knight1999epi}, exactly as in the proof of \cref{thm:normality_general_risk}, gives
	\begin{equation*}
		\hat{h}_T\wto h_{c, s}.
	\end{equation*}
	The expression on the right is the fixed-$m_0$ limit derived in \cref{prop:no_risk_free_lowdim}. Since $m_0=m_{*,u}^c$, it follows that
	\begin{equation*}
		\sqrt{T}\left(
		\hat{\ww}_{\hat{m}}(\mu_0)-\ww_*(\mu_0)
		\right)
		\wto
		\normal\left(
		0,
		\Sigma\left(\mu_0,\rho_{m_{*,u}^c}\right)
		\right).
	\end{equation*}
	This completes the proof.
\end{proof}

\subsection{Proof of \cref{thm:radial_npmle_consistency}}
\begin{proof}
	For a probability measure $P$ supported on $(0,\infty)$, let
	\begin{equation*}
		f_P(x)
		=
		\int_0^\infty k_N(x\mid v)\d P(v),
		\qquad x>0,
	\end{equation*}
	and define
	\begin{equation*}
		L_T(P)
		=
		\frac{1}{T}\sum_{i=1}^T\log f_P(D_i),
		\qquad
		\hat{L}_T(P)
		=
		\frac{1}{T}\sum_{i=1}^T\log f_P(\hat{D}_i).
	\end{equation*}
	Thus, $f_{P_V}$ is the density of $D=V S$, where $V\sim P_V$ and $S\sim\chi_N^2$ are independent.
	We first show that $P_V$ can be arbitrarily approximated by probability measures supported on $[a_T, b_T]$ in $W_1$ metric. Let $V\sim P_V$ and set
	\begin{equation*}
		\bar{V}_T
		=
		\left(V\vee\sqrt{a_T}\right)\wedge\sqrt{b_T},
		\qquad
		\kappa_T=\E[\bar{V}_T],
		\qquad
		V_T=\frac{\bar{V}_T}{\kappa_T}.
	\end{equation*}
	Since $V$ is integrable,
	\begin{equation*}
		\E|\bar{V}_T-V|
		\le
		\sqrt{a_T}
		+
		\E\left[V\bone_{\{V>\sqrt{b_T}\}}\right]
		\longrightarrow0.
	\end{equation*}
	Hence, $\kappa_T\to1$ and $\E[V_T]=1$. The use of the interior truncation
	interval $[\sqrt{a_T},\sqrt{b_T}]$ leaves enough slack to preserve the
	sieve support after normalization. Indeed, since $a_T\to0$, $b_T\to\infty$,
	and $\kappa_T\to1$, for all sufficiently large $T$,
	\[
		\kappa_T\le a_T^{-1/2},
		\qquad
		\kappa_T\ge b_T^{-1/2}.
	\]
	Consequently,
	\[
		a_T
		\le \frac{\sqrt{a_T}}{\kappa_T}
		\le \frac{\sqrt{b_T}}{\kappa_T}
		\le b_T.
	\]
	Thus the support of $V_T$ is contained in $[a_T,b_T]$ for all
	sufficiently large $T$. Therefore, $P_{V,T}:=\operatorname{Law}(V_T)$
	belongs to $\mathcal{P}_T$ eventually, and the coupling above gives
	\begin{equation}\label{eq:radial_sieve_approximation}
		W_1(P_{V,T},P_V)
		\le
		\E|V_T-V|
		\le
		|1-\kappa_T|+\E|\bar{V}_T-V|
		\longrightarrow0.
	\end{equation}

	We next show that
	\begin{equation}\label{eq:radial_benchmark_likelihood}
		L_T(P_{V,T})\pto M(P_V),
		\qquad
		M(P):=\E[\log f_P(D)].
	\end{equation}
	With $y=x/(2v)$, the kernel in \eqref{eq:scaled_chisquare_kernel} satisfies
	\begin{equation*}
		k_N(x\mid v)
		=
		\frac{y^{N/2}e^{-y}}{\Gamma(N/2)x}
		\le
		\frac{C_N}{x}
	\end{equation*}
	for a constant $C_N<\infty$. Choose $0<c<d<\infty$ such that $P_V([c,d])=:\eta>0$. Since $\kappa_T\to1$, both $P_V$ and, eventually, $P_{V,T}$ assign mass at least $\eta$ to $[c/2,2d]$. It follows that, for $P=P_V$ and $P=P_{V,T}$,
	\begin{equation*}
		f_P(x)
		\ge
		\frac{\eta}{2^{N/2}\Gamma(N/2)(2d)^{N/2}}
		x^{N/2-1}\exp\left(-\frac{x}{c}\right).
	\end{equation*}
	Consequently, for some $C<\infty$ independent of $T$,
	\begin{equation}\label{eq:radial_log_likelihood_envelope}
		|\log f_{P_{V,T}}(x)|
		+
		|\log f_{P_V}(x)|
		\le
		G(x)
		:=
		C\left(1+x+|\log x|\right).
	\end{equation}
	Since $D=VU$, the assumptions imply
	\begin{equation*}
		\E[D]=N,
		\qquad
		\E|\log D|
		\le
		\E|\log V|+\E|\log U|
		<\infty.
	\end{equation*}
	Thus, $\E[G(D)]<\infty$. For every $x>0$, the map $v\mapsto k_N(x\mid v)$ is bounded and continuous on $(0,\infty)$ and vanishes at both endpoints. Equation~\eqref{eq:radial_sieve_approximation} therefore gives $f_{P_{V,T}}(x)\to f_{P_V}(x)$. By dominated convergence,
	\begin{equation}\label{eq:radial_expected_likelihood_convergence}
		\E[\log f_{P_{V,T}}(D)]
		\longrightarrow
		M(P_V).
	\end{equation}

	We next apply the triangular-array weak law of large numbers to the quantity $L_T(P_{V,T})$. Let $Y_{T,i}=\log f_{P_{V,T}}(D_i)$ and, for $K>0$, let $Y_{T,i}^{(K)}=Y_{T,i}\bone_{\{|Y_{T,i}|\le K\}}$. The variables in each row are i.i.d. By Chebyshev's and Markov's inequalities and \eqref{eq:radial_log_likelihood_envelope}, for every $\veps>0$,
	\begin{align*}
		&\P\left(
		\left|
		\frac{1}{T}\sum_{i=1}^T
		\{Y_{T,i}-\E[Y_{T,i}]\}
		\right|>\veps
		\right)\nonumber\\
		&\qquad\le
		\frac{4K^2}{T\veps^2}
		+
		\frac{4}{\veps}
		\E\left[G(D)\bone_{\{G(D)>K\}}\right].
	\end{align*}
	Letting first $T\to\infty$ and then $K\to\infty$ proves that $L_T(P_{V,T})-\E[\log f_{P_{V,T}}(D)]\pto0$. This is the standard truncation proof of the weak law for row-wise independent triangular arrays; see also \cite{gut1992weak}. Together with \eqref{eq:radial_expected_likelihood_convergence}, this proves \eqref{eq:radial_benchmark_likelihood}.
	
	We next show that $P_V$ is the unique maximizer of the population log-likelihood $M(P)$.
	Let $\overline{\R}_+=[0,\infty]$ be the compactification induced by $v\mapsto v/(1+v)$, and let $\Theta=\mathscr{P}(\overline{\R}_+)$. Extend the kernel by setting $k_N(x\mid0)=k_N(x\mid\infty)=0$. For each $x>0$, the map $P\mapsto f_P(x)$ is continuous on the compact metrizable space $\Theta$.
	
	The functional $M$ is upper semicontinuous on $\Theta$. Indeed, if
	\begin{equation*}
		B(x)=\max\left\{0,\log\left(\frac{C_N}{x}\right)\right\},
	\end{equation*}
	then $\log f_P(x)\le B(x)$ for every $P\in\Theta$ and $\E[B(D)]<\infty$. If $P_j\stackrel{w}{\to}P$, Fatou's lemma applied to $B(D)-\log f_{P_j}(D)$ yields
	\begin{equation*}
		\limsup_{j\to\infty}M(P_j)\le M(P).
	\end{equation*}

	We now identify the unique maximizer of $M$. For $P\in\Theta$, let $q=P((0,\infty))$. If $q=0$, then $M(P)=-\infty$. If $q>0$, let $\bar{P}$ be the conditional distribution of $P$ on $(0,\infty)$. Since $f_P=qf_{\bar{P}}$,
	\begin{equation*}
		M(P)-M(P_V)
		=
		\log q
		-
		\operatorname{KL}(f_{P_V},f_{\bar{P}})
		\le0.
	\end{equation*}
	Equality requires $q=1$ and $f_{\bar{P}}=f_{P_V}$ almost everywhere. The mixing distribution is identifiable: if $D=VU$, then
	\begin{equation*}
		\varphi_{\log D}(t)
		=
		\varphi_{\log V}(t)
		2^{it}\frac{\Gamma(N/2+it)}{\Gamma(N/2)},
	\end{equation*}
	and the last factor is nonzero for every $t\in\R$. Hence equality of the mixture distributions implies equality of the characteristic functions of $\log V$. It follows that $P_V$ is the unique maximizer of $M$ over $\Theta$.
	
	We next record the uniform separation consequence of this fact. Let $\mathcal{G}$ be an open weak neighborhood of $P_V$ and set $\mathcal{F}=\Theta\setminus\mathcal{G}$. Upper semicontinuity, compactness of $\mathcal{F}$, and uniqueness imply that, for some $\gamma>0$,
	\begin{equation}\label{eq:radial_population_separation}
		\sup_{P\in\mathcal{F}}M(P)
		\le
		M(P_V)-\gamma.
	\end{equation}
	Let $d$ metrize weak convergence on $\Theta$, e.g., we can choose $d = d_{\rm BL}$, the bounded Lipschitz distance. For $P\in\mathcal{F}$, define
	\begin{equation*}
		C_j(P)=\{Q\in\Theta:d(Q,P)\le j^{-1}\},
		\qquad
		s_{P,j}(x)=\sup_{Q\in C_j(P)}f_Q(x).
	\end{equation*}
	Compactness and continuity give $s_{P,j}(x)\downarrow f_P(x)$. For $K>0$, set
	\begin{equation*}
		H_{P,j,K}(x)=\max\{-K,\log s_{P,j}(x)\}.
	\end{equation*}
	Dominated convergence, followed by monotone convergence as $K\to\infty$, gives
	\begin{equation*}
		\lim_{K\to\infty}\lim_{j\to\infty}
		\E[H_{P,j,K}(D)]
		=
		M(P).
	\end{equation*}
	In view of \eqref{eq:radial_population_separation}, for every $P\in\mathcal{F}$ we can choose $j_P$ and $K_P$ such that
	\begin{equation*}
		\E[H_{P,j_P,K_P}(D)]
		\le
		M(P_V)-\frac{3\gamma}{4}.
	\end{equation*}
	The open balls $\{Q:d(Q,P)<j_P^{-1}\}$, $P\in\mathcal{F}$, cover $\mathcal{F}$ and therefore admit a finite subcover. For a measure $Q$ in one of these balls, $\log f_Q\le H_{P,j_P,K_P}$. Applying the weak law to the finitely many integrable envelopes in the subcover yields
	\begin{equation}\label{eq:radial_empirical_separation}
		\P\left(
		\sup_{P\in\mathcal{F}}L_T(P)
		\le
		M(P_V)-\frac{\gamma}{2}
		\right)
		\longrightarrow1.
	\end{equation}

	We next move from $\hat{L}_T$ to $L_T$. Write $X_i=\rr_i-\bmu$ and $\bar{X}_T=T^{-1}\sum_{i=1}^T X_i$. The fourth-moment assumption gives
	\begin{equation*}
		\norm{\bar{X}_T}_2=O_p(T^{-1/2}),
		\qquad
		\norm{\hat{\bSigma}-\bSigma}=O_p(T^{-1/2}),
		\qquad
		\norm{\hat{\bSigma}^{-1}-\bSigma^{-1}}=O_p(T^{-1/2}).
	\end{equation*}
	Expanding $\hat{D}_i-D_i$ and applying the triangle inequality gives
	\begin{align*}
		\frac{1}{T}\sum_{i=1}^T|\hat{D}_i-D_i|
		\le{}&
		\norm{\hat{\bSigma}^{-1}-\bSigma^{-1}}
		\frac{1}{T}\sum_{i=1}^T\norm{X_i}_2^2\\
		&+
		2\norm{\bar{X}_T}_2\norm{\hat{\bSigma}^{-1}}
		\frac{1}{T}\sum_{i=1}^T\norm{X_i}_2
		+
		\norm{\hat{\bSigma}^{-1}}\norm{\bar{X}_T}_2^2.
	\end{align*}
	The sample averages are $O_p(1)$, so
	\begin{equation}\label{eq:estimated_radius_average_rate}
		\frac{1}{T}\sum_{i=1}^T|\hat{D}_i-D_i|
		=
		O_p(T^{-1/2}).
	\end{equation}

	For $P\in\mathcal{P}_T$, differentiation under the integral gives
	\begin{equation*}
		\frac{\d}{\d x}\log f_P(x)
		=
		\frac{N/2-1}{x}
		-
		\frac{1}{2}
		\frac{
			\int v^{-N/2-1}e^{-x/(2v)}\d P(v)
		}{
			\int v^{-N/2}e^{-x/(2v)}\d P(v)
		}.
	\end{equation*}
	For $P,Q\in\mathcal{P}_T$, the first term cancels from the derivative of $\log f_P-\log f_Q$. Each ratio in the second term is a weighted average of $1/v$ and hence belongs to $[1/b_T,1/a_T]$. The mean value theorem and \eqref{eq:estimated_radius_average_rate} therefore imply
	\begin{align}\label{eq:radial_likelihood_contrast}
		\sup_{P,Q\in\mathcal{P}_T}
		\bigg|
		&\{\hat{L}_T(P)-\hat{L}_T(Q)\}
		-
		\{L_T(P)-L_T(Q)\}
		\bigg|
		\nonumber\\
		&\le
		\frac{1}{2a_TT}\sum_{i=1}^T|\hat{D}_i-D_i|
		=
		O_p\left(\frac{1}{\sqrt{T}a_T}\right)
		=o_p(1).
	\end{align}
	Since $\hat{P}_V$ maximizes $\hat{L}_T$ over $\mathcal{P}_T$ and $P_{V,T}\in\mathcal{P}_T$ eventually, \eqref{eq:radial_likelihood_contrast} gives
	\begin{equation}\label{eq:radial_approximate_maximizer}
		L_T(\hat{P}_V)
		\ge
		L_T(P_{V,T})-o_p(1).
	\end{equation}
	Notice that this argument is pathwise and does not require $\hat{D}_1,\ldots,\hat{D}_T$ to be independent.
	
	We are now in position to establish the consistency of $\hat{P}_V$ and $\hat{P}_{\lambda}$. Fix an open weak neighborhood $\mathcal{G}$ of $P_V$. Equations~\eqref{eq:radial_benchmark_likelihood}, \eqref{eq:radial_empirical_separation}, and \eqref{eq:radial_approximate_maximizer} imply
	\begin{equation*}
		\P(\hat{P}_V\notin\mathcal{G})\longrightarrow0.
	\end{equation*}
	Hence, $\hat{P}_V\stackrel{w}{\to}P_V$ in probability on $\Theta$. Since every $\hat{P}_V$ has first moment one, Markov's inequality rules out escape of mass to infinity, so the same convergence holds in the usual weak topology on $[0,\infty)$. Both $\hat{P}_V$ and $P_V$ have first moment one. The characterization of $W_1$ convergence by weak convergence and convergence of first moments, applied pathwise along almost surely convergent subsequences, therefore gives
	\begin{equation*}
		W_1(\hat{P}_V,P_V)\pto0.
	\end{equation*}
	Finally, $\hat{P}_{\lambda}$ and $P_{\lambda}$ are the images of $\hat{P}_V$ and $P_V$ under $v\mapsto\sqrt{v}$. Since
	\begin{equation*}
		|\sqrt{v}-\sqrt{w}|^2\le|v-w|,
		\qquad v,w\ge0,
	\end{equation*}
	the coupling characterization of the Wasserstein metrics yields
	\begin{equation*}
		W_2^2(\hat{P}_{\lambda},P_{\lambda})
		\le
		W_1(\hat{P}_V,P_V)
		\pto0.
	\end{equation*}
	This completes the proof.
\end{proof}

\section{Proofs for Appendix~\ref{app:risk_free}}
\label{app:proofs_risk_free}

This section collects the proofs of the results for portfolios with a
risk-free asset stated in Appendix~\ref{app:risk_free}. The arguments
parallel the corresponding results for the canonical target-return
problem in the main text; we give the details needed to identify the
changes induced by removing the budget constraint.

\subsection{Proof of \Cref{prop:normality_risk_free_constraint}}

\begin{proof}
	The proof follows that of \Cref{prop:normality_risky_constraints}
	by removing the deterministic budget constraint. In
	\Cref{thm:normality_general_risk}, take
	$F_s^{(e)}(\rr)=\rr^\top$ and $b_s^{(e)}=\mu_0$, with no other
	constraints. Since $\mu_0\neq0$, the population optimizer is nonzero.
	The identification and curvature conditions are verified in the same
	way, using the nonsingularity of $K_*^f$.
	The limiting KKT system is obtained from that in the proof of
	\Cref{prop:normality_risky_constraints} by deleting the budget row
	and column and replacing $(H_*,V_*,r_*)$ by $(H_*^f,V_*^f,r_*^f)$.
	Its inverse therefore gives the stated covariance
	$\Sigma^f(\mu_0,\rho_m)$. Consistency and asymptotic normality follow
	directly from \Cref{thm:normality_general_risk}.
\end{proof}

\subsection{Proof of \Cref{prop:risk_free_lowdim}}

\begin{proof}
	The proof follows the argument of \Cref{prop:no_risk_free_lowdim}.
	The only difference is that the budget constraint is absent, so the
	constraint matrix $\bmu^{(1)}$ and its right-hand side
	$\mu_0^{(1)}$ are replaced by $\bmu$ and $\mu_0$, respectively.
	Consequently, the population optimizer and the quantities used in
	the limiting quadratic program become
	\[
	\ww_*^f=\frac{\mu_0\bSigma^{-1}\bmu}{\bmu^\top\bSigma^{-1}\bmu},
	\qquad
	(s_*^f)^2=\frac{\mu_0^2}{\bmu^\top\bSigma^{-1}\bmu},
	\qquad
	b_*^f=\frac{\bSigma^{-1}\bmu}{\bmu^\top\bSigma^{-1}\bmu}.
	\]
	The tangent covariance matrix is $P_\perp^f$, and the multiplier is
	$\eta_s^f=1-c^f\rho_m(\lambda Z)$. Apply
	\Cref{prop:normality_risk_free_constraint} and repeat the same
	Gaussian score calculation with these substitutions. The two
	limiting components remain uncorrelated; their covariances are
	obtained by replacing $s_*^2$, $P_\mu$, $P_\perp$, and $A$ by
	$(s_*^f)^2$, $b_*^f(b_*^f)^\top$, $P_\perp^f$, and $A^f$,
	respectively. This gives the asserted formula.
\end{proof}

\subsection{Proof of \Cref{prop:asymptotic_variance_mv_risk_free}}

\begin{proof}
	The argument is the risk-free specialization of the proof of
	\Cref{prop:asymptotic_variance_mv}. Remove the budget constraint
	and use the estimator
	\[
	\hw_{\rm MV}^f(\mu_0)
	=\frac{\mu_0\hat{\bS}^{-1}\hat{\bmu}}
	{\hat{\bmu}^\top\hat{\bS}^{-1}\hat{\bmu}}.
	\]
	The same joint central limit theorem and differentiation argument
	apply. In the covariance calculation, replace $\ww_*$ and $b_*$
	by $\ww_*^f$ and $b_*^f$ from the preceding proof, replace
	$P_\perp$ by $P_\perp^f$, and replace $\gamma_s^{(e)}$ by
	$\mu_0/(\bmu^\top\bSigma^{-1}\bmu)$.
	Substitution into the covariance expression in
	\Cref{prop:asymptotic_variance_mv} yields exactly
	$\Sigma_{\rm MV}^f(\mu_0)$, completing the proof.
\end{proof}

\subsection{Proof of \Cref{cor:compare_variance_risk_free}}

\begin{proof}
	The first terms in the covariance expressions established in the
	preceding two propositions coincide and equal
	$\Sigma_{\mathrm{common}}^f(\mu_0)$. For the
	criterion-dependent component, note that
	\[
	c^f
	=
	\frac{\sgn(\mu_0)}
	{\|\bSigma^{-1/2}\bmu\|_2}
	\]
	and
	\[
	\eta_s^f
	=
	1-c^f\rho_m(\lambda Z).
	\]
	Substitution into the definition of $A^f$ identifies the scalar
	coefficient of $\Sigma_\perp^f(\mu_0)$ with $I(m;c^f)$. Likewise,
	\[
	J(c^f)
	=
	\frac{\E[\lambda^4]}{\E[\lambda^2]^2}
	+
	(c^f)^2\E[\lambda^2]
	=
	\frac{\E[\lambda^4]}{\E[\lambda^2]^2}
	+
	\frac{\E[\lambda^2]}
	{\bmu^\top\bSigma^{-1}\bmu},
	\]
	which is the scalar multiplying the same matrix in the MVO covariance.
	Hence
	\[
	\Sigma_{\rm SRM}^f(m)
	=
	\Sigma_{\mathrm{common}}^f(\mu_0)
	+
	I(m;c^f)\Sigma_\perp^f(\mu_0),
	\]
	and
	\[
	\Sigma_{\rm MV}^f
	=
	\Sigma_{\mathrm{common}}^f(\mu_0)
	+
	J(c^f)\Sigma_\perp^f(\mu_0).
	\]
	Since $\Sigma_\perp^f(\mu_0)\succeq\bzero$, the stated Loewner-order
	comparison follows.
\end{proof}

\subsection{Proof of \Cref{cor:risk_free_oracle_adaptivity}}

\begin{proof}
	Set
	\[
	m_0=m_{*,u}^{c^f},
	\qquad
	\ww_*^f=\ww_*^f(\mu_0),
	\qquad
	s_*^f=\|\bSigma^{1/2}\ww_*^f\|_2.
	\]
	The proof follows the argument of
	\Cref{thm:feasible_asymptotic_normality}. The empirical-process and
	random-integration steps are unchanged. The only difference is that
	there is no budget constraint, so the local empirical constraint set is
	\[
	\hat{\cC}_{T,h}^f
	=
	\left\{
	h\in\R^N:
	\hat{\E}_T[\rr]^\top h
	=
	-(\ww_*^f)^\top\mathbb{G}_T(\rr)
	\right\}.
	\]
	Its limit is
	\[
	\cC_h^f
	=
	\left\{
	h\in\R^N:
	\bmu^\top h
	=
	-(\ww_*^f)^\top\mathbb{G}(\rr)
	\right\}.
	\]
	The corresponding Hessian $H_*^f$ has null space
	$\operatorname{span}\{\ww_*^f\}$. Since
	$(\ww_*^f)^\top\bmu=\mu_0\neq0$, $H_*^f$ is positive definite on
	the tangent space $\{h:\bmu^\top h=0\}$, and the limiting constrained
	quadratic problem has a unique minimizer.
	
	The same plug-in continuity argument as in
	\Cref{prop:plugin_optimal_measure} gives
	\[
	\hat c^f\pto c^f,
	\qquad
	d_{\rm BL}(\hat m^f,m_{*,u}^{c^f})\pto0.
	\]
	Applying the random-measure integration argument from the proof of
	\Cref{thm:feasible_asymptotic_normality}, together with the fixed-measure
	risk-free limit in \Cref{prop:risk_free_lowdim}, yields
	\[
	\sqrt{T}
	\left(
	\hat{\ww}_{\hat m^f}^f(\mu_0)
	-
	\ww_*^f(\mu_0)
	\right)
	\wto
	\normal\left(
	0,
	\Sigma^f\left(
	\mu_0,
	\rho_{m_{*,u}^{c^f}}
	\right)
	\right).
	\]
	The corresponding consistency statement follows from the same uniform
	objective argument used in the proof of
	\Cref{thm:feasible_asymptotic_normality}.
\end{proof}

\end{document}